\documentclass[
	english
]{scrartcl}

\usepackage[utf8]{inputenc}

\usepackage{amsmath,amsfonts,amsthm,amssymb}
\usepackage{mathtools}
\usepackage[colorlinks,citecolor=blue]{hyperref}
\usepackage{cancel}
\usepackage{relsize}
\usepackage{algorithmic}
\usepackage[capitalise,nameinlink]{cleveref}
\crefname{assumption}{Assumption}{Assumptions}
\Crefname{ALC@unique}{Step}{Steps}

\usepackage{changes}

\usepackage{enumitem}
\setlist[enumerate]{label=(\alph*)}
\usepackage{pdfrender}

\usepackage[figure]{hypcap}
\usepackage{graphicx}
\graphicspath{{./img/}}
\usepackage{subcaption} 

\usepackage{preamble}

\usepackage{marginnote}

\allowdisplaybreaks

\numberwithin{equation}{section}

\newcommand\norm[1]{\left\Vert#1\right\Vert}
\newcommand\nnorm[1]{\Vert#1\Vert}

\newcommand\N{\mathbb{N}}
\newcommand\R{\mathbb{R}}

\newcommand\B{\mathbb{B}}

\newcommand\tto{\rightrightarrows}

\newcommand{\spa}{\operatorname{span}}

\newcommand{\sgn}{\operatorname{sgn}}

\newcommand{\dist}{\operatorname{dist}}
\newcommand{\conv}{\operatorname{conv}}
\newcommand{\cone}{\operatorname{cone}}
\newcommand{\supp}{\operatorname{supp}}

\newcommand{\gph}{\operatorname{gph}}
\newcommand{\epi}{\operatorname{epi}}

\DeclareMathOperator*{\argmin}{\operatorname{argmin}}

\renewcommand{\subseteq}{\subset}

\DeclareMathAlphabet{\mathpzc}{OT1}{pzc}{m}{it}
\newcommand\oo{\mathpzc{o}}

\newtheorem{theorem}{Theorem}[section]
\newtheorem{lemma}[theorem]{Lemma}
\newtheorem{proposition}[theorem]{Proposition}
\newtheorem{assumption}[theorem]{Assumption}

\newtheorem{corollary}[theorem]{Corollary}
\newtheorem{remark}[theorem]{Remark}
\newtheorem{definition}[theorem]{Definition}
\newtheorem{example}[theorem]{Example}

\makeatletter
\long\def\@firstoffiveparen#1#2#3#4#5{\textup{\tagform@{#1}}}
\def\eqref@nolink#1{\textup{\tagform@{\ref*{#1}}}}
\def\eqref@link#1{%
\Hy@safe@activestrue
\expandafter\@setref\csname r@#1\endcsname\@firstoffiveparen{#1}%
\Hy@safe@activesfalse
}
\protected\def\eqref{\@ifstar\eqref@nolink\eqref@link}
\makeatother

\definecolor{mygreen}{rgb}{0.0,0.7,0.0}
\definecolor{mybrown}{rgb}{0.5,0.5,0.0}

\begin{document}

\title{%
	Approximate directional stationarity 
	and associated qualification conditions
	}%
\author{%
	Isabella K\"aming%
	\footnote{%
		Technische Universit\"at Dresden,
		Institute of Numerical Mathematics,
		01062 Dresden,
		Germany,
		\email{isabella.kaeming@tu-dresden.de},
		\orcid{0009-0001-8371-6025}
		}
	\and
	Patrick Mehlitz%
	\footnote{%
		Philipps-Universität Marburg,
		Department of Mathematics and Computer Science,
		35032 Marburg,
		Germany,
		\email{mehlitz@uni-marburg.de},
		\orcid{0000-0002-9355-850X}%
		}
	}

\publishers{}
\maketitle

\begin{abstract}
	Approximate stationarity conditions
	provide necessary optimality conditions without requiring additional assumptions
	by demanding that a perturbed stationarity system possesses solutions as
	the involved perturbations tend to zero. Together with
	associated approximate constraint qualifications,
	which are typically rather mild,
	they raised much interest in the optimization community during the last decade.
	In parallel, directional stationarity conditions became quite popular
	as they sharpen standard stationarity conditions by incorporating data
	associated with underlying critical directions.   
	The purpose of this paper is twofold.
	First, we melt the aforementioned concepts of approximate and directional
	stationarity to formulate and study so-called approximate directional
	stationarity. 
	For the underlying model problem,
	an optimization problem with nonsmooth geometric constraints is chosen,
	which covers diverse practically relevant applications.	
	The role of approximate directional stationarity 
	as a necessary optimality condition is investigated
	in much detail, complementing results from the literature.
	Second, we formulate a qualification condition which, based on an
	approximately directionally stationary point, can be exploited to infer
	its directional stationarity.
	The latter condition depends on one particular sequence verifying
	approximate directional stationarity and merely requires to check a simple
	condition of Mangasarian--Fromovitz type stated in terms
	of the directional tools of limiting variational analysis.
	This contrasts standard approximate constraint
	qualifications that typically demand a certain stable behavior of all
	sequences validating approximate stationarity.
	Throughout, various approaches to verify directional
	stationarity of local minimizers are established, and
	illustrative examples are presented to make the theoretical results
	more accessible.
\end{abstract}


\begin{msc}	
	\mscLink{49J52}, \mscLink{49J53}, \mscLink{90C30}, \mscLink{90C46}
\end{msc}

\section{Introduction}\label{sec:intro}

Approximate stationarity conditions in constrained optimization,
demanding that perturbed stationarity conditions hold along 
a sequence converging to some reference point while all involved perturbations tend to zero,
raise broad interest for mainly two reasons.
On the one hand, they provide necessary optimality conditions for local minimality,
stronger than Fritz--John type conditions,
even in the absence of a qualification condition,
see, e.g., \cite[Proposition~5.4]{KaemingFischerZemkoho2025}
and \cite[Lemma~3.4]{Mehlitz2020b}.
On the other hand, several classes of solution algorithms for constrained optimization problems
can be shown to produce sequences whose accumulation points are approximately stationary,
see, e.g., \cite[Section~5.2]{AndreaniHaeserSchuverdtSilva2012}, 
\cite[Section~5]{AndreaniHaeserSecchinSilva2019}, and \cite[Theorem~5.1]{AndreaniMartinezSvaiter2010}.
Approximate stationarity conditions can be traced back to \cite{KrugerMordukhovich1980},
where they have been derived as a consequence of the so-called extremal principle
of variational analysis,
see \cite[Sections~2.1 and 2.2, Lemma~2.32]{Mordukhovich2006} 
for a detailed explanation,
and they can also be verified via the so-called decoupling approach,
see \cite{BorweinZhu1996,Lassonde2001} for classical references 
and \cite{FabianKrugerMehlitz2024} for a modern treatment.
The aforementioned references are concerned with optimization in abstract spaces
where sharp (i.e., point-based) optimality conditions of adequate strength
typically do not hold, which is why approximate optimality conditions play a crucial role.
In \cite{AndreaniHaeserMartinez2011,AndreaniMartinezSvaiter2010},
approximate stationarity has been rediscovered in form of the so-called
approximate Karush--Kuhn--Tucker conditions for standard nonlinear optimization problems,
and their significant relevance in numerical optimization has been illustrated.
Ever since, approximate stationarity has been generalized to diverse more challenging problem classes,
like problems with 
geometric, see, e.g., \cite{AndreaniHaeserSecchinSilva2019,KaemingMehlitz2025,MovahedianPourahmad2024},
or nonsmooth constraints, see, e.g., \cite{HelouSantosSimoes2020,Mehlitz2020b,Mehlitz2023}.

In order to ensure that a given local minimizer of some constrained optimization problem
is stationary, typically, validity of a qualification condition is required.
In \cite{Gfrerer2013}, it has been shown that postulating validity of a qualification condition
in a given critical direction (and not w.r.t.\ the whole feasible set)
at the point of interest is enough for that purpose,
and that the resulting stationarity conditions can be enriched by directional information,
making them generally more restrictive 
than their non-directional counterpart.
These insights motivated the investigation of directional stationarity and regularity conditions
for diverse problem classes,
like bilevel optimization problems, see, e.g., \cite{BaiYe2022,BaiYeZeng2025},
or problems with geometric constraints, see, e.g., 
\cite{BaiYeZhang2019,BenkoCervinkaHoheisel2019,Gfrerer2014,GfrererKlatte2016,GfrererYeZhou2022,OuyangYeZhang2024}.
Typically, these directional conditions are stated in terms of 
directional limiting normal cones and subdifferentials, 
see \cite{BenkoGfrererOutrata2019} for a comprehensive study and historical references.
Alternatively, their convexified counterparts, which
correspond to directional normal cones and subdifferentials in the sense of Clarke,
may be used, see \cite{Clarke1990} for the origin of their non-directional versions.
Directional stationarity has been shown to serve as a necessary
optimality condition in the presence of the so-called directional metric subregularity
constraint qualification, see \cite{Gfrerer2013} again,
and most of the related results in the literature built upon this insight.
Throughout the paper and as a first main contribution, we will investigate several alternative approaches 
that ensure directional stationarity of a local minimizer in detail 
and compare them using illustrative examples.

In \cite{BenkoMehlitz2024b}, the authors aimed at combining
the two theories outlined above in order to define and study a suitable concept
of approximate directional stationarity for nonsmooth optimization problems.
A main result of \cite{BenkoMehlitz2024b} states that,
given a local minimizer of a constrained optimization problem,
the local minimizer is either stationary or there exists a critical direction such that
it is directionally stationary in this direction,
see \cite[Corollary~4.5]{BenkoMehlitz2024b} and \cref{thm:locmin_dir_akkt}.
In order to develop new directional qualification conditions,
which was the driving force behind \cite{BenkoMehlitz2024b},
this observation has been sufficient.
However, it fails to draw a full picture of the nature of approximate directional stationarity.
Particularly, \cite{BenkoMehlitz2024b} did not clarify whether approximate directional
stationarity conditions hold in \emph{all} critical directions.

In this paper, as a second main contribution, 
we complement the results in \cite{BenkoMehlitz2024b} by proving precisely
the latter claim for so-called implicit critical directions, see \cref{thm:locmin_dir_akkt_II},
where a direction is referred to as implicitly critical whenever it is tangent to the feasible set while being
non-ascending for the objective function.
We also show that this result does not extend to so-called explicit critical directions,
which are non-ascending for the objective function and taken from a suitable linearization cone,
see \cref{ex:AKKT_in_critical_directions}.
All these results are obtained for an optimization problem with a feasible set 
of preimage structure $F^{-1}(\Gamma)$, where $F\colon\R^n\to\R^\ell$ is
directionally differentiable as well as locally Lipschitz continuous while
$\Gamma\subset\R^\ell$ is a closed set.
As mentioned in \cref{rem:relations_to_other_CQs},
this setting covers the type of constraints used in \cite{BenkoMehlitz2024b}.
Throughout, we are working with the (directional) limiting tools of variational analysis.
Let us, however, mention that similar results can be obtained exploiting the (directional)
Clarke tools of variational analysis.

To motivate the third main contribution of this paper,
let us come back to the concept of approximate stationarity for a moment.
As mentioned earlier, the latter serves as a necessary condition for local optimality
without any additional assumptions.
One may now ask what it takes to ensure that a given approximately stationary point
is already stationary in classical sense.
This question leads to the development of so-called approximate (or asymptotic) constraint qualifications,
demanding a somewhat stable behavior of all sequences which can be used to justify
approximate stationarity of the reference point,
the first of which is the cone-continuity property that addresses standard nonlinear problems,
see \cite{AndreaniMartinezRamosSilva2016}.
The latter paper also revealed that the cone-continuity property is a rather mild
constraint qualification.
This observation motivated the extension of approximate constraint qualifications
to optimization problems with
geometric, see, e.g., 
\cite{AndreaniHaeserSecchinSilva2019,JiaKanzowMehlitzWachsmuth2023,KanzowRaharjaSchwartz2021,Ramos2021},
and
nonsmooth constraints,
see, e.g., \cite{Mehlitz2020b,Mehlitz2023},
in the finite-dimensional case,
and even reasonable infinite-dimensional counterparts have been developed,
see, e.g., 
\cite{BoergensKanzowMehlitzWachsmuth2020,KrugerMehlitz2022}.
In \cite[Section~5]{BenkoMehlitz2024b}, 
the authors enriched approximate constraint qualifications
by incorporating directional information.
The obvious conceptual drawback of approximate constraint qualifications is that,
in order to verify their validity, one has to check a technical stability property
for all sequences quantifying approximate stationarity of the reference point, 
and these are, typically, infinitely many.
These observations provided the starting point for the developments published in \cite{KaemingFischerZemkoho2025}.
Therein, the authors introduce the so-called subset Mangasarian--Fromovitz condition (subMFC for short)
for optimization problems with nonsmooth inequality constraints.
The latter condition claims the existence of a sequence quantifying approximate stationarity
while possessing a certain (point-based) property which is related to the Mangasarian--Fromovitz
constraint qualification and, thus, easy to verify.
It has been shown in \cite[Theorem~3.9]{KaemingFischerZemkoho2025} that an approximately stationary
point which satisfies subMFC is already stationary,
and the considerations in \cite[Sections~4 and 5]{KaemingFischerZemkoho2025} underline that
subMFC is a rather weak qualification condition.
In \cite{KaemingMehlitz2025}, subMFC was extended to optimization problems with smooth so-called
orthodisjunctive constraints that cover, for example,
cardinality, complementarity, switching, and vanishing constraints,
see \cite{BenkoCervinkaHoheisel2019,Mehlitz2020a} for details.
In \cref{sec:subMFC} of this paper, we generalize subMFC even further in two ways.
First, it is shown that it can be extended to optimization problems with nonsmooth orthodisjunctive
constraints. 
Second, we enrich it by incorporating directional information, thus,
complementing the results of \cite{BenkoMehlitz2024b} as well.
Our second main result, \cref{thm:ODPsubMFC}, verifies that these extensions of
subMFC serve as (directional) qualification conditions ensuring (directional) stationarity
of local minimizers.

The remainder of this paper is organized as follows.
In \cref{sec:preliminearies}, we summarize the notation used in this paper,
recall some fundamentals from variational analysis,
and comment on some first-order necessary and sufficient optimality conditions
in constrained optimization.
Afterwards, in \cref{sec:dir_QCs}, we are concerned with approximate directional
stationarity conditions and constraint qualifications in nonsmooth geometrically
constrained optimization.
In \cref{sec:dir_stat}, we recall a suitable notion of directional stationarity
and associated constraint qualifications from the literature.
\Cref{sec:approx_dir_stat} introduces an approximate counterpart of directional
stationarity and clarifies its role as a necessary optimality condition.
Thereafter, in \cref{sec:approx_CQs}, we briefly reflect on approximate
directional constraint qualifications from the literature and present some
associated consequences of our earlier findings.
In \cref{sec:subMFC}, we generalize subMFC to problems with nonsmooth orthodisjunctive
constraints, which is a special instance of the model
investigated in \cref{sec:dir_QCs}. After a motivation of this problem class
in \cref{sec:subMFC_setting}, suitable generalized notions of subMFC
are introduced and studied in \cref{sec:subMFC_concept}.
We also comment on relations to other qualification conditions.
Some concluding remarks close the paper in \cref{sec:conclusions}.

\section{Notation and preliminaries}\label{sec:preliminearies}

This section is dedicated to the presentation of comments regarding the notation used in this paper
and some preliminary results from variational analysis.
Furthermore, we review some first-order optimality conditions
for constrained optimization problems.

\subsection{Fundamental notation}

Throughout, let $\N$, $\R$, $\R_+$, and $\R_-$ denote the positive integers
as well as the real, nonnegative real, and nonpositive real numbers, respectively.
Depending on the context, $0$ is used to represent the scalar zero or the all-zero vector.
We denote the standard sign function by $\sgn\colon\R\to\{-1,0,1\}$.
Given positive integers $n,m\in\N$, we canonically identify the Cartesian product
$\R^n\times\R^m$ with $\R^{n+m}$.
Particularly, a tuple $(x_1,\ldots,x_n)$ of real numbers $x_1,\ldots,x_n\in\R$
will be identified with the vector $(x_j)_{j=1}^n\in\R^n$.
For some vector $x\in\R^n$, $\supp(x)\coloneqq\{i\in\{1,\ldots,n\}\,|\,x_i\neq 0\}$
represents the support of $x$. 
Given any nonempty index set $I\subset\{1,\dotsc,n\}$, the vector $x_I\in\R^{|I|}$ is obtained
from $x$ by deleting the components with indices in 
$\{1,\dotsc,n\}\setminus I$, while, for notational convenience, the components of $x_I$ 
are still indexed with the elements of $I$. 
An analogous notation is employed for the restriction of vector-valued mappings.
For a set $C\subset\R^n$, 
$\spa C$, $\cone C$, and $\conv C$ denote the linear, conic, and convex hull of $C$,
respectively.
The closed, convex cone $C^\circ\coloneqq\{y\in\R^n\,|\,x^\top y\leq 0\,\forall x\in C\}$
is called the polar cone of $C$. Moreover, the set $C$ is called polyhedral if it can be
represented as the union of finitely many convex polyhedral sets.

Let $\norm{\cdot}$ be the Euclidean norm in $\R^n$.
We use $\mathbb S\coloneqq\{y\in\R^n\,|\,\norm{y}=1\}$ to denote the unit sphere.
Given $\varepsilon>0$, 
$\mathbb B_\varepsilon(x)\coloneqq\{y\in\R^n\,|\,\norm{y-x}\leq\varepsilon\}$
is the closed ball centered at $x$ with radius $\varepsilon$.
Furthermore, given a direction $d\in\R^n$, $\varepsilon>0$, and $\delta>0$, 
the directional neighborhood in direction $d$ is defined by 
\[
	\mathbb{B}_{\varepsilon,\delta}(d) 
	\coloneqq 
	\left\{ y\in\mathbb B_\varepsilon(0) \,\middle|\, \bigl\Vert\norm{d} y - \norm{y}d\bigr\Vert \leq \delta\norm{y}\norm{d} \right\}.
\]
For sequences $\{x^k\}_{k=1}^\infty\subset\R^n$ and $\{\alpha_k\}_{k=1}^\infty\subset\R_+$,
we write $x^k\in\oo(\alpha_k)$ to express 
the existence of a null sequence $\{\varepsilon_k\}_{k=1}^\infty\subset[0,\infty)$
such that $\nnorm{x^k}\leq\alpha_k\varepsilon_k$ holds for all $k\in\N$.
Let us emphasize that $\{\alpha_k\}_{k=1}^\infty$ may have vanishing elements
in this definition of $\oo$.

Let $\upsilon\colon\R^n\to\R^m$ be a mapping.
Given $x\in\R^n$, recall that $\upsilon$ is called directionally differentiable at $x$
if, for each $d\in\R^n$, the directional derivative
of $\upsilon$ at $x$ in direction $d$
\[
	\upsilon'(x;d)
	\coloneqq
	\lim_{t\downarrow 0}\frac{\upsilon(x+t d)-\upsilon(x)}{t}
\]
is well-defined and componentwise finite. 
Whenever $\upsilon$ is directionally differentiable at
each point from $\R^n$, it is referred to as directionally differentiable.
Let us note that whenever $\upsilon$ is directionally differentiable and locally Lipschitz continuous,
then the limit
\[
	\lim\limits_{t\downarrow 0,\,d'\to d}\frac{\upsilon(x+t d')-\upsilon(x)}{t}
\]
is well-defined, componentwise finite, and equals $\upsilon'(x;d)$ for each $x,d\in\R^n$.
Whenever $\upsilon$ is differentiable,
$\upsilon'\colon\R^n\to\R^{m\times n}$ is used to denote the derivative of $\upsilon$.
For $x\in\R^n$, $\upsilon'(x)\in\R^{m\times n}$ is the Jacobian of $\upsilon$ at $x$.
In the case where $m\coloneqq 1$ holds, $\nabla \upsilon(x)\coloneqq \upsilon'(x)^\top$
represents the gradient of $\upsilon$ at $x$.

For a set-valued mapping $\Upsilon \colon \R^n \tto \R^\ell$, we use 
$\gph \Upsilon \coloneqq \{(x, y) \in \R^{n+\ell} \,|\, y \in \Upsilon(x)\}$
to represent its graph, and
the inverse mapping $\Upsilon^{-1} \colon \R^\ell \tto \R^n$ is defined according to
$\Upsilon^{-1}(y) \coloneqq \{x \in \R^n \,|\, (x, y) \in \gph \Upsilon\}$.
The following concept is taken from \cite[Definition~1]{Gfrerer2013}
and will be frequently used in this paper.
\begin{definition}
	Let $\Upsilon \colon \R^n \tto \R^\ell$ be a set-valued mapping, 
	$(\bar x,\bar y)\in\gph\Upsilon$, and $d\in\R^n$. Then
	$\Upsilon$ is said to be metrically subregular in direction $d$ at $(\bar x,\bar y)$
	if there are constants $\varepsilon>0$, $\delta>0$, and $\kappa>0$ such that
	\[
		\dist(x,\Upsilon^{-1}(\bar y))\leq\kappa\,\dist(\bar y,\Upsilon(x)),
		\qquad
		\forall x\in \{\bar x\} + \mathbb B_{\varepsilon,\delta}(d).
	\]
	If $d\coloneqq 0$ can be chosen,
	$\Upsilon$ is said to be metrically subregular at $(\bar x,\bar y)$.
\end{definition}

A set-valued mapping $\Upsilon\colon\R^n\tto\R^\ell$ is referred to as polyhedral
if its graph can be represented as the union of finitely many convex polyhedral sets.
It is well known from \cite[Proposition~1]{Robinson1981} that each polyhedral
set-valued mapping is metrically subregular at all points of its graph,
see \cite[Theorem~3H.3]{DontchevRockafellar2014} as well.

\subsection{Variational analysis}

For a closed set $\Omega\subset\R^n$ and some point $\bar x\in\Omega$, 
we refer to
\[
	T_\Omega(\bar x)
	\coloneqq
	\left\{d\in\R^n\,\middle|\,
		\begin{aligned}
			&\exists\{d^k\}_{k=1}^\infty\subset\R^n,\,\exists\{t_k\}_{k=1}^\infty\subset\R_+\colon
			\\
			&\quad d^k\to d,\,t_k\downarrow 0,\,\bar x+t_kd^k\in\Omega\,\forall k\in\N
		\end{aligned}
	\right\}
\]
as the tangent (or Bouligand) cone to $\Omega$ at $\bar x$.
Furthermore,
\begin{align*}
	\widehat N_\Omega(\bar x)
	&\coloneqq
	T_\Omega(\bar x)^\circ,
	\\
	N_\Omega(\bar x)
	&\coloneqq
	\left\{\eta\in\R^n\,\middle|\,
		\begin{aligned}
		&\exists\{x^k\}_{k=1}^\infty\subset\Omega,\,\exists\{\eta^k\}_{k=1}^\infty\subset\R^n\colon
		\\
		&\quad x^k\to \bar x,\,\eta^k\to\eta,\,\eta^k\in\widehat N_\Omega(x^k)\,\forall k\in\N
		\end{aligned}
	\right\}
\end{align*}
are referred to as the regular normal cone and the limiting normal cone to $\Omega$ at $\bar x$, respectively.
For the purpose of completeness, 
let us set $T_\Omega(x)\coloneqq\widehat N_\Omega(x)\coloneqq N_\Omega(x)\coloneqq\emptyset$
whenever $x\notin\Omega$.
We note that the limiting normal cone is robust in the sense that,
given $\eta\in\R^n$ as well as sequences $\{x^k\}_{k=1}^\infty\subset\Omega$ and $\{\eta^k\}_{k=1}^\infty\subset\R^n$
such that $x^k\to\bar x$, $\eta^k\to\eta$, and $\eta^k\in N_\Omega(x^k)$ for all $k\in\N$,
we also have $\eta\in N_\Omega(\bar x)$.
Finally, given some direction $d\in\R^n$, the set
\[
	N_\Omega(\bar x;d)
	\coloneqq
	\left\{\eta\in\R^n\,\middle|\,
		\begin{aligned}
			&\exists\{d^k\}_{k=1}^\infty\subset\R^n,\,\exists\{t_k\}_{k=1}^\infty\subset\R_+
			,\,\{\eta^k\}_{k=1}^\infty\subset\R^n\colon
			\\
			&\quad d^k\to d,\,t_k\downarrow 0,\,\eta^k\to\eta,\,\eta^k\in\widehat N_\Omega(\bar x+t_kd^k)\,\forall k\in\N
		\end{aligned}
	\right\}
\]
is the so-called limiting normal cone to $\Omega$ in direction $d$ at $\bar x$.
Note that this set is trivially empty if $d\notin T_\Omega(\bar x)$.
For $x\notin\Omega$, we set $N_\Omega(x;d)\coloneqq\emptyset$ for completeness.
Clearly, we have $N_\Omega(\bar x;0)=N_\Omega(\bar x)$ and $N_\Omega(\bar x;d)\subset N_\Omega(\bar x)$.
It follows, e.g., from \cite[Proposition~2]{GfrererYeZhou2022} that the directional limiting normal cone
is robust in the sense that,
given $\eta\in\R^n$ as well as sequences $\{d^k\}_{k=1}^\infty \subset\R^n$, $\{t_k\}_{k=1}^\infty\subset\R_+$,
and $\{\eta^k\}_{k=1}^\infty\subset\R^n$
such that $d^k\to d$, $t_k\downarrow 0$, $\eta^k\to\eta$, and $\eta^k\in N_\Omega(\bar x+t_kd^k)$ for all $k\in\N$,
we also have $\eta\in N_\Omega(\bar x;d)$.
We note that whenever $\Omega$ is the union of finitely many closed, convex sets, then
\begin{equation}\label{eq:dir_normal_cone_for_convex}
	N_\Omega(\bar x;d)
	\subset 
	N_\Omega(\bar x)\cap\{d\}^\perp,
\end{equation}
and equality holds if $\Omega$ is convex and $d\in T_\Omega(\bar x)$, 
see \cite[Lemma~2.1]{Gfrerer2014}.
Let us also mention that these results are preserved if $\Omega$ enjoys
these properties only locally around $\bar x$ as the (directional) limiting normal cone
only depends on the local structure of $\Omega$ around $\bar x$. 

For a directionally differentiable and locally Lipschitz continuous function $\varphi\colon\R^n\to\R$,
some point $\bar x\in\R^n$, and some direction $d\in\R^n$,
we denote the 
(analytic) limiting subdifferential of $\varphi$ in direction $d$ at $\bar x$ as
\[
	\partial \varphi(\bar{x}; d) 
	\coloneqq
	\{
	\xi\in\R^n\,|\,(\xi,-1)\in N_{\epi \varphi}((\bar x,\varphi(\bar x));(d,\varphi'(\bar x;d)))
	\},
\]
where
$\epi \varphi \coloneqq \{ (x,\alpha) \in\R^{n}\times\R \mid \varphi(x) \leq \alpha \}$ 
is the epigraph of $\varphi$.
We refer to $\partial\varphi(\bar x)\coloneqq\partial\varphi(\bar x;0)$ 
as the limiting subdifferential of $\varphi$ at $\bar x$.
Clearly, we have
\[	
	\partial\varphi(\bar x)
	=
	\{\xi\in\R^n\,|\,(\xi,-1)\in N_{\epi \varphi}((\bar x,\varphi(\bar x)))\}
\]
and $\partial \varphi(\bar{x}; d) \subset \partial\varphi(\bar x)$.

For a directionally differentiable and locally Lipschitz continuous mapping 
$F\colon\R^n\to\R^\ell$, some point $\bar x\in\R^n$, 
and $(d,w)\in T_{\gph F}((\bar x,F(\bar x)))$, the mapping
$D^*F(\bar x;(d,w))\colon\R^\ell\tto\R^n$ defined via
\[
	D^*F(\bar x;(d,w))(\upsilon)
	\coloneqq
	\{\xi\in\R^n\,|\,(\xi,-\upsilon)\in N_{\gph F}((\bar x,F(\bar x));(d,w))\},
	\qquad
	\forall\upsilon\in\R^\ell
\]
is the 
limiting coderivative of $F$ in direction $(d,w)$ at $\bar x$.
Note that we have
\[
	(d,w)\in T_{\gph F}((\bar x,F(\bar x)))
	\quad\Longleftrightarrow\quad
	w=F'(\bar x;d)
\]
due to the directional differentiability of $F$.
Furthermore, \cite[Corollary~4.1, Proposition~5.1]{BenkoGfrererOutrata2019} yield
\begin{equation}\label{eq:coderivative_vs_subdifferential}
	D^*F(\bar x;(d,F'(\bar x;d)))(\lambda)
	=
	\partial\langle\lambda,F\rangle(\bar x;d),
	\qquad
	\forall\lambda\in\R^\ell
\end{equation}
in this situation.
Herein, given $\lambda\in\R^\ell$,
$\langle\lambda,F\rangle\colon\R^n\to\R$ is the scalarization mapping
defined via $\langle\lambda,F\rangle(x)\coloneqq\lambda^\top F(x)$ for all $x\in\R^n$.
Note that choosing $d\coloneqq 0$ in the definition of the directional limiting coderivative
recovers the definition 
of the classical limiting coderivative $D^*F(\bar x)\colon\R^\ell\tto\R^n$.
According to \cite[Theorem~1.32]{Mordukhovich2018}, we have
\begin{equation}\label{eq:scalarization_rule}
	D^*F(\bar x)(\lambda)
	=
	\partial\langle\lambda,F\rangle(\bar x),
	\qquad
	\forall\lambda\in\R^\ell.
\end{equation}
Whenever $F$ is continuously differentiable at $\bar x$, we find
\[
	\partial\langle\lambda,F\rangle(\bar x;d)
	=
	\partial\langle\lambda,F\rangle(\bar x)
	=
	\{F'(\bar x)^\top\lambda\},
	\qquad
	\forall\lambda\in\R^\ell,
\]
see \cite[Remark~2.1]{BenkoGfrererOutrata2019}.

Below, we present some preliminary results for later use.
To start, let us elaborate on a robustness property of the
limiting subdifferential associated with scalarization mappings.

\begin{lemma}\label{lem:robustness_of_subdifferential_scalarization_function}
	Let $F\colon\R^n\to\R^\ell$ be directionally differentiable and locally Lipschitz continuous.
	Fix $x,d,\xi\in\R^n$, $\lambda\in\R^\ell$, and sequences
	$\{x^k\}_{k=1}^\infty\subset\R^n$, $\{\xi^k\}_{k=1}^\infty\subset\R^n$, 
	and $\{\lambda^k\}_{k=1}^\infty\subset\R^\ell$ 
	such that the convergences $x^k\to x$, $\xi^k\to\xi$, and $\lambda^k\to\lambda$
	as well as $\xi^k\in\partial\langle\lambda^k,F\rangle(x^k)$ for all $k\in\N$ hold.
	Then we have $\xi\in\partial\langle\lambda,F\rangle(x)$.
	
	If, additionally, $x^k\neq x$ for all $k\in\N$ and $(x^k-x)/\nnorm{x^k-x}\to d$ hold,
	then we even have $\xi\in\partial\langle\lambda,F\rangle(x;d)$.
\end{lemma}
\begin{proof}
	Noting that $\xi^k\in\partial\langle \lambda^k,F\rangle(x^k)$ is equivalent to
	$(\xi^k,-\lambda^k)\in N_{\gph F}((x^k,F(x^k)))$ for each $k\in\N$ by \eqref{eq:scalarization_rule},
	taking the limit $k\to\infty$ while exploiting robustness of the limiting normal cone
	implies $(\xi,-\lambda)\in N_{\gph F}((x,F(x)))$, which yields $\xi\in\partial\langle\lambda,F\rangle(x)$
	via \eqref{eq:scalarization_rule}.
	
	Now, assume that $x^k\neq x$ for all $k\in\N$ and $(x^k-x)/\nnorm{x^k-x}\to d$ hold.
	Then, for each $k\in\N$, we have
	\[
		(\xi^k,-\lambda^k)\in N_{\gph F}\left((x,F(x))+\nnorm{x^k-x}\left(\frac{x^k-x}{\nnorm{x^k-x}},\frac{F(x^k)-F(x)}{\nnorm{x^k-x}}\right)\right).
	\]
	Robustness of the directional limiting normal cone yields the inclusion
	\[
		(\xi,-\lambda)\in N_{\gph F}((x,F(x));(d,F'(x;d))),
	\]
	and, hence, $\xi\in\partial\langle\lambda,F\rangle(x;d)$ via \eqref{eq:coderivative_vs_subdifferential}.
\end{proof}

The following lemma provides a scaled sum rule for the directional limiting subdifferential.
Let us emphasize that picking the all-zero vector as the underlying direction yields
an analogous result for the (non-directional) limiting subdifferential.

\begin{lemma}\label{lem:sum_rule}
	Let $F\colon\R^n\to\R^\ell$ be directionally differentiable and locally Lipschitz continuous,
	and fix $x,d\in\R^n$ as well as $\lambda\in\R^\ell$.
	Then we have the inclusion
	\[
		\partial\langle\lambda,F\rangle(x;d)
		\subset
		\mathsmaller\sum\nolimits_{i=1}^\ell|\lambda_i|\partial(\sgn(\lambda_i)\,F_i)(x;d),
	\]
	and equality holds whenever all but at most one component function of $F$
	are continuously differentiable at $x$.
\end{lemma}
\begin{proof}
	To prove the general inclusion, 
	we observe that $\langle\lambda,F\rangle=g\circ\Phi$ holds for the functions
	$g\colon\R^{\ell n}\to\R$ and $\Phi\colon\R^n\to\R^{\ell n}$ given by
	\[
		g(x^1,\ldots,x^\ell)\coloneqq \mathsmaller\sum\nolimits_{i=1}^\ell (\lambda_i\,F_i)(x^i),
		\qquad
		\forall x^1,\ldots,x^\ell\in\R^n
	\]
	and
	\[
		\Phi(x)\coloneqq(x,\ldots,x), 
		\qquad
		\forall x\in\R^n,
	\]
	respectively.
	Noting that $\lambda_i\,F_i$ is locally Lipschitz continuous for each $i\in\{1,\ldots,\ell\}$,
	the qualification condition 
	in \cite[formula (23)]{BenkoGfrererOutrata2019} is valid
	due to \cite[Corollary~5.4]{BenkoGfrererOutrata2019},
	and we can apply \cite[Proposition~4.2, Theorem~4.1]{BenkoGfrererOutrata2019} to obtain the inclusion
	\[
		\partial\langle\lambda,F\rangle(x;d)
		\subset
		\mathsmaller\sum\nolimits_{i=1}^\ell\partial(\lambda_i\,F_i)(x;d),
	\]
	where we also exploited \cite[Corollary~4.1]{BenkoGfrererOutrata2019}.
	From its definition one can easily check that the directional limiting subdifferential
	is homogeneous w.r.t.\ positive scalarization of the involved function,
	see, e.g., \cite[formula (5.6)]{LongWangYang2017}.
	Hence, the claimed inclusion is valid.
	
	To conclude the proof, let us assume without loss of generality that at most $F_1$ is not
	continuously differentiable at $x$. 
	Then we have
	\begin{align*}
		|\lambda_i|\,\partial(-\sgn(\lambda_i)\,F_i)(x;d) &= \{-\lambda_i\nabla F_i(x)\},
		\\
		|\lambda_i|\,\partial(\sgn(\lambda_i)\,F_i)(x;d) &= \{\lambda_i\nabla F_i(x)\}
	\end{align*}
	for all $i\in\{2,\ldots,\ell\}$.
	Exploiting the already verified inclusion, we find
	\begin{align*}
		|\lambda_1|\partial(\sgn(\lambda_1)\,F_1)(x;d)
		&=
		\partial(\lambda_1\,F_1)(x;d)
		\\
		&=
		\partial\left(\langle\lambda,F\rangle + \mathsmaller\sum\nolimits_{i=2}^\ell (-\lambda_i)\,F_i\right)(x;d)
		\\
		&\subset
		\partial\langle\lambda,F\rangle(x;d) + \mathsmaller\sum\nolimits_{i=2}^\ell |\lambda_i|\,\partial(-\sgn(\lambda_i)\,F_i)(x;d)
		\\
		&=
		\partial\langle\lambda,F\rangle(x;d) - \left\{\mathsmaller\sum\nolimits_{i=2}^\ell \lambda_i\nabla F_i(x)\right\},
	\end{align*}
	so that a rearrangement yields the second claim of the lemma.
\end{proof}

Formally, the general inclusion in \cref{lem:sum_rule} could also be distilled using the sum rule 
from \cite[Corollary~4.5]{BenkoGfrererOutrata2019}.
Unfortunately, the formula stated therein contains misleading typos, which is why we did not refer to it.
The first part of the proof of \cref{lem:sum_rule}, however, 
is inspired by the one of \cite[Corollary~4.5]{BenkoGfrererOutrata2019}.
Let us also mention that the second assertion of \cref{lem:sum_rule} 
can also be derived via \cite[Theorem~5.6]{LongWangYang2017}.

The next lemma considers the directional limiting normal cone to sets
given as the preimage of a closed set under some potentially nonsmooth mapping
of special structure.

\begin{lemma}\label{lem:normals_to_perturbed_feasible_set}
	Let $F\colon\R^n\to\R^\ell$ be directionally differentiable and locally Lipschitz continuous.
	Furthermore, let $\Gamma\subset\R^\ell$ be closed.
	Consider $Z\coloneqq \{(x,\delta)\in\R^{n+\ell}\,|\,F(x)-\delta\in\Gamma\}$
	and fix $(x,\delta)\in Z$.
	Then we have
	\[
		N_Z((x,\delta))
		\subset
		\{
			(\xi,-\lambda)\in\R^{n+\ell}\,|\,
			\xi\in\partial\langle\lambda,F\rangle(x),\,\lambda\in N_\Gamma(F(x)-\delta)
		\}.
	\]
\end{lemma}
\begin{proof}
	Noting that $Z=\gph F + (\{0\}\times(-\Gamma))$, 
	we can apply the sum rule from \cite[Section~5.1.1]{BenkoMehlitz2022}.
	The required inner semicompactness assumption follows from continuity of $F$.
	Finally, one has to keep \eqref{eq:scalarization_rule} in mind to obtain the desired result.
\end{proof}

\subsection{First-order optimality conditions}

Here, we review some elementary first-order necessary and sufficient
optimality conditions for optimization problems of the form
\begin{equation}\label{eq:elementary_P}
		\min\limits_x \quad f(x)	 
		\quad\textup{s.t.}\quad	
		x\in X,
\end{equation}
where $f\colon\R^n\to\R$ is continuously differentiable and $X\subset\R^n$
is nonempty and closed.

To start, let us inspect primal first-order optimality conditions.
\begin{proposition}\label{prop:primal_fo_conditions}
	Let $\bar x\in\R^n$ be feasible for \eqref{eq:elementary_P}.
	Then the following assertions hold.
	\begin{enumerate}
		\item\label{item:primal_fo_condition_nec}
		 	Let $\bar x$ be a local minimizer of \eqref{eq:elementary_P}.
			Then, for all $d\in T_X(\bar{x})$, we have $f'(\bar{x})d \geq 0$.
		\item\label{item:primal_fo_condition_suf} 
			Assume that $f'(\bar x)d>0$ holds for all $d\in T_X(\bar x)\setminus\{0\}$.
			Then $\bar x$ is a strict local minimizer of \eqref{eq:elementary_P}.
	\end{enumerate}
\end{proposition}
\begin{proof}
	The first assertion is taken from \cite[Theorem~6.12]{RockafellarWets1998},
	while the second one can be distilled from \cite[Proposition~1]{MehlitzZemkoho2021}.
\end{proof}

For the formulation of dual optimality conditions,
we will make use of a so-called critical cone.
Given a feasible point $\bar x\in \R^n$ of \eqref{eq:elementary_P},
the (implicit) critical cone is given by
\begin{equation}\label{eq:implicit_critical_cone}
	\widehat C(\bar x)
	\coloneqq
	\{d\in \R^n\,|\,f'(\bar x)d\leq 0,\,d\in T_X(\bar x)\}.
\end{equation}
Note that, due to \cref{prop:primal_fo_conditions}\,\ref{item:primal_fo_condition_nec},
whenever $\bar x$ is a local minimizer of \eqref{eq:elementary_P}, then
\[
	\widehat C(\bar x)
	=
	\{d\in \R^n\,|\,f'(\bar x)d = 0,\,d\in T_X(\bar x)\}.
\]
Furthermore, \cref{prop:primal_fo_conditions}\,\ref{item:primal_fo_condition_suf}
yields the following corollary.
\begin{corollary}\label{cor:primal_fo_conditions_suf}
	Let $\bar x\in\R^n$ be feasible for \eqref{eq:elementary_P}
	such that $\widehat C(\bar x)=\{0\}$.
	Then $\bar x$ is a strict local minimizer of \eqref{eq:elementary_P}.
\end{corollary}

As announced, we will now review some dual first-order optimality conditions
for \eqref{eq:elementary_P}.

\begin{proposition}\label{prop:dual_fo_conditions}
	Let $\bar x\in\R^n$ be a local minimizer of \eqref{eq:elementary_P}.
	Then the following assertions hold.
	\begin{enumerate}
		\item\label{item:dual_fo_cond_regular}
		 	We have $-\nabla f(\bar x)\in\widehat N_X(\bar x)$.
		\item\label{item:dual_fo_cond_direcioanal} 
			For all $d\in \widehat C(\bar x)\cap\mathbb S$, we have
			$-\nabla f(\bar{x}) \in N_X(\bar{x};d)$.
	\end{enumerate}
\end{proposition}
\begin{proof}
	The first assertion is taken from \cite[Theorem~6.12]{RockafellarWets1998}
	again and, in fact, equivalent to 
	\cref{prop:primal_fo_conditions}\,\ref{item:primal_fo_condition_nec}.
	To justify the second assertion,
	given $d\in\widehat C(\bar x)\cap\mathbb{S}$,
	we first recall that $f'(\bar x)d=0$ holds as $\bar x$ is a local minimizer
	of \eqref{eq:elementary_P}.
	Then the desired result follows from \cite[Proposition~3.2]{OuyangYeZhang2024}.
\end{proof}

Let us explain why merely considering directions $d\in\widehat C(\bar x)\cap \mathbb{S}$
in \cref{prop:dual_fo_conditions}\,\ref{item:dual_fo_cond_direcioanal} is reasonable.
First, $d\in T_X(\bar x)$ is needed because $N_X(\bar x;d)$ would be
empty otherwise.
Second, for any such direction $d$, \cref{prop:primal_fo_conditions}\,\ref{item:primal_fo_condition_nec}
yields $f'(\bar x)d\geq 0$. 
If $f'(\bar x)d>0$, then $f(\bar x+td)>f(\bar x)$ must be true for all sufficiently small $t>0$,
such that direction $d$ is not critical for the local optimality of $\bar x$
as the local variational behavior of $f$ in that direction is already clear.
Hence, it remains to discuss directions from $\widehat C(\bar x)$.
Concerning $d\coloneqq 0$, we recall $N_X(\bar x;0)=N_X(\bar x)$,
and as $-\nabla f(\bar x)\in N_X(\bar x)$ is already covered by the (generally less restrictive) condition
provided in \cref{prop:dual_fo_conditions}\,\ref{item:dual_fo_cond_regular}, this case is not relevant for
\cref{prop:dual_fo_conditions}\,\ref{item:dual_fo_cond_direcioanal}.
Finally, as $\widehat C(\bar x)$ is a cone,
it is enough to inspect directions from $\widehat C(\bar x)\cap\mathbb S$.

\section{Approximate directional stationarity conditions and constraint qualifications in geometrically-constrained optimization}\label{sec:dir_QCs}

We are concerned with the optimization problem
\begin{equation}\label{eq:nonsmooth_problem}\tag{P}
	\min\limits_x \quad f(x)	 
	\quad\textup{s.t.}\quad	
	F(x)\in \Gamma,
\end{equation}
where $f\colon\R^n\to\R$ is continuously differentiable,
$F\colon\R^n\to\R^\ell$ is directionally differentiable and locally Lipschitz continuous,
and $\Gamma\subset\R^\ell$ is a nonempty, closed set.
For later use, we denote the feasible set of \eqref{eq:nonsmooth_problem} 
by $X\coloneqq F^{-1}(\Gamma)\subset\R^n$.
Let us note that assuming continuous differentiability of the objective function 
is not too restrictive in \eqref{eq:nonsmooth_problem}. 
Indeed, if $f$ is assumed to be merely directionally differentiable and locally Lipschitz continuous,
one could consider the surrogate problem
\[
	\begin{aligned}
		&\min\limits_{x,\alpha}&	&\alpha&	 
		\\
		&\text{\,s.t.}&	&(f(x)-\alpha,F(x)) \in \R_-\times \Gamma&
	\end{aligned}
\]
instead, for which it is known that its local and global minimizers correspond to
the local and global minimizers of \eqref{eq:nonsmooth_problem}, 
see e.g.\ \cite{Stein2025} for a recent survey on this so-called epigraph reformulation of \eqref{eq:nonsmooth_problem}.

This section investigates (approximate) directional stationarity conditions
and constraint qualifications for \eqref{eq:nonsmooth_problem}. After
revisiting an appropriate notion of directional stationarity and corresponding
constraint qualifications from the literature in \cref{sec:dir_stat},
we develop an approximate counterpart and examine its role as a necessary optimality
condition in much detail in \cref{sec:approx_dir_stat}. We conclude with a discussion of approximate
constraint qualifications in \cref{sec:approx_CQs}. Altogether, we will develop multiple 
different approaches to verify directional stationarity of local minimizers.

\subsection{Directional stationarity conditions}\label{sec:dir_stat}

Our aim is to study directional optimality conditions and constraint qualifications  
for problem \eqref{eq:nonsmooth_problem}, 
see \cite[Section~3]{BaiYe2022} for an introduction addressing the special case of
inequality-constrained problems.
Besides the implicit critical cone from \eqref{eq:implicit_critical_cone},
we are concerned with the so-called explicit critical cone
of \eqref{eq:nonsmooth_problem}
associated with some feasible point $\bar x\in\R^n$ of that problem and given by
\begin{align*}
	C(\bar x)
	&\coloneqq
	\{d\in\R^n\,|\,f'(\bar x)d\leq 0,\,d\in T^\textup{lin}_{F,\Gamma}(\bar x)\}.
\end{align*}
Above, we made use of the linearization cone
\[
	T^\textup{lin}_{F,\Gamma}(\bar x)
	\coloneqq
	\{d\in\R^n\,|\,F'(\bar x;d)\in T_\Gamma(F(\bar x))\}.
\]
Observe that $T_X(\bar x)\subset T^\textup{lin}_{F,\Gamma}(\bar x)$ is valid.
Indeed, for some $\bar x\in X$, $d\in T_X(\bar x)$ guarantees
the existence of sequences
$\{d^k\}_{k=1}^\infty\subset\R^n$ and $\{t_k\}_{k=1}^\infty\subset\R_+$
such that $d^k\to d$, $t_k\downarrow 0$, and $F(\bar x+t_kd^k)\in \Gamma$ for all $k\in\N$.
Hence,
\[
	F(\bar x)+t_k\frac{F(\bar x+t_kd^k)-F(\bar x)}{t_k}\in \Gamma,
	\qquad
	\forall k\in\N,
\]
and as we have $(F(\bar x+t_kd^k)-F(\bar x))/t_k\to F'(\bar x;d)$ by
directional differentiability and local Lipschitz continuity of $F$, $F'(\bar x;d)\in T_\Gamma(F(\bar x))$
follows by definition of the tangent cone.
Due to $T_X(\bar x)\subset T^\textup{lin}_{F,\Gamma}(\bar x)$, we always have 
$\widehat C(\bar x)\subseteq C(\bar x)$,
where $\widehat C(\bar x)$ is the implicit critical cone from \eqref{eq:implicit_critical_cone}.

Below, we recall the definitions of two prominent stationarity conditions
that address \eqref{eq:nonsmooth_problem}.
\begin{definition}\label{def:dir_MStat}
	Let $\bar x\in\R^n$ be feasible for \eqref{eq:nonsmooth_problem}.
	\begin{enumerate}
		\item We say that $\bar x$ is M-stationary 
			if there is a multiplier $\lambda\in\R^\ell$ solving the system
			\begin{subequations}\label{eq:Mst}
				\begin{align}
					\label{eq:Mst_x}
						-\nabla f(\bar x)&\in \partial\langle\lambda,F\rangle(\bar x),
						\\
					\label{eq:Mst_lambda}
						\lambda&\in N_\Gamma(F(\bar x)).
				\end{align}
			\end{subequations}
		\item\label{item:dir_MStat} 
			For $d\in C(\bar x)$, we say that $\bar x$ is M-stationary in direction $d$
			if there is a multiplier $\lambda\in\R^\ell$ solving the system
			\begin{subequations}\label{eq:dir_Mst}
				\begin{align}
					\label{eq:dir_Mst_x}
						-\nabla f(\bar x)&\in \partial\langle\lambda,F\rangle(\bar x;d),
						\\
					\label{eq:dir_Mst_lambda}
						\lambda&\in N_\Gamma(F(\bar x);F'(\bar x;d)).
				\end{align}
			\end{subequations}
	\end{enumerate}
\end{definition}
Clearly, whenever a feasible point $\bar x\in\R^n$ of \eqref{eq:nonsmooth_problem} 
is M-stationary in direction $d\in C(\bar x)$, 
then (non-directional) M-stationarity of $\bar x$ is inherent.
Noting that the conditions \eqref{eq:dir_Mst} reduce to \eqref{eq:Mst} for $d\coloneqq 0$,
one typically is interested in M-stationarity in directions from $C(\bar x)\cap\mathbb S$.
Let us emphasize that in \cref{def:dir_MStat}\,\ref{item:dir_MStat} we use the 
explicit critical cone $C(\bar x)$,
which is always computable in terms of initial problem data,
instead of the generally smaller
implicit critical cone $\widehat C(\bar x)$,
which might be difficult to access.

Next, we are going to review some constraint qualifications for \eqref{eq:nonsmooth_problem}.

\begin{definition}
	Let $\bar{x}\in\R^n$ be feasible for \eqref{eq:nonsmooth_problem}.
	We say that the metric subregularity constraint qualification
	in direction $d\in\R^n$ (MSCQ(d)) holds at $\bar{x}$ if 
	there exist positive constants $\varepsilon>0$, $\delta>0$, and $\kappa>0$ such that
	\[
		\dist (x, F^{-1}(\Gamma)) \leq \kappa \dist (F(x),\Gamma),
		\qquad
		\forall x\in \{\bar{x}\} + \mathbb B_{\varepsilon,\delta}(d).
	\]
	If $d\coloneqq 0$ can be chosen,
	we say that MSCQ holds at $\bar x$.
\end{definition}

It is well known that MSCQ (MSCQ$(d)$) at a feasible point $\bar x\in\R^n$ of 
\eqref{eq:nonsmooth_problem} is equivalent to the 
metric subregularity of the feasibility mapping $x\mapsto \{F(x)\}-\Gamma$ 
of \eqref{eq:nonsmooth_problem} at $(\bar x,0)$ (in direction $d$). Note that
the definition of (directional) metric subregularity of a set-valued mapping 
can be found in \cite[Definition 1.2]{Gfrerer2013}.
As a consequence of the equivalence mentioned above, we can apply \cite[Lemma~2.7]{Gfrerer2014}
to find that validity of MSCQ at $\bar x$ is equivalent to validity of 
MSCQ$(d)$ at $\bar{x}$ for each nonvanishing direction $d\in\R^n$.

\begin{definition}\label{def:NNAMCQ}
	Let $\bar x\in\R^n$ be feasible for \eqref{eq:nonsmooth_problem}.
	We say that the first-order sufficient condition for metric subregularity
	in direction $d\in\R^n$ (FOSCMS$(d)$) holds at $\bar x$ whenever
	\[
	0\in \partial\langle\lambda,F\rangle(\bar x;d),\,
	\lambda\in N_\Gamma(F(\bar x);F'(\bar x;d))
	\quad\Longrightarrow\quad
	\lambda=0
	\]
	holds. If $d\coloneqq0$ can be chosen, i.e., if
	\[
	0\in \partial\langle\lambda,F\rangle(\bar x),\,
	\lambda\in N_\Gamma(F(\bar x))
	\quad\Longrightarrow\quad
	\lambda=0
	\]
	is valid, the no nonzero abnormal multiplier constraint qualification (NNAMCQ)
	is said to hold at $\bar x$.
\end{definition}

Given a feasible point $\bar x\in\R^n$ of \eqref{eq:nonsmooth_problem},
it is obvious that NNAMCQ, 
which is also referred to as the generalized Mangasarian--Fromovitz
constraint qualification in the literature,
is sufficient for FOSCMS$(d)$ 
for each $d\in \R^n$.
Let us mention that FOSCMS$(d)$ originates from \cite{GfrererKlatte2016}. 
It should be noted that FOSCMS$(d)$ at $\bar x$
is trivially satisfied for all $d\in \R^n\setminus T^\textup{lin}_{F,\Gamma}(\bar x)$
as $N_\Gamma(F(\bar x);F'(\bar x;d))$ is empty in this case.
For $d\in\R^n$,
\cite[Proposition~2.2]{BenkoGfrererOutrata2019} and \eqref{eq:coderivative_vs_subdifferential} show that
FOSCMS$(d)$ at $\bar x$ is sufficient for MSCQ$(d)$ at $\bar x$.

In the subsequently stated lemma, we list some consequences of MSCQ and its directional
version.
\begin{lemma}\label{lem:preimage_rule_dir_MS}
	Let $\bar x\in\R^n$ be feasible for \eqref{eq:nonsmooth_problem}.
	Then the following assertions hold.
	\begin{enumerate}
		\item\label{lem:preimage_rule_dir_MS_nondir} 
			If MSCQ holds at $\bar x$, then
			\[
				T_X(\bar x)
				=
				T^\textup{lin}_{F,\Gamma}(\bar x),
				\qquad
				N_X(\bar x)
				\subset 
				\bigcup\limits_{\lambda\in N_\Gamma(F(\bar x))}
				\partial\langle \lambda,F\rangle(\bar x).
			\]
		\item\label{lem:preimage_rule_dir_MS_dir} 
			Given $d\in T^\textup{lin}_{F,\Gamma}(\bar x)$ such that
			MSCQ$(d)$ holds at $\bar x$, we have $d\in T_X(\bar x)$
			and 
			\[
				N_X(\bar x;d)
				\subset
				\bigcup\limits_{\lambda\in N_\Gamma(F(\bar x);F'(\bar x;d))}
				\partial\langle \lambda,F\rangle(\bar x;d).
			\]
	\end{enumerate}
\end{lemma}
\begin{proof}
	In the first assertion, equivalence of the tangent cone and the linearization cone
	is taken from \cite[Proposition~1]{HenrionOutrata2005}, where the required 
	calmness assumption follows from MSCQ at $\bar x$ as described in 
	\cite[p. 438]{HenrionOutrata2005}.
	The upper estimate for the limiting normal cone
	follows, for example, from \cite[Theorem~3.1]{BenkoGfrererOutrata2019} with $h\coloneqq 0$
	when taking \eqref{eq:coderivative_vs_subdifferential} into account.
	
	Let us take a look at the second assertion.
	We can follow the proof of \cite[Proposition~1]{HenrionOutrata2005},
	wherein $x_l \in \{\bar{x}\} + \mathbb B_{\varepsilon,\delta}(d) $
	can be guaranteed for large enough $l \in\N$, to obtain $d\in T_X(\bar x)$
	via MSCQ$(d)$ at $\bar x$.
	The inclusion for the directional limiting normal cone 
	is taken from \cite[Theorem~3.1]{BenkoGfrererOutrata2019}, 
	see \eqref{eq:coderivative_vs_subdifferential} again.
\end{proof}

As a corollary of the above lemma, we obtain the following well-known optimality conditions.
\begin{corollary}\label{cor:MStat_via_MS}
	Let $\bar x\in\R^n$ be a local minimizer of \eqref{eq:nonsmooth_problem}.
	Then the following assertions hold.
	\begin{enumerate}
		\item If MSCQ holds at $\bar x$, then $\bar x$ is M-stationary.
		\item \label{cor:MStat_via_MS2} Fix $d\in C(\bar x)\cap\mathbb S$.
			If MSCQ$(d)$ holds at $\bar x$,
			then $\bar x$ is M-stationary in direction $d$.
		\item If MSCQ holds at $\bar x$, then $\bar x$ is M-stationary in direction $d$
			for all $d\in C(\bar{x})\cap \mathbb S$.
	\end{enumerate}
\end{corollary}
\begin{proof}
	The first assertion follows from 
	\cref{prop:dual_fo_conditions}\,\ref{item:dual_fo_cond_regular}
	and \cref{lem:preimage_rule_dir_MS}\,\ref{lem:preimage_rule_dir_MS_nondir}.
	In order to verify the second assertion,
	we first apply \cref{lem:preimage_rule_dir_MS}\,\ref{lem:preimage_rule_dir_MS_dir} to obtain
	$d\in\widehat C(\bar x)\cap\mathbb S$.
	Now, \cref{prop:dual_fo_conditions}\,\ref{item:dual_fo_cond_direcioanal} 
	can be used to find
	$-\nabla f(\bar x)\in N_X(\bar x;d)$.
	Applying \cref{lem:preimage_rule_dir_MS}\,\ref{lem:preimage_rule_dir_MS_dir} 
	once again yields the claim.
	Finally, for the third assertion, we use that MSCQ at $\bar x$ implies
	MSCQ$(d)$ at $\bar x$ for any direction $d\in\R^n$ by definition.
	Hence, the claim immediately follows from statement~\ref{cor:MStat_via_MS2}.
\end{proof}

Note that results similar to \cref{cor:MStat_via_MS}\,\ref{cor:MStat_via_MS2} 
have been shown, e.g., in 
\cite[Theorem~3.1]{BaiYe2022}
and
\cite[Theorem~7]{Gfrerer2013}.

Subsequently, we present a slightly different way to ensure
directional M-stationarity at local minimizers of \eqref{eq:nonsmooth_problem},
which avoids assuming validity of MSCQ (in critical directions)
and instead relies on a generalized version of Guignard's constraint qualification
tailored to \eqref{eq:nonsmooth_problem}.

\begin{definition}\label{def:GACQ}
	Let $\bar x\in\R^n$ be feasible for \eqref{eq:nonsmooth_problem}.
	\begin{enumerate}
		\item We say that the generalized Abadie constraint qualification (GACQ)
			holds at $\bar x$ if
			\[
				T_X(\bar x)=T^\textup{lin}_{F,\Gamma}(\bar x).
			\]
		\item We say that the generalized Guignard constraint qualification (GGCQ)
			holds at $\bar x$ if
			\[
				\widehat N_X(\bar x)=T^\textup{lin}_{F,\Gamma}(\bar x)^\circ.
			\]				
	\end{enumerate}
\end{definition}

Given any feasible point of \eqref{eq:nonsmooth_problem}, 
validity of GACQ at this point implies validity of GGCQ there.
Moreover, it is apparent from 
\cref{lem:preimage_rule_dir_MS}\,\ref{lem:preimage_rule_dir_MS_nondir} 
that both concepts are implied by MSCQ.

With the aid of GGCQ (and GACQ), non-directional M-stationarity can be ensured
as outlined in the following remark.

\begin{remark}\label{rem:GGCQ_CQ_under_subregularity}
Let $\bar x\in\R^n$ be a local minimizer of \eqref{eq:nonsmooth_problem}
such that $F$ is continuously differentiable around $\bar x$.
If GGCQ holds at $\bar x$ while 
the linearized feasibility mapping 
\begin{equation}\label{eq:linearized_feasibility_mapping}
	u\mapsto \{F'(\bar x)u\}-T_\Gamma(F(\bar x))
\end{equation}
is metrically subregular at $(0,0)$, 
then $\bar x$ is M-stationary, which follows from \cite[Proposition~3]{BenkoGfrerer2017}
in combination with the necessary optimality condition
$-\nabla f(\bar{x}) \in \widehat N_X(\bar{x})$ from 
\cref{prop:dual_fo_conditions}\,\ref{item:dual_fo_cond_regular}.
Both GGCQ at $\bar x$ and the metric subregularity assumption hold under MSCQ at $\bar x$, 
see \cref{lem:preimage_rule_dir_MS}\,\ref{lem:preimage_rule_dir_MS_nondir}
and \cite[Lemma~4]{Gfrerer2019}. Alternatively, the metric subregularity assumption 
is inherently satisfied if $T_\Gamma(F(\bar x))$
is a polyhedral cone (which is always true if $\Gamma$ is polyhedral), as this 
implies the polyhedrality of the above mapping.
\end{remark}

A result related to \cref{rem:GGCQ_CQ_under_subregularity}, which now concerns directional
M-stationarity, is stated below.

\begin{lemma}\label{lem:dir_MSt_under_GACQ}
	Let $\bar x\in\R^n$ be a local minimizer of \eqref{eq:nonsmooth_problem}
	such that $F$ is continuously differentiable around $\bar x$,
	let GGCQ hold at $\bar x$,
	and
	fix a direction $d\in C(\bar x)\cap\mathbb S$.
	Finally, let one of the following conditions hold.
	\begin{enumerate}
	\item\label{item:loc_poly} 
		The set $\Gamma$ is polyhedral locally around $F(\bar x)$.
	\item\label{item:loc_cvx} 
		The set $\Gamma$ is convex locally around $F(\bar x)$,
		and the mapping
		from \eqref{eq:linearized_feasibility_mapping}
		is metrically subregular in direction $d$ at $(0,0)$.
	\end{enumerate}
	Then $\bar x$ is M-stationary in direction $d$.
\end{lemma}

\begin{proof}
	First, consider the additional assumptions in~\ref{item:loc_poly}.
	Then the assertion follows from \cite[Theorem~3.9]{Gfrerer2014}
	while noting that polyhedrality of $\Gamma$ locally around
	$F(\bar x)$ is enough to run the stated argumentation.
	
	Next, assume that we are given the additional assumptions in~\ref{item:loc_cvx}.
	As $-\nabla f(\bar{x}) \in \widehat{N}_X(\bar x)$ always provides a necessary
	optimality condition, see \cref{prop:dual_fo_conditions}\,\ref{item:dual_fo_cond_regular},
	validity of GGCQ at $\bar x$ guarantees that $0$ is a minimizer of
	\[
		\min\limits_u \quad f'(\bar x)u	 
		\quad\textup{s.t.}\quad	
		F'(\bar x)u\in T_\Gamma(F(\bar x)).
	\]
	As the preimage $F'(\bar x)^{-1}T_\Gamma(F(\bar x))$ is a closed cone,
	we obtain
	\[
		T_{F'(\bar x)^{-1}T_\Gamma(F(\bar x))}(0)
		=
		F'(\bar x)^{-1}T_\Gamma(F(\bar x))
		=
		T^\textup{lin}_{F,\Gamma}(\bar x),
	\]
	such that $d$ is an associated (implicit) critical direction for the above problem at $0$.
	Hence, \cref{prop:dual_fo_conditions}\,\ref{item:dual_fo_cond_direcioanal} yields
	$-\nabla f(\bar x)\in N_{F'(\bar x)^{-1}T_\Gamma(F(\bar x))}(0;d)$.
	The imposed directional metric subregularity assumption corresponds to MSCQ($d$) at 0 for the
	above problem, which allows us to
	utilize \cref{lem:preimage_rule_dir_MS}\,\ref{lem:preimage_rule_dir_MS_dir} 
	to obtain
	$-\nabla f(\bar x)\in F'(\bar x)^\top N_{T_\Gamma(F(\bar x))}(0;F'(\bar x)d)$.
	As we have $d\in C(\bar{x})$ while $T_\Gamma(F(\bar x))$ is convex and 
	$\Gamma$ is locally convex around $\bar x$, we can apply the equality version of
	\eqref{eq:dir_normal_cone_for_convex} twice to find
	\begin{align*}
		N_{T_\Gamma(F(\bar x))}(0;F'(\bar x)d)
		&=
		N_{T_\Gamma(F(\bar x))}(0)\cap\{F'(\bar x)d\}^\perp
		\\
		&=
		N_\Gamma(F(\bar x))\cap\{F'(\bar x)d\}^\perp
		=
		N_\Gamma(F(\bar x);F'(\bar x)d).
	\end{align*}
	Hence, the proof is complete.
\end{proof}

As a corollary,
we obtain the following result, which complements
\cite[Corollary~1]{BenkoGfrerer2017} and \cite[Theorem~7]{FlegelKanzowOutrata2007}
and follows from \cite[Theorem~3.9]{Gfrerer2014}.

\begin{corollary}\label{cor:dir_MSt_disjunctive}
	Let $\bar x\in\R^n$ be a local minimizer of \eqref{eq:nonsmooth_problem}
	such that $F$ is continuously differentiable around $\bar x$,
	let GGCQ hold at $\bar x$,
	and
	let $\Gamma$ be polyhedral locally around $F(\bar x)$.
	Then $\bar x$ is M-stationary in direction $d$ for all
	$d\in C(\bar x)\cap\mathbb S$.
\end{corollary}

The subsequently stated example illustrates that even GACQ on its own
does not serve as a constraint qualification for \eqref{eq:nonsmooth_problem} in general.

\begin{example}\label{ex:SOC_example}
	Consider the linear second-order cone problem
	\[
		\min\limits_x\quad x_1\quad\textup{s.t.}\quad 
		F(x)\coloneqq(x_1,x_2,x_2)\in \Gamma \coloneqq \{y\in\R^3\,|\,(y_1^2+y_2^2)^{1/2}\leq y_3\}.
	\]
	Its feasible set is $X= \{ 0 \} \times \R_+$, and
	we consider the global minimizer $\bar x\coloneqq(0,0)$.
	As $F$ is linear while $\Gamma$ is a cone,
	we find
	\[
		T_X(\bar x) 
		= 
		T^\textup{lin}_{F,\Gamma}(\bar x) 
		=
		X, 
	\]
	and, hence, GACQ is valid at $\bar x$.
	Furthermore, we note
	$
		\widehat C(\bar x)
		=
		C(\bar x)
		=
		X.
	$
	
	Observing that
	\begin{align*}
		F'(\bar x)^\top N_\Gamma(F(\bar x))
		=
		F'(\bar x)^\top\Gamma^\circ
		&=
		\{
			(\lambda_1,\lambda_2+\lambda_3)\in\R^2\,|\,
			(\lambda_1^2+\lambda_2^2)^{1/2}\leq -\lambda_3
		\}
	\end{align*}
	is valid, M-stationarity of $\bar x$ requires $(-1,0)\in F'(\bar x)^\top N_\Gamma(F(\bar x))$, i.e.,
	the existence of $\lambda\in\R^3$ such that
	\[
		\lambda_1=-1,
		\quad
		\lambda_2=-\lambda_3,
		\quad
		(\lambda_1^2+\lambda_2^2)^{1/2}\leq-\lambda_3.
	\]
	The third condition particularly requires $\lambda_2 = -\lambda_3 \geq 0$, 
	such that inserting the first two conditions into the third one and
	taking squares afterwards leads to 
	$1+\lambda_2^2\leq\lambda_2^2$, which cannot hold for any $\lambda_2 \geq 0$.
	Hence, $\bar x$ is not M-stationary and, thus,
	also not M-stationary in the uniquely determined direction $d\coloneqq(0,1)$
	from $\widehat C(\bar x)\cap\mathbb S = C(\bar x)\cap\mathbb S$.
	
	Particularly, this example shows that GACQ is, in general, 
	not a constraint qualification that ensures (directional) M-stationarity of local minimizers
	in the absence of additional assumptions.
	In fact, $\Gamma$ is not polyhedral locally around $F(\bar x)$, 
	and, according to \cref{lem:dir_MSt_under_GACQ},
	the linearized feasibility mapping \eqref{eq:linearized_feasibility_mapping}
	cannot be metrically subregular at $(0,0)$ (in direction $d$).
\end{example}

As mentioned earlier, GGCQ is implied by MSCQ, and the following example shows that this implication is strict.
Let us note that, according to \cite[Example~1]{HenrionOutrata2005}, 
even GACQ does not imply MSCQ, so the above is actually well known.
However, our example especially illustrates that GGCQ is of particular use to infer
M-stationarity of a local minimizer in a direction $d\in C(\bar x)\cap\mathbb{S}$ in situations where 
even MSCQ($d$) is violated.

\begin{example}\label{ex:GGCQ_without_MSCQ}
	Consider the optimization problem
	\[
		\min\limits_x\quad x_1\quad\textup{s.t.}\quad 
		F(x)\coloneqq (-x_1,-x_2,x_1x_2)\in \Gamma \coloneqq \R_- \times \R_- \times \{ 0 \}
	\]
	with the feasible set $X= (\R_+ \times \{ 0 \}) \cup (\{ 0 \} \times \R_+)$ and global 
	minimizer $\bar x \coloneqq (0,0)$. We obtain $T_X(\bar x) = X$,
	$T^\textup{lin}_{F,\Gamma}(\bar x) = \R_+ \times \R_+$, and
	\[
		\widehat N_X(\bar x) = \R_- \times \R_- = T^\textup{lin}_{F,\Gamma}(\bar x)^\circ,
	\]
	which shows that GGCQ is fulfilled at $\bar x$. Moreover, we find 
	$\widehat C(\bar x) = C(\bar x) = \{ 0 \} \times \R_+$ and, thus, 
	may consider direction
	$d\coloneqq (0,1) \in C(\bar x) \cap \mathbb{S}$. 
	
	To verify M-stationarity in direction $d$ at $\bar x$, we need to prove 
	\begin{align*}
		(-1,0) 
		\in 
		F'(\bar{x})^\top N_\Gamma(F(\bar x);F'(\bar x)d) 
		&= 
		F'(\bar{x})^\top N_\Gamma(F(\bar x); (0,-1,0))
		\\
		&=
		\{ (-\lambda_1,-\lambda_2)\in \R^2 \mid \lambda_1\in\R_+, \lambda_2 = 0, \lambda_3\in\R\},
	\end{align*}
	which clearly holds using the multiplier $\lambda \coloneqq (1,0,0)$.
	
	Concerning MSCQ$(d)$ at $\bar x$, notice that, for any fixed $\varepsilon>0$ and $\delta>0$, 
	the point
	$x^k \coloneqq (1/k^2,1/k)$ fulfills $x^k\in \{ \bar x\} + \mathbb B_{\varepsilon,\delta}(d)$ 
	for $k\in\N$ sufficiently large. Moreover, we calculate
	\[
		\dist(x^k,F^{-1}(\Gamma)) = \dist(x^k,X) = \frac{1}{k^2}, 
		\quad
		\dist (F(x^k), \Gamma) = x_1^k x_2^k = \frac{1}{k^3},
	\]
	such that MSCQ$(d)$ at $\bar x$ requires 
	\[
		\frac{1}{k^2} \leq \kappa \frac{1}{k^3}
	\]
	for all $k\in\N$ sufficiently large, which can clearly not be fulfilled for any fixed $\kappa>0$.
	Thus, MSCQ$(d)$ does not hold at $\bar x$, such that MSCQ at $\bar x$ is neither fulfilled.
	Hence, noting the continuous differentiability of $F$ and the polyhedrality of $\Gamma$, 
	(directional) M-stationarity of $\bar x$ as verified above can only be inferred using 
	\cref{cor:dir_MSt_disjunctive}, but not \cref{cor:MStat_via_MS}.
\end{example}

\subsection{Approximate directional stationarity conditions}\label{sec:approx_dir_stat}

Several approximate variants of stationarity conditions are studied in the literature, as they
may be used as necessary optimality conditions that do not require a qualification condition to hold,
see, for example, \cite{KaemingMehlitz2025,Mehlitz2020b,Mehlitz2023,MovahedianPourahmad2024}.
Thus, let us now introduce an approximate variant of M-stationarity for \eqref{eq:nonsmooth_problem}.

\begin{definition}\label{def:approx_stat_sequence_non_dir}
	Let $\bar x\in\R^n$ be feasible for \eqref{eq:nonsmooth_problem}.
	A sequence $\{(x^k,\lambda^k,\delta^k,\varepsilon^k)\}_{k=1}^\infty\subset\R^{n+\ell+\ell+n}$
	satisfying
	\begin{equation}\label{eq:approximate_statonarity}
		\varepsilon^k-\nabla f(x^k)
		\in  \partial\langle\lambda^k,F\rangle(x^k),
		\qquad
		\lambda^k\in N_\Gamma(F(x^k)-\delta^k)
	\end{equation}
	for each $k\in\N$ as well as the convergences
	\begin{equation}\label{eq:sequences_conv}
		(x^k, \delta^k, \varepsilon^k) \to (\bar x, 0, 0) 
	\end{equation}
	is called an approximately M-stationary (AM-stationary) sequence w.r.t.\ $\bar x$.
	If such a sequence exists, 
	$\bar x$ is referred to as AM-stationary.
\end{definition}

Note that the above definition could likewise be stated using the regular 
counterparts of the limiting normal cone and the limiting subdifferential. However,
to enable an easier comparison with available results 
in the literature, such as \cite{BenkoMehlitz2024b},
and to avoid additional technicalities in the following, 
we only consider the limiting normal cone and the limiting subdifferential in 
the course of this paper.

Our interest in AM-stationarity is motivated by the following result,
which follows, e.g., from \cite[Theorem~8]{MovahedianPourahmad2024}.
\begin{lemma}\label{lem:approx_MSt_NC}
	Let $\bar x\in\R^n$ be a local minimizer of \eqref{eq:nonsmooth_problem}.
	Then $\bar x$ is AM-stationary.
\end{lemma}

Using the subsequent result, a connection between AM-stationarity
and M-stationarity can be established.

\begin{lemma}\label{lem:taking_the_limit_in_approximate_conditions}
	Let $\bar x\in\R^n$ be an AM-stationary point of \eqref{eq:nonsmooth_problem}.
	Let $\{(x^k,\lambda^k,\delta^k,\varepsilon^k)\}_{k=1}^\infty\subset\R^{n+\ell+\ell+n}$
	be an AM-stationary sequence w.r.t.\ $\bar x$.
	If $\{\lambda^k\}_{k=1}^\infty$ is bounded, then $\bar x$ is M-stationary.
\end{lemma}
\begin{proof}
	As $\{\lambda^k\}_{k=1}^\infty$ is bounded,
	we may pick a subsequence (without relabeling) such that $\lambda^k\to\bar\lambda$
	holds for some $\bar\lambda\in\R^\ell$.
	Then, due to \eqref{eq:sequences_conv}, \cref{lem:robustness_of_subdifferential_scalarization_function}, 
	and the robustness of the limiting normal cone, we may take the 
	limit $k\to\infty$ in \eqref{eq:approximate_statonarity}
	in order to find that $\bar x$ is M-stationary with multiplier $\bar\lambda$. 
\end{proof}

Owing to \cref{lem:approx_MSt_NC} and the fact that M-stationarity follows for local minimizers
of \eqref{eq:nonsmooth_problem} only in the presence of a qualification condition,
we observe that, for an AM-stationary sequence $\{(x^k,\lambda^k,\delta^k,\varepsilon^k)\}_{k=1}^\infty\subset\R^{n+\ell+\ell+n}$ 
w.r.t.\ a given point, the sequence $\{\lambda^k\}_{k=1}^\infty$ 
is not bounded in general.

Motivated by the recent study \cite{BenkoMehlitz2024b},
we now introduce a \emph{directional} version of
AM-stationarity.

\begin{definition}\label{def:approx_stat_sequence_dir}
	Let $\bar x\in\R^n$ be feasible for \eqref{eq:nonsmooth_problem},
	and let $d\in T^\textup{lin}_{F,\Gamma}(\bar x)\cap\mathbb S$
	be chosen arbitrarily.
	A sequence $\{(x^k,\lambda^k,\delta^k,\varepsilon^k)\}_{k=1}^\infty\subset\R^{n+\ell+\ell+n}$
	that is AM-stationary w.r.t.\ $\bar{x}$ and additionally fulfills
	$x^k\neq\bar x$ for all $k\in\N$ and 
	\begin{subequations}\label{eq:approx_stat_dir}
		\begin{align}
			\label{eq:sequences_dir_conv}
			&
			\frac{x^k-\bar x}{\nnorm{x^k-\bar x}}\to d,
			\quad 
			\frac{\delta^k}{\nnorm{x^k-\bar x}}\to 0,
			\\
			\label{eq:perturbation_vs_multiplier}
			&
			\nnorm{\delta^k}\lambda^k-\nnorm{\lambda^k}\delta^k\in\oo(\nnorm{\delta^k}\nnorm{\lambda^k}),
			\\
			\label{eq:multipliers_not_arbitrarily_unbounded}
			&
			\left\{
				\frac{\nnorm{\delta^k}\nnorm{\lambda^k}}{\nnorm{x^k-\bar x}}
			\right\}_{k=1}^\infty\text{ bounded}
		\end{align}
	\end{subequations}
	is called an AM-stationary sequence w.r.t.\ $\bar x$ in direction $d$.
	If such a sequence exists,
	$\bar x$ is referred to as AM-stationary in direction $d$.
\end{definition}

Observe that \eqref{eq:approximate_statonarity} implicitly requires 
$F(x^k)-\delta^k\in\Gamma$, i.e.,
\begin{equation}\label{eq:suitable_repr_of_constraint_viol}	
	F(\bar x) 
	+ 
	\nnorm{x^k-\bar x}
		\left(
			\frac{F(x^k)-F(\bar x)}{\nnorm{x^k-\bar x}}
			-
			\frac{\delta^k}{\nnorm{x^k-\bar x}}
		\right)
	=
	F(x^k)-\delta^k
	\in 
	\Gamma
\end{equation}
for each $k\in\N$. Thus, the directional convergences from
\eqref{eq:sequences_dir_conv} as well as the definition of the tangent cone
automatically yield $F'(\bar x;d)\in T_\Gamma(F(\bar x))$, i.e., 
$d\in T^\textup{lin}_{F,\Gamma}(\bar x)$, which
explains the choice of the direction in \cref{def:approx_stat_sequence_dir}.
Properties \eqref{eq:perturbation_vs_multiplier} and \eqref{eq:multipliers_not_arbitrarily_unbounded} 
provide some detailed
information about the behavior of the involved perturbations $\{\delta^k\}_{k=1}^\infty\subset\R^\ell$
and multipliers $\{\lambda^k\}_{k=1}^\infty\subset\R^\ell$.
For example, in the nontrivial case where $\delta^k$ and $\lambda^k$ do not vanish for all $k\in\N$,
condition \eqref{eq:perturbation_vs_multiplier} translates into
\begin{equation}\label{eq:residual_va_multiplier_nontrivial}
	\frac{\delta^k}{\nnorm{\delta^k}}-\frac{\lambda^k}{\nnorm{\lambda^k}} \to 0.
\end{equation}
Furthermore, \eqref{eq:multipliers_not_arbitrarily_unbounded}
implies that $\{\nnorm{\delta^k}\nnorm{\lambda^k}\}_{k=1}^\infty$ 
(and, thus, $\{(\lambda^k)^\top\delta^k\}_{k=1}^\infty$) is a null sequence
that converges at least as fast to zero as $\{\nnorm{x^k-\bar x}\}_{k=1}^\infty$.
Particularly, $\{\lambda^k\}_{k=1}^\infty$ may diverge 
(as $\{\delta^k\}_{k=1}^\infty$ is a null sequence),
but the speed of divergence is not arbitrary.

With the following result, we are able to refine \cref{lem:taking_the_limit_in_approximate_conditions} by
incorporating directional information.

\begin{lemma}\label{lem:taking_the_limit_in_directional_approximate_conditions}
	Let $\bar x\in\R^n$ be feasible for \eqref{eq:nonsmooth_problem},
	and let $d\in C(\bar x)\cap\mathbb S$
	be chosen arbitrarily.
	Let $\{(x^k,\lambda^k,\delta^k,\varepsilon^k)\}_{k=1}^\infty\subset\R^{n+\ell+\ell+n}$
	be an AM-stationary sequence w.r.t.\ $\bar x$ in direction $d$.
	If $\{\lambda^k\}_{k=1}^\infty$ is bounded, then $\bar x$ is M-stationary in direction $d$.
\end{lemma}
\begin{proof}
	Due to boundedness of $\{\lambda^k\}_{k=1}^\infty$, we can take a subsequence (without relabeling)
	such that $\lambda^k \to \bar\lambda$ holds for some $\bar{\lambda} \in\R^\ell$.
	With \eqref{eq:sequences_conv}, \eqref{eq:suitable_repr_of_constraint_viol},
	\cref{lem:robustness_of_subdifferential_scalarization_function}, and the 
	robustness of the directional limiting normal cone, taking the limit 
	$k\to\infty$ in \eqref{eq:approximate_statonarity} while respecting the
	directional convergences \eqref{eq:sequences_dir_conv} yields
	that $\bar{x}$ is M-stationary in direction $d$ with multiplier $\bar \lambda$.
\end{proof}

Let us note that,
in the setting of \cref{lem:taking_the_limit_in_directional_approximate_conditions},
boundedness of $\{\lambda^k\}_{k=1}^\infty$ is inherent whenever FOSCMS$(d)$ is valid at $\bar x$.
Indeed, this can easily be shown via a contradiction argument 
and \cref{lem:robustness_of_subdifferential_scalarization_function}.
In contrast, one should observe that validity of MSCQ$(d)$ at $\bar x$ is in general
not enough to guarantee boundedness of $\{\lambda^k\}_{k=1}^\infty$, even if $\bar{x}$ 
is a local minimizer.
\begin{example}\label{ex:mscq_not_enough_for_boundedness}
	Let us consider the optimization problem
	\[
		\min\limits_x\quad x_1\quad\textup{s.t.}\quad F(x)\coloneqq(x_1,-x_1,x_1+x_2) \in \Gamma \coloneqq \R_- \times \R_- \times \R_-.
	\]	
	We note that $\bar x\coloneqq (0,0)$ is one of its minimizers, and
	we have
	\[
		T^\textup{lin}_{F,\Gamma}(\bar x)=C(\bar x)=T_X(\bar{x})=\widehat C(\bar x)=X=\{0\}\times\R_-.
	\]
	Choosing $\{(x^k,\lambda^k,\delta^k,\varepsilon^k)\}_{k=1}^\infty\subset\R^{2+3+3+2}$ according to
	\[
		x^k\coloneqq\left(0,-\frac1k\right),
		\quad
		\lambda^k\coloneqq (k,k+1,0),
		\quad
		\delta^k\coloneqq (0,0,0),
		\quad
		\varepsilon^k\coloneqq (0,0),
		\qquad
		\forall k\in\N
	\] 
	confirms that $\bar x$ is AM-stationary in direction $d\coloneqq(0,-1)\in C(\bar x)\cap\mathbb S$.
	As we are considering linear inequality constraints, 
	the associated feasibility mapping is polyhedral,
	which ensures its metric subregularity at each point of its graph
	and, hence, (directional) MSCQ at $\bar{x}$.
	Thus, $\bar x$ is M-stationary (in direction $d$) by \cref{cor:MStat_via_MS}.
	However, $\{\lambda^k\}_{k=1}^\infty$ is unbounded.
	Note that FOSCMS$(d)$ at $\bar{x}$ reduces to
	\[
		(0,0) 
		= 
		(1,0)\lambda_1 + (-1,0)\lambda_2 + (1,1)\lambda_3,\quad
		\lambda_1,\lambda_2\geq 0,\,\lambda_3=0
		\quad
		\Longrightarrow
		\quad
		\lambda = 0,
	\]	
	see also \eqref{eq:dir_normal_cone_for_convex}, 
	and is indeed clearly violated.
\end{example}

Given a local minimizer $\bar{x}\in\R^n$ of \eqref{eq:nonsmooth_problem}, 
we have seen that MSCQ$(d)$ at $\bar{x}$ for a critical direction $d\in C(\bar x)\cap\mathbb S$ 
does not guarantee boundedness of $\{ \lambda^k \}_{k=1}^\infty$ for an arbitrary 
AM-stationary sequence $\{(x^k,\lambda^k,\delta^k,\varepsilon^k)\}_{k=1}^\infty\subset\R^{n+\ell+\ell+n}$ 
w.r.t.\ $\bar{x}$ in direction $d$. In this regard, the upcoming result 
is remarkable, as it shows that there necessarily still exists at least \textit{one} 
AM-stationary sequence w.r.t.\ $\bar{x}$ in direction $d$ with bounded 
$\{ \lambda^k \}_{k=1}^\infty$ under MSCQ$(d)$ at $\bar x$ (which, as we know from 
\cref{lem:preimage_rule_dir_MS}\,\ref{lem:preimage_rule_dir_MS_dir}, then also
implies $d\in \widehat C(\bar x)\cap\mathbb S$). 

Additionally, for arbitrary $d\in \widehat C(\bar{x})\cap\mathbb S$, the following result 
shows that AM-stationarity in direction $d$ is 
a necessary optimality condition that holds independently of the validity of qualification conditions.

\begin{theorem}\label{thm:locmin_dir_akkt_II}
	Let $\bar x\in\R^n$ be a local minimizer of \eqref{eq:nonsmooth_problem},
	and let $d\in\widehat C(\bar x)\cap\mathbb S$ be chosen arbitrarily.
	Then there exists an AM-stationary sequence
	$\{(x^k,\lambda^k,\delta^k,\varepsilon^k)\}_{k=1}^\infty\subset\R^{n+\ell+\ell+n}$
	w.r.t.\ $\bar x$ in direction $d$ such that
	\[
		\frac{\nnorm{\delta^k}\nnorm{\lambda^k}}{\nnorm{x^k-\bar x}}\to 0.
	\]
	If, additionally, MSCQ$(d)$ holds at $\bar x$,
	then $\{\lambda^k\}_{k=1}^\infty$ is bounded
	for this particular AM-stationary sequence.
\end{theorem}
\begin{proof}
	Let $\varepsilon>0$ be chosen such that $f(x)\geq f(\bar x)$ 
	is valid for all $x\in X\cap\mathbb B_\varepsilon(\bar x)$.
	Due to $d\in\widehat C(\bar x)$, 
	on the one hand, \cref{prop:primal_fo_conditions}\,\ref{item:primal_fo_condition_nec}
	implies $f'(\bar x)d=0$.
	On the other hand,
	we find sequences $\{\tilde x^k\}_{k=1}^\infty\subset X$
	and $\{t_k\}_{k=1}^\infty\subset\R_+$ such that
	$\tilde x^k\to\bar x$, $t_k\downarrow 0$, and $(\tilde x^k-\bar x)/t_k\to d$.
	Furthermore, we may assume $\tilde x^k\neq \bar x$ for all $k\in\N$ as $d\in\mathbb S$.
	We note that
	\[
		\frac{\tilde x^k-\bar x}{\nnorm{\tilde x^k-\bar x}}
		=
		\frac{\tilde x^k-\bar x}{t_k}
		\frac{t_k}{\nnorm{\tilde x^k-\bar x}}
		\to
		\frac{d}{\nnorm{d}}
		=
		d
	\]
	and
	\begin{equation}\label{eq:t_k_vs_norm_of_iterates}
		\frac{t_k}{\nnorm{\tilde x^k-\bar x}}\to \frac{1}{\nnorm{d}}=1.
	\end{equation}
	As $\tilde x^k \to \bar{x}$ yields $\tilde x^k \in \mathbb B_\varepsilon(\bar x)$
	for sufficiently large $k\in\N$, 
	let us pass, without loss of generality, to the subsequence 
	for which $\tilde x^k \in \B_\varepsilon(\bar{x})$ holds for all $k\in\N$.
	Then, for all $k\in\N$, $f(\tilde x^k)\geq f(\bar x)$ is true
	due to $\tilde x^k\in X\cap\mathbb B_\varepsilon(\bar x)$. Together with
	\[
		0
		=
		f'(\bar x)d 
		= 
		\lim\limits_{k\to\infty}\frac{f(\tilde x^k)-f(\bar x)}{\nnorm{\tilde x^k-\bar x}}
	\]
	and \eqref{eq:t_k_vs_norm_of_iterates}, 
	this guarantees the existence of a sequence $\{r_k\}_{k=1}^\infty\subset\R_+$ such that $r_k\downarrow 0$ and
	\begin{equation}\label{eq:estimating_directional_derivative_objective}
		0< f(\tilde x^k)-f(\bar x)+ t_k^2 \leq t_kr_k^2,
		\qquad
		\forall k\in\N.
	\end{equation}
	For each $k\in\N$, 
	we note that $\tilde x^k$ is an $f(\tilde x^k)-f(\bar x)+t_k^2$-minimizer of the restricted problem
	\[
		\begin{aligned}
		&\min\limits_{x}&	&f(x)&	 
		\\
		&\text{\,s.t.}&		&F(x)\in \Gamma,\,x\in \mathbb B_\varepsilon(\bar x).&
		\end{aligned}
	\]
	Applying Ekeland's variational principle, 
	see e.g.\ \cite[Proposition~1.43]{RockafellarWets1998},
	we find, for all $k\in\N$, some $\hat x^k\in\R^n$ with $\nnorm{\hat x^k-\tilde x^k}\leq t_kr_k$
	such that $\hat x^k$ is the uniquely determined minimizer of problem
	\begin{equation}\label{eq:ekeland_problem}\tag{P$^\textup{Eke}(k)$}
		\begin{aligned}
		&\min\limits_{x}&	&f(x) + \frac{f(\tilde x^k)-f(\bar x)+t_k^2}{t_kr_k}\nnorm{x-\hat x^k}&	 
		\\
		&\text{\,s.t.}&		&F(x)\in \Gamma,\,x\in \mathbb B_\varepsilon(\bar x).&
		\end{aligned}
	\end{equation}
	Let us note that
	\[
		\frac{\nnorm{\hat x^k-\bar x}}{t_k}
		\leq
		\frac{\nnorm{\tilde x^k-\bar x}}{t_k} + \frac{\nnorm{\hat x^k-\tilde x^k}}{t_k}
		\leq
		\frac{\nnorm{\tilde x^k-\bar x}}{t_k} + r_k
		\to 1
	\]
	and
	\[
		\frac{\nnorm{\hat x^k-\bar x}}{t_k}
		\geq
		\frac{\nnorm{\tilde x^k-\bar x}}{t_k} - \frac{\nnorm{\hat x^k-\tilde x^k}}{t_k}
		\geq
		\frac{\nnorm{\tilde x^k-\bar x}}{t_k} - r_k
		\to 1
	\]
	yield $\nnorm{\hat x^k-\bar x}/t_k\to 1$, and this gives $\hat x^k \to \bar x$ as well as
	\begin{align*}
		\frac{\hat x^k-\bar x}{\nnorm{\hat x^k-\bar x}}
		&=
		\frac{\hat x^k-\tilde x^k}{\nnorm{\hat x^k-\bar x}}
		+
		\frac{\tilde x^k-\bar x}{\nnorm{\hat x^k-\bar x}}
		\\
		&=
		\frac{\hat x^k-\tilde x^k}{t_k}\frac{t_k}{\nnorm{\hat x^k-\bar x}}
		+
		\frac{\tilde x^k-\bar x}{t_k}\frac{t_k}{\nnorm{\hat x^k-\bar x}}
		\to d.
	\end{align*}
	To proceed, let us inspect the penalized problem
	\begin{equation}\label{eq:ekeland_problem_penalized}\tag{P$^\textup{Eke}(k,\mu)$}
		\begin{aligned}
		&\min\limits_{x,\delta}&	&f(x) + \frac{f(\tilde x^k)-f(\bar x)+t_k^2}{t_kr_k}\nnorm{x-\hat x^k} 
										  + \frac{\mu}{2}\nnorm{\delta}^2 + \frac12\nnorm{x-\hat x^k}^2&	 
		\\
		&\text{\,s.t.}&				&F(x)-\delta\in \Gamma,&
		\\
		&&							&x\in \mathbb B_\varepsilon(\bar x),\,\delta\in\mathbb B_1(0)&
		\end{aligned}
	\end{equation}
	for all $k\in\N$, where $\mu\in\N$ is a penalty parameter.
	As $\hat x^k$ is the minimizer of \eqref{eq:ekeland_problem} for all $k\in\N$,
	$(\hat x^k,0)$ is feasible for \eqref{eq:ekeland_problem_penalized} for all $k,\mu\in\N$.
	Thus, together with continuity of $F$ and closedness of $\Gamma$, it is guaranteed
	that the feasible set of \eqref{eq:ekeland_problem_penalized} 
	is nonempty and compact for all $k,\mu\in\N$.
	As the objective function of \eqref{eq:ekeland_problem_penalized} is continuous,
	the latter admits a global minimizer $(\hat x^{k,\mu},\hat \delta^{k,\mu})$ for all $k,\mu\in\N$.
	Noting that $\{\hat x^{k,\mu}\}_{\mu=1}^\infty$ remains bounded by definition,
	we can take a subsequence (without relabeling)
	in order to find, for all $k\in\N$, some $\check x^k\in\mathbb B_\varepsilon(\bar x)$ such that 
	$\hat x^{k,\mu}\to\check x^k$ as $\mu\to\infty$.
	
	Recalling that $(\hat x^k,0)$ is feasible for \eqref{eq:ekeland_problem_penalized} for all $k,\mu\in\N$,
	we find
	\begin{equation}\label{eq:penalty_estimate_for_ekeland_problem}
		\begin{aligned}
		f(\hat x^{k,\mu}) 
		&
		+ 
		\frac{f(\tilde x^k)-f(\bar x)+t_k^2}{t_kr_k}\nnorm{\hat x^{k,\mu}-\hat x^k} 
		\\
		&+ 
		\frac{\mu}{2}\nnorm{\hat\delta^{k,\mu}}^2 
		+ 
		\frac12\nnorm{\hat x^{k,\mu}-\hat x^k}^2
		\leq
		f(\hat x^k),
		\qquad
		\forall k,\mu\in\N
		\end{aligned}
	\end{equation}
	by global optimality of $(\hat x^{k,\mu},\hat\delta^{k,\mu})$.
	From \eqref{eq:estimating_directional_derivative_objective} and
	\eqref{eq:penalty_estimate_for_ekeland_problem}, we obtain
	\[
		\nnorm{\hat\delta^{k,\mu}}^2
		\leq
		\frac{2}{\mu}(f(\hat x^k)-f(\hat x^{k,\mu})),
		\qquad
		\forall k,\mu\in\N,
	\]
	and, for all $k\in\N$, since $\{f(\hat x^{k,\mu})\}_{\mu=1}^\infty$ is bounded by continuity of $f$,
	$\hat\delta^{k,\mu}\to 0$ follows as $\mu\to\infty$.
	Thus, from $F(\hat x^{k,\mu})-\hat\delta^{k,\mu}\in\Gamma$ for each $k,\mu\in\N$
	and closedness of $\Gamma$, it follows for each $k\in\N$ that $F(\check x^k)\in\Gamma$, 
	implying feasibility of $\check x^k$ for \eqref{eq:ekeland_problem}.
	As $\hat x^k$ is the minimizer of the latter for each $k\in\N$, 
	\[
		f(\hat x^k) 
		\leq 
		f(\check x^k) + \frac{f(\tilde x^k)-f(\bar x)+t_k^2}{t_kr_k}\nnorm{\check x^k-\hat x^k},
		\qquad
		\forall k\in\N
	\]
	follows. For all $k\in\N$, exploiting \eqref{eq:penalty_estimate_for_ekeland_problem} once more, we find
	\begin{align*}
		f(\hat x^k) 
		&\leq 
		f(\check x^k) + \frac{f(\tilde x^k)-f(\bar x)+t_k^2}{t_kr_k}\nnorm{\check x^k-\hat x^k}
		\\
		&\leq
		f(\check x^k) + \frac{f(\tilde x^k)-f(\bar x)+t_k^2}{t_kr_k}\nnorm{\check x^k-\hat x^k}
			+ \frac12\nnorm{\check x^k-\hat x^k}^2
		\\
		&=
		\lim\limits_{\mu\to\infty}
		\left(
			f(\hat x^{k,\mu}) 
			+ 
			\frac{f(\tilde x^k)-f(\bar x)+t_k^2}{t_kr_k}\nnorm{\hat x^{k,\mu}-\hat x^k}
			+ 
			\frac12\nnorm{\hat x^{k,\mu}-\hat x^k}^2
		\right)
		\\
		&\leq
		f(\hat x^k),
	\end{align*}
	and $\check x^k=\hat x^k$ follows.
	Particularly, we have $\hat x^{k,\mu}\to\hat x^k$ as $\mu\to\infty$ for all $k\in\N$.
	
	Let us assume, by considering the tail of the sequences if necessary,
	that $\nnorm{\hat x^{k,\mu} - \bar x}<\varepsilon$ and $\nnorm{\hat\delta^{k,\mu}}<1$
	hold for all $k,\mu\in\N$.
	The former of these requirements can be ensured as we know $\hat x^{k,\mu}\to\hat x^k$ as
	$\mu\to\infty$ and $\hat x^k\to\bar x$.
	As the objective function of \eqref{eq:ekeland_problem_penalized} is locally
	Lipschitz continuous for all $k,\mu\in\N$, we can apply the necessary optimality conditions
	from \cite[Theorem~8.15]{RockafellarWets1998}
	together with \cite[Exercises~8.8(c), 8.27]{RockafellarWets1998} and
	\cref{lem:normals_to_perturbed_feasible_set} to find
	$\hat \xi^{k,\mu}\in\mathbb B_1(0)$ such that, for all $k,\mu\in\N$,
	\begin{align*}
		\hat x^k-\hat x^{k,\mu} 
		- 
		\frac{f(\tilde x^k)-f(\bar x)+t_k^2}{t_kr_k}\hat \xi^{k,\mu}
		-
		\nabla f(\hat x^{k,\mu})
		&\in 
		\partial\langle\mu\hat\delta^{k,\mu},F\rangle(\hat x^{k,\mu}),
		\\
		\mu\hat\delta^{k,\mu}
		&\in 
		N_\Gamma(F(\hat x^{k,\mu})-\hat \delta^{k,\mu}).
	\end{align*}
	For each $k\in\N$, pick $\mu(k)\in\N$ so large such that
	$\nnorm{\hat x^{k,\mu(k)}-\hat x^k}\leq t_kr_k$ and $\nnorm{\hat \delta^{k,\mu(k)}}\leq 1/k$,
	and set
	\begin{equation}\label{eq:some_settings}
		\begin{aligned}
		x^k&\coloneqq \hat x^{k,\mu(k)},&
		\quad
		\hat\lambda^k&\coloneqq \mu(k)\hat\delta^{k,\mu(k)},&
		\\
		\hat\delta^k&\coloneqq \hat \delta^{k,\mu(k)},&
		\quad
		\hat\varepsilon^k&\coloneqq \hat x^k-\hat x^{k,\mu(k)}
		-
		\frac{f(\tilde x^k)-f(\bar x)+t_k^2}{t_kr_k}\hat \xi^{k,\mu(k)}.&
		\end{aligned}	
	\end{equation}
	Then, from above, we find
	\begin{subequations}\label{eq:nec_cond_pen}
		\begin{align}
			\label{eq:nec_cond_pen_i}
			\hat\varepsilon^k-\nabla f(x^k)&\in \partial\langle\hat\lambda^k,F\rangle(x^k),
			\\
			\label{eq:nec_cond_pen_ii}
			\hat\lambda^k&\in N_\Gamma(F(x^k)-\hat\delta^k)
		\end{align}
	\end{subequations}
	for each $k\in\N$, and $\nnorm{x^k-\hat x^k}\leq t_kr_k$ is valid as well.
	Let us note that
	\[
		\frac{\nnorm{x^k-\bar x}}{t_k}
		\leq
		\frac{\nnorm{\hat x^k-\bar x}}{t_k} + \frac{\nnorm{x^k-\hat x^k}}{t_k}
		\leq
		\frac{\nnorm{\hat x^k-\bar x}}{t_k} + r_k
		\to
		1
	\]
	and
	\[
		\frac{\nnorm{x^k-\bar x}}{t_k}
		\geq
		\frac{\nnorm{\hat x^k-\bar x}}{t_k} - \frac{\nnorm{x^k-\hat x^k}}{t_k}
		\geq
		\frac{\nnorm{\hat x^k-\bar x}}{t_k} - r_k
		\to
		1
	\]
	yield $\nnorm{x^k-\bar x}/t_k\to 1$, and this gives $x^k \to \bar{x}$ as well as 
	\begin{align*}
		\frac{x^k-\bar x}{\nnorm{x^k-\bar x}}
		&=
		\frac{x^k-\hat x^k}{\nnorm{x^k-\bar x}}
		+
		\frac{\hat x^k-\bar x}{\nnorm{x^k-\bar x}}
		\\
		&=
		\frac{x^k-\hat x^k}{t_k}\frac{t_k}{\nnorm{x^k-\bar x}}
		+
		\frac{\hat x^k-\bar x}{\nnorm{\hat x^k-\bar x}}
		\frac{\nnorm{\hat x^k-\bar x}}{t_k}\frac{t_k}{\nnorm{x^k-\bar x}}
		\to d.
	\end{align*}
	Particularly, $x^k\neq\bar x$ is valid for all large enough $k\in\N$, and we pass,
	without loss of generality, to the subsequence for which $x^k\neq\bar x$ holds for
	all $k\in\N$.
	From above $\hat\delta^k\to 0$ is also true.
	Furthermore, we have for all $k\in\N$ that
	\[
		\nnorm{\hat\varepsilon^k}
		\leq
		\nnorm{\hat x^k-x^k} + \frac{f(\tilde x^k)-f(\bar x)+t_k^2}{t_kr_k}
		\leq
		(t_k+1)r_k
		\to
		0,
	\]
	and $\hat\varepsilon^k\to 0$ follows.
	
	To finalize the proof, we distinguish between two cases.
	First, let us assume that $\{\hat\lambda^k\}_{k=1}^\infty$ does not converge
	to zero. Then, along a subsequence (without relabeling),
	$\{\nnorm{\hat\lambda^k}\}_{k=1}^\infty$ is bounded away from zero.
	Let us note that this, particularly, yields $\hat\delta^k\neq 0$
	for all $k\in\N$ due to \eqref{eq:some_settings}.
	This, in turn, also yields $x^k\notin X$ for all $k\in\N$.
	Indeed, if $x^k\in X$ for some $k\in\N$,
	then, recalling that $x^k\in\mathbb B_\varepsilon(\bar x)$ holds true,
	$(x^k,0)$ would be feasible for
	\hyperref[eq:ekeland_problem_penalized]{(P$^\textup{Eke}(k,\mu(k))$)}
	with a smaller objective value than $(x^k,\hat\delta^k)$,
	which is a contradiction to the global optimality of the latter.
	From \eqref{eq:estimating_directional_derivative_objective},
	\eqref{eq:penalty_estimate_for_ekeland_problem}, and \eqref{eq:some_settings}
	we find
	\[
		\frac{f(x^k)-f(\bar x)}{\nnorm{x^k-\bar x}}
		-
		\frac{f(\hat x^k)-f(\bar x)}{\nnorm{x^k-\bar x}}
		+
		\frac12\frac{\nnorm{\hat\delta^k}\nnorm{\hat\lambda^k}}{\nnorm{x^k-\bar x}}
		\leq 0,
		\qquad
		\forall k\in\N.
	\]
	The first summand converges to $f'(\bar x)d=0$,
	and it is easy to show that the same holds true for the second summand
	by exploiting $\nnorm{\hat x^k-\bar x}/t_k\to 1$ and $\nnorm{x^k-\bar x}/t_k\to 1$.
	Hence, the third summand has to converge to $0$ as well.
	Our assumptions, thus, guarantee that $\hat\delta^k/\nnorm{x^k-\bar x}\to 0$.
	Furthermore, we have $\nnorm{\hat\delta^k}\hat\lambda^k=\nnorm{\hat\lambda^k}\hat\delta^k$ 
	for each $k\in\N$ by construction.
	Thus, choosing $\lambda^k\coloneqq \hat\lambda^k$, $\delta^k\coloneqq \hat\delta^k$,
	and $\varepsilon^k\coloneqq \hat\varepsilon^k$ for each $k\in\N$ yields the first claim.
	
	Regarding the second claim, let us show that $\{\lambda^k\}_{k=1}^\infty$ 
	remains bounded if MSCQ$(d)$ holds at $\bar x$.
	To this end, we first note that, for any $\tilde \varepsilon >0$ and $\tilde \delta >0$, we have
	$ x^k \in \{\bar{x}\} + \mathbb B_{\tilde\varepsilon,\tilde\delta}(d) $
	for sufficiently large $k\in\mathbb{N}$ due to $x^k\to\bar{x}$ and 
	$(x^k-\bar x)/\nnorm{x^k-\bar x}\to d$.
	Now, assuming that MSCQ$(d)$ holds at $\bar{x}$, there exists $\kappa>0$ such that,
	for sufficiently large $k\in\mathbb{N}$, we find $\grave x^k\in F^{-1}(\Gamma)$ with
	\begin{equation}\label{eq:applying_dir_mscq_in_proof}
		\nnorm{x^k - \grave x^k} 
		= 
		\dist(x^k, F^{-1}(\Gamma)) 
		\leq 
		\kappa \dist(F(x^k),\Gamma) 
		\leq 
		\kappa \nnorm { \delta^k },
	\end{equation}
	where the last inequality follows from $F(x^k)-\delta^k\in \Gamma$.
	Let us also note that $\nnorm{x^k-\grave x^k}>0$ holds true for all sufficiently large $k\in\N$
	as $x^k\notin X=F^{-1}(\Gamma)$ is valid.
	Due to $\delta^k \to 0$, we obtain $\nnorm{x^k - \grave x^k}\to 0$, and, in turn,
	the convergence $\grave x^k \to \bar{x}$. As this implies 
	$\grave x^k \in \mathbb B_\varepsilon(\bar x)$ for $k\in\mathbb{N}$ large enough
	while $F(\grave x^k) \in \Gamma$ holds by the choice of $\grave x^k$, we obtain that 
	$(\grave x^k, 0)$ is feasible for \eqref{eq:ekeland_problem_penalized} for all 
	$k\in\mathbb{N}$ large enough and $\mu\in\mathbb{N}$.
	Recalling that $(x^k,\delta^k)$ is a global minimizer of 
	\hyperref[eq:ekeland_problem_penalized]{(P$^\textup{Eke}(k,\mu(k))$)},
	it follows for all $k\in\mathbb{N}$ large enough that
	\begin{align*}
		& f(x^k) + \frac{f(\tilde x^k)-f(\bar x)+t_k^2}{t_kr_k}\nnorm{x^k-\hat x^k} 
					+ \frac{\mu(k)}{2}\nnorm{\delta^k}^2 + \frac12\nnorm{x^k-\hat x^k}^2 \\
		\leq\,
		& f(\grave x^k) + \frac{f(\tilde x^k)-f(\bar x)+t_k^2}{t_kr_k}\nnorm{\grave x^k-\hat x^k} 
					+ \frac12\nnorm{\grave x^k-\hat x^k}^2,
	\end{align*}
	which yields, together with \eqref{eq:estimating_directional_derivative_objective},
	\eqref{eq:some_settings}, \eqref{eq:applying_dir_mscq_in_proof}, and
	$\delta^k \neq 0$ for all $k\in\N$,
	\begin{align*}
		\nnorm{ \lambda^k }
		=\, &
		\frac{\mu(k) \nnorm{ \delta^k }^2}{\nnorm{ \delta^k }}  \\
		\leq\, &
		\,\frac{2(f(\grave x^k)-f(x^k)) }{\nnorm{ \delta^k }}
		+ 	\frac{2(f(\tilde x^k)-f(\bar x)+t_k^2)}{t_kr_k}
			\frac{\nnorm{\grave x^k-\hat x^k} -\nnorm{x^k-\hat x^k}}{\nnorm{ \delta^k }} \\
		&\,+ \frac{\nnorm{\grave x^k-\hat x^k}^2-\nnorm{x^k-\hat x^k}^2}{{\nnorm{ \delta^k }}} \\
		\leq\, &
		\frac{2\kappa(f(\grave x^k)-f(x^k)) }{\nnorm{ x^k -\grave x^k }}
		+ 2\kappa r_k \frac{\nnorm{\grave x^k-\hat x^k} -\nnorm{x^k-\hat x^k}}{\nnorm{ x^k -\grave x^k }}
		+ \frac{\kappa(\nnorm{\grave x^k-\hat x^k}^2-\nnorm{x^k-\hat x^k}^2)}{{\nnorm{ x^k -\grave x^k}}}.
	\end{align*}
	The first summand is bounded as $f$ is continuously differentiable
	while $x^k\to\bar x$ and $\grave x^k\to\bar x$,
	whereas the second and third summands are bounded due to the inequalities
	\[
		\nnorm{\grave x^k-\hat x^k} -\nnorm{x^k-\hat x^k} \leq \nnorm{ x^k -\grave x^k}
	\]
	and
	\[
		\left| \nnorm{\grave x^k-\hat x^k}^2-\nnorm{x^k-\hat x^k}^2 \right|
		\leq
		\nnorm{\grave x^k - x^k} (\nnorm{x^k - \hat x^k} + \nnorm{ \grave x^k - \hat x^k})
	\]
	for all $k\in\N$, respectively. Thus, we have shown that $\{\lambda^k\}_{k=1}^\infty$ is bounded.
	
	Considering the second case, assume that $\hat\lambda^k\to 0$.
	For each $k\in\N$, \eqref{eq:nec_cond_pen_i} yields the existence of
	$\eta^k\in \partial\langle\hat\lambda^k,F\rangle(x^k)$ such that $\hat\varepsilon^k-\nabla f(x^k) = \eta^k$.
	Recalling \eqref{eq:scalarization_rule} and local Lipschitz continuity of $F$,
	\cite[Theorem~4.7]{Mordukhovich2006} guarantees $\eta^k\to 0$.
	Hence, noting that $0\in \partial\langle 0, F\rangle(x^k)$ and $0\in N_\Gamma(F(x^k))$ hold for each $k\in\N$,
	we can choose $\lambda^k\coloneqq 0$, $\delta^k\coloneqq 0$, and 
	$\varepsilon^k\coloneqq \hat\varepsilon^k-\eta^k$ for each $k\in\N$ in order to
	show the first claim.
	As $\{\lambda^k\}_{k=1}^\infty$ is bounded, the second assertion of the theorem is
	trivial in this case.
\end{proof}

In particular, the proof of \cref{thm:locmin_dir_akkt_II} shows the following.
\begin{remark}\label{rem:nonvanishing_seq}
	Let $\bar x\in\R^n$ be a local minimizer of \eqref{eq:nonsmooth_problem}, and let 
	$d\in\widehat C(\bar x)\cap \mathbb{S}$ be chosen arbitrarily. 
	Then we either have $\nabla f(\bar x) = 0$ or there
	exists an AM-stationary sequence
	$\{(x^k,\lambda^k,\delta^k,\varepsilon^k)\}_{k=1}^\infty
	\subset
	\R^{n+\ell+\ell+n}$
	w.r.t.\ $\bar x$ in direction $d$ where $\{ \nnorm{\lambda^k} \}_{k=1}^\infty$ is bounded 
	away from zero and $\delta^k \neq 0$ as well as $x^k\notin X$ hold for all $k\in\N$.
	This type of disjunction is rather clunky to work with as it is a mixture 
	of a point-based as well as a sequence-based condition.
	However, as the proof of \cref{thm:locmin_dir_akkt_II} illustrates, 
	if $\nabla f(\bar x) = 0$, we can still find a sequence
	$\{(x^k,\lambda^k,\delta^k,\varepsilon^k)\}_{k=1}^\infty
		\subset
		\R^{n+\ell+\ell+n}$
	that is AM-stationary w.r.t.\ $\bar x$ in direction $d$, but we cannot ensure 
	that the multipliers $\lambda^k$ and perturbations $\delta^k$
	of this sequence are nonvanishing for all $k\in\N$. 
	Beware that our definition of an AM-stationary sequence w.r.t.\ $\bar x$
	in direction $d$ covers both situations by merely requiring
	\eqref{eq:perturbation_vs_multiplier} instead of the seemingly more natural condition
	\eqref{eq:residual_va_multiplier_nontrivial}, as the latter
	implicitly requires $\delta^k\neq 0$ and $\lambda^k\neq 0$ for all $k\in\N$.
\end{remark}

As a corollary of \cref{thm:locmin_dir_akkt_II}, we are now in position to verify in an 
alternative way, which exploits AM-stationarity, that the result from
\cref{cor:MStat_via_MS}~\ref{cor:MStat_via_MS2} holds, i.e., that a given local 
minimizer of \eqref{eq:nonsmooth_problem}
is M-stationary in a critical direction $d\in C(\bar x) \cap \mathbb S$
provided MSCQ$(d)$ holds at this point.

\begin{corollary}\label{thm:locmin_dir_subreg}
	Let $\bar x\in\R^n$ be a local minimizer of \eqref{eq:nonsmooth_problem},
	and let $d\in C(\bar x) \cap \mathbb S$ be chosen such
	that MSCQ$(d)$ holds at $\bar{x}$.
	Then $\bar x$ is M-stationary in direction $d$.
\end{corollary}
\begin{proof}
	With the aid of \cref{lem:preimage_rule_dir_MS}\,\ref{lem:preimage_rule_dir_MS_dir}
	we immediately obtain that 
	MSCQ$(d)$ at $\bar{x}$
	implies $d\in\widehat C(\bar x) \cap \mathbb S$.
	The existence of an AM-stationary sequence
	$\{(x^k,\lambda^k,\delta^k,\varepsilon^k)\}_{k=1}^\infty
		\subset
		\R^{n+\ell+\ell+n}$ w.r.t.\ $\bar x$ in direction $d$ such that
	 $\{\lambda^k\}_{k=1}^\infty$ is bounded now follows 
	from \cref{thm:locmin_dir_akkt_II}.
	Hence, M-stationarity of $\bar x$ in direction $d$ is clear from 
	\cref{lem:taking_the_limit_in_directional_approximate_conditions}.
\end{proof}

Another important consequence of \cref{thm:locmin_dir_akkt_II} is listed below.
\begin{corollary}\label{cor:dir_asymp_stat_under_GACQ}
	Let $\bar x\in\R^n$ be a local minimizer of \eqref{eq:nonsmooth_problem}
	at which GACQ holds.
	Then $\bar x$ is AM-stationary in direction $d$ 
	for all $d\in C(\bar x)\cap\mathbb S$.
\end{corollary}

We have seen in \cref{sec:dir_stat} that,
even in situations where $F$ is continuously differentiable,
GACQ alone might not be enough to guarantee M-stationarity
of local minimizers in all critical directions,
see \cref{ex:SOC_example}.
In fact, some additional qualification conditions on the problem
data are required to get this assertion,
like local polyhedrality of $\Gamma$ 
or local convexity of $\Gamma$ together with directional metric subregularity
of a linearized feasibility mapping,
see \cref{lem:dir_MSt_under_GACQ} and \cref{cor:dir_MSt_disjunctive}.
Similarly, \cref{cor:dir_asymp_stat_under_GACQ} merely guarantees
AM-stationarity of local minimizers in all critical directions in the presence
of GACQ. Some additional mild directional qualification condition then
might be enough to guarantee directional M-stationarity,
as we will illustrate later.

Recall that \cref{thm:locmin_dir_akkt_II} implies, given a local minimizer $\bar x\in\R^n$ of \eqref{eq:nonsmooth_problem},
that $\bar x$ is AM-stationary in each direction from $\widehat C(\bar x)\cap\mathbb S$.
Naturally, one may ask the question whether $\bar x$ 
is already AM-stationary in each direction from the generally larger set $ C(\bar x)\cap\mathbb S$, 
which would be of special interest as the implicit critical cone $\widehat C(\bar x)$ is, as the name suggests,
an implicit object and typically hard to compute in the absence of qualification conditions.
As the following example shows, this is unfortunately not the case.

\begin{example}\label{ex:AKKT_in_critical_directions}
	Consider the optimization problem
	\[
		\min\limits_x\quad x_1+x_2\quad\textup{s.t.}\quad F(x)\coloneqq x_1^2+x_2^2 \in \Gamma \coloneqq \R_-.
	\]
	The origin $\bar x \coloneqq (0,0)$ is the unique feasible point and, consequently, the global minimizer.
	One can check that we have $\widehat C(\bar x)=\{0\}\times\{0\}$ and
	$C(\bar x)=\{d\in\R^2\,|\,d_1+d_2\leq 0\}$. 
	We note that \cref{thm:locmin_dir_akkt_II} does not yield any helpful
	information in the present situation as $\widehat C(\bar x)\cap\mathbb S$ is empty.
	Pick $d\in C(\bar x)\cap\mathbb S$ arbitrarily
	and let $\{(x^k,\lambda^k,\delta^k,\varepsilon^k)\}_{k=1}^\infty\subset\R^{2+1+1+2}$ 
	be an AM-stationary sequence w.r.t.\ $\bar x$ in direction $d$.
	This requires
	\begin{align*}
		&\varepsilon^k
		-
		\begin{pmatrix}
			1 \\ 1	
		\end{pmatrix}
		=
		2\lambda^kx^k,
		\quad
		x^k\neq\bar x,
		\\
		&\lambda^k\geq 0,\quad (x_1^k)^2+(x_2^k)^2\leq\delta^k,\quad\lambda^k((x_1^k)^2+(x_2^k)^2-\delta^k)=0
	\end{align*}
	for all $k\in\N$, the convergences \eqref{eq:sequences_conv}, and \eqref{eq:approx_stat_dir}.
	Indeed, one could choose 
	\[
		x^k\coloneqq \left(-\frac{1}{2k},-\frac{1}{2k}\right),
		\quad
		\lambda^k\coloneqq k,
		\quad
		\delta^k\coloneqq\frac1{2k^2},
		\quad
		\varepsilon^k\coloneqq (0,0),
		\qquad
		\forall k\in\N
	\]
	in order to satisfy all these conditions for direction 
	$d\coloneqq-\frac{1}{\sqrt 2}(1,1)\in C(\bar x)\cap\mathbb S$.
	
	Let us now pick direction $d\coloneqq (-1,0) \in C(\bar x)\cap\mathbb S$
	and assume that there exist
	sequences which satisfy the conditions stated above.
	By \eqref{eq:sequences_dir_conv} this particularly requires the directional convergences
	\[
		\frac{x_1^k}{\norm{x^k}}\to-1,\quad\frac{x_2^k}{\norm{x^k}}\to 0.
	\]
	Note that $2\lambda^kx^k_i=\varepsilon^k_i-1\to -1$, $i=1,2$, immediately yields
	that $x^k_1$, $x^k_2$, and $\lambda^k$ cannot vanish for all $k\in\N$ large enough.
	Hence, we find
	\[
		\frac{\varepsilon^k_1-1}{2\lambda^k\norm{x^k}}=\frac{x^k_1}{\norm{x^k}}\to -1,
		\quad
		\frac{\varepsilon^k_2-1}{2\lambda^k\norm{x^k}}=\frac{x^k_2}{\norm{x^k}}\to 0.
	\]
	Due to $\varepsilon^k\to 0$, this can only be fulfilled if $\lambda^k\nnorm{x^k}\to\frac12$
	and $\lambda^k\nnorm{x^k}\to \infty$ hold simultaneously.
	However, this is clearly impossible.
	Hence, the assertion of \cref{thm:locmin_dir_akkt_II} 
	does not necessarily hold for each critical direction from $C(\bar x)\cap\mathbb S$. 
\end{example}

In \cref{ex:AKKT_in_critical_directions}
we have shown that the assertion of \cref{thm:locmin_dir_akkt_II} cannot be
extended to all directions from $C(\bar x)\cap\mathbb S$,
but we also saw therein that \emph{some} critical direction from $C(\bar x)\cap\mathbb S$
\emph{exists} such that the local minimizer $\bar x\in\R^n$ under consideration
is AM-stationary in this direction.
Below, we show by a result similar to \cref{thm:locmin_dir_akkt_II} 
that this behavior did not happen by chance.
A similar observation has been made in \cite[Corollary~4.5]{BenkoMehlitz2024b}.
However, to make this paper self-contained, we provide a proof.
Note the similarities and differences to the proof of \cref{thm:locmin_dir_akkt_II}.

\begin{theorem}\label{thm:locmin_dir_akkt}
	If $\bar{x}\in\R^n$ is a local minimizer of \eqref{eq:nonsmooth_problem}, 
	then at least one of the following assertions holds.
	\begin{enumerate}
		\item\label{item:locmin_akkt1} 
			The point $\bar{x}$ is M-stationary.
		\item\label{item:locmin_akkt2} 
		 	There exist a critical direction $d\in C(\bar x)\cap\mathbb S$ 
		 	and an AM-stationary sequence 
			$\{(x^k,\lambda^k,\delta^k,\varepsilon^k)\}_{k=1}^\infty\subset\R^{n+\ell+\ell+n}$
			w.r.t.\ $\bar x$ in direction $d$
			that fulfills $\nnorm{\lambda^k}\to\infty$
			as well as $x^k\notin X$ and $\delta^k \neq 0$ for all $k\in\N$.
	\end{enumerate} 
\end{theorem}
\begin{proof}
	The assumptions of the theorem guarantee that there is some $\varepsilon>0$ such that
	$f(x)\geq f(\bar x)$ holds for all $x\in X\cap \mathbb B_\varepsilon(\bar x)$.
	For $k\in\N$, we now consider the surrogate problem
	\begin{equation}\label{eq:locmin_prob}\tag{P$(k)$}
		\begin{aligned}
		&\min\limits_{x,\delta}&	&f(x)+\frac{k}{2}\norm{\delta}^2 + \frac12\norm{x-\bar x}^2&	 
		\\
		&\text{\,s.t.}&				&F(x)-\delta\in \Gamma,&
		\\
		&&							&x\in \mathbb B_\varepsilon(\bar x),\,\delta\in \mathbb B_1(0).
	\end{aligned}
	\end{equation}
	Obviously, $(\bar x,0)$ is feasible for \eqref{eq:locmin_prob} for all $k\in\N$.
	Together with continuity of $F$ as well as closedness of $\Gamma$, this yields
	that the feasible set of \eqref{eq:locmin_prob} is nonempty and compact for all $k\in\N$.
	As the objective function of \eqref{eq:locmin_prob} is continuous as well,
	the latter admits a global minimizer $(x^k,\delta^k)$ for all $k\in\N$.
	Noting that $\{x^k\}_{k=1}^\infty$ remains
	bounded by definition,
	we can take a subsequence (without relabeling) in order to find $x^\star\in \mathbb B_\varepsilon(\bar x)$
	such that $x^k\to x^\star$.

	Recalling that $(\bar x,0)$ is feasible for \eqref{eq:locmin_prob} for each $k\in\N$,
	we obtain
	\begin{equation}
		\label{eq:locmin}
			f(x^k)+\frac{k}{2}\nnorm{\delta^k}^2+\frac12\nnorm{x^k-\bar x}^2
			\leq
			f(\bar x),
			\qquad
			\forall k\in\N
	\end{equation}
	by global optimality of $(x^k,\delta^k)$.
	By continuity of $f$, $\{f(x^k)\}_{k=1}^\infty$ is bounded, so that
	\[
		\nnorm{\delta^k}^2
		\leq
		\frac{2}{k}(f(\bar x)-f(x^k)),
		\qquad
		\forall k\in\N
	\]
	yields $\delta^k\to 0$.
	Hence, from $F(x^k)-\delta^k\in \Gamma$ for each $k\in\N$ and closedness of $\Gamma$, we find
	$F(x^\star)\in \Gamma$, i.e., $x^\star$ is feasible for \eqref{eq:nonsmooth_problem}.
	Thus, $x^\star\in \mathbb B_\varepsilon(\bar x)$ yields $f(\bar x)\leq f(x^\star)$.
	Exploiting \eqref{eq:locmin} once more, we find
	\begin{align*}
		f(\bar x)
		\leq 
		f(x^\star)
		\leq
		f(x^\star) + \frac12\nnorm{x^\star-\bar x}^2
		=
		\lim\limits_{k\to\infty}\left(f(x^k)+\frac12\nnorm{x^k-\bar x}^2\right)
		\leq
		f(\bar x),
	\end{align*}
	and $x^\star = \bar x$ follows.
	Particularly, this yields $x^k\to\bar x$.

	Let us assume, by considering the tail of the sequences if necessary,
	that $\nnorm{x^k-\bar x}<\varepsilon$ and $\nnorm{\delta^k}<1$ hold for all $k\in\N$.
	Recalling that $(x^k,\delta^k)$ is a global minimizer of \eqref{eq:locmin_prob} for all $k\in\N$,
	\cref{prop:dual_fo_conditions}\,\ref{item:dual_fo_cond_regular} and
	\cref{lem:normals_to_perturbed_feasible_set}
	yield
	\begin{align*}
		\bar x-x^k-\nabla f(x^k)&\in \partial\langle k\delta^k,F\rangle(x^k),
		\\
		k\delta^k&\in N_\Gamma(F(x^k)-\delta^k)
	\end{align*}
	for each $k\in\N$.
	Hence, defining 
	\begin{equation}\label{eq:settings_for_directional_stationarity}
		\lambda^k\coloneqq k\delta^k,\quad\varepsilon^k\coloneqq \bar x - x^k,
		\qquad\forall k\in\N
	\end{equation}
	yields validity of \eqref{eq:approximate_statonarity} for each $k\in\N$.
	Furthermore, $\varepsilon^k\to 0$ holds true as $x^k\to\bar x$.
	Hence, \eqref{eq:sequences_conv} is verified.
	
	Suppose that the sequence element $x^{k_0}$ is feasible for \eqref{eq:nonsmooth_problem} for some $k_0\in\N$.
	Then $f(\bar x)\leq f(x^{k_0})$ follows, and, together with \eqref{eq:locmin},
	$x^{k_0}=\bar x$ and $\delta^{k_0}=0$ are obtained.
	Particularly, \eqref{eq:approximate_statonarity} reduces to
	\[
		-\nabla f(\bar x)\in\partial\langle\lambda^{k_0},F\rangle(\bar x),\quad\lambda^{k_0}\in N_\Gamma(F(\bar x))
	\]
	for $k_0$, where we applied \eqref{eq:settings_for_directional_stationarity}.
	Thus, $\bar x$ is an M-stationary point of \eqref{eq:nonsmooth_problem} with multiplier $\lambda^{k_0}$,
	such that~\ref{item:locmin_akkt1} is verified.
	
	Consequently, we may assume that $x^k$ is not feasible for \eqref{eq:nonsmooth_problem} for all $k\in\N$.
	This, particularly, yields $x^k\neq\bar x$ and $\delta^k\neq 0$ for each $k\in\N$.
	Without loss of generality, let us assume that $\{(x^k-\bar x)/\nnorm{x^k-\bar x}\}_{k=1}^\infty$
	converges to some $d\in\mathbb S$.
	
	From \eqref{eq:locmin} and \eqref{eq:settings_for_directional_stationarity} we find
	\begin{equation}\label{eq:bounded_difference_quotient}
		\frac{f(x^k)-f(\bar x)}{\nnorm{x^k-\bar x}}
		+
		\frac12\frac{\nnorm{\delta^k}\nnorm{\lambda^k}}{\nnorm{x^k-\bar x}}
		\leq 0,
		\qquad\forall k\in\N.
	\end{equation}
	As the first summand converges to $f'(\bar x)d$ while the second one is nonnegative,
	the sequence
	$\{\nnorm{\delta^k}\nnorm{\lambda^k}/\nnorm{x^k-\bar x}\}_{k=1}^\infty$ needs to be bounded,
	i.e., \eqref{eq:multipliers_not_arbitrarily_unbounded} holds.
	Suppose that $\{\delta^k/\nnorm{x^k-\bar x}\}_{k=1}^\infty$ does not converge to zero.
	Then, along a subsequence, it is bounded away from zero.
	Consequently, on the same subsequence, $\{\lambda^k\}_{k=1}^\infty$ has to remain bounded.
	Applying \cref{lem:taking_the_limit_in_approximate_conditions} to the corresponding subsequence
	of $\{ (x^k, \lambda^k, \delta^k, \varepsilon^k) \}_{k=1}^\infty$ then shows M-stationarity of $\bar x$,
	such that situation~\ref{item:locmin_akkt1} is at hand.
	
	Thus, we may assume $\delta^k/\nnorm{x^k-\bar x}\to 0$ in the remainder of the proof, 
	i.e., all convergences from \eqref{eq:sequences_dir_conv} hold.
	Furthermore, we may also assume $\nnorm{\lambda^k}\to\infty$, as otherwise 
	we end up with M-stationarity of $\bar x$ via 
	\cref{lem:taking_the_limit_in_approximate_conditions}  again.
	As $\delta^k/\nnorm{\delta^k}=\lambda^k/\nnorm{\lambda^k}$ holds for all $k\in\N$ 
	by \eqref{eq:settings_for_directional_stationarity}, 
	we also have \eqref{eq:perturbation_vs_multiplier}.
	
	Let us verify $d\in C(\bar x)$.
	To this end, we first make use of \eqref{eq:bounded_difference_quotient} again in order to find
	\[
		f'(\bar x)d=\lim\limits_{k\to\infty}\frac{f(x^k)-f(\bar x)}{\nnorm{x^k-\bar x}}\leq 0.
	\]
	Furthermore, for each $k\in\N$, we have
	\[
		F(\bar x)
		+
		\nnorm{x^k-\bar x}
			\left(
				\frac{F(x^k)-F(\bar x)}{\nnorm{x^k-\bar x}}-\frac{\delta^k}{\nnorm{x^k-\bar x}}
			\right)
		=
		F(x^k)-\delta^k \in \Gamma,
	\]
	and as $(F(x^k)-F(\bar x))/\nnorm{x^k-\bar x}\to F'(\bar x;d)$ holds
	while $\delta^k/\nnorm{x^k-\bar x}\to 0$ follows from \eqref{eq:sequences_dir_conv},
	the definition of the tangent cone yields $F'(\bar x;d)\in T_\Gamma(F(\bar x))$.
	Thus, we have proven that assertion \ref{item:locmin_akkt2} holds true.
\end{proof}

Thus, whenever M-stationarity fails at some $\bar x\in\R^n$, 
\cref{thm:locmin_dir_akkt}~\ref{item:locmin_akkt2} now provides some insight into what we 
can expect to find instead.
Additionally, let us note that it is generally also possible that assertions \ref{item:locmin_akkt1} 
and \ref{item:locmin_akkt2} of \cref{thm:locmin_dir_akkt} hold at the same time.
In particular, if MSCQ$(d)$ holds at $\bar x$ for each critical direction 
$d\in C(\bar x)\cap\mathbb S \neq \emptyset$, then we know from 
\cref{cor:MStat_via_MS,thm:locmin_dir_subreg} 
that $\bar{x}$ is M-stationary in each of these critical directions, such that
\cref{thm:locmin_dir_akkt}~\ref{item:locmin_akkt1} is clearly valid. However, even in 
this case, it is possible to fulfill
\cref{thm:locmin_dir_akkt}~\ref{item:locmin_akkt2} at the same time, as we have seen in 
\cref{ex:mscq_not_enough_for_boundedness}. 
Moreover, certain regularity conditions can be used to guarantee that 
\cref{thm:locmin_dir_akkt}~\ref{item:locmin_akkt2} cannot hold without 
\cref{thm:locmin_dir_akkt}~\ref{item:locmin_akkt1}, as we will see later.

Next, we show by an example how \cref{thm:locmin_dir_akkt} can be applied whenever
only one assertion holds.
\begin{example}\label{ex:how_to_apply_dir_AKKT}
	Consider the optimization problem
	\[
		\min\limits_x\quad x\quad\textup{s.t.}\quad F(x)\coloneqq x \in \Gamma \coloneqq \{ 0 \}.
	\]	
	Clearly, its global minimizer is the unique feasible point $\bar{x} \coloneqq 0$.
	One can easily compute $T^\textup{lin}_{F,\Gamma}(\bar x) = C(\bar{x}) = \{ 0 \}$, 
	such that $C(\bar{x})\cap\mathbb{S} = \emptyset$ follows. 
	Consequently, assertion \ref{item:locmin_akkt2} of \cref{thm:locmin_dir_akkt}
	is violated, and, hence, assertion \ref{item:locmin_akkt1} yields that $\bar{x}$
	must be M-stationary. 
	Indeed, the M-stationarity conditions \eqref{eq:Mst} at $\bar{x}$ demand
	$-1= \lambda $ for some $\lambda\in \R$ and, thus,
	can be fulfilled using the multiplier $\lambda = -1$.
	
	Returning to the optimization problem 
	\[
		\min\limits_x\quad x_1+x_2\quad\textup{s.t.}\quad 
		F(x)\coloneqq x_1^2+x_2^2\in\Gamma\coloneqq\R_-
	\]
	from \cref{ex:AKKT_in_critical_directions} with the global minimizer $\bar{x}\coloneqq (0,0)$,
	the corresponding M-stationarity conditions \eqref{eq:Mst} at $\bar{x}$ require
	$(-1,-1)= 2\lambda \cdot 0$ for some $\lambda\geq 0$ and cannot be fulfilled. 
	Thus, assertion \ref{item:locmin_akkt1} of \cref{thm:locmin_dir_akkt}
	is violated, and it follows that assertion \ref{item:locmin_akkt2} holds true.
	Indeed, a critical direction $d\in C(\bar x) \cap \mathbb S$ and an associated 
	AM-stationary sequence w.r.t.\ $\bar x$ in direction $d$
	fulfilling \ref{item:locmin_akkt2} of \cref{thm:locmin_dir_akkt} are provided in 
	\cref{ex:AKKT_in_critical_directions}.
\end{example}

\cref{thm:locmin_dir_akkt} and \cref{ex:how_to_apply_dir_AKKT} give rise to the following result.

\begin{corollary}\label{cor:trivial_critical_cone}
	Let $\bar x\in\R^n$ be feasible for \eqref{eq:nonsmooth_problem}
	such that $C(\bar x)=\{0\}$.
	Then $\bar x$ is a strict local minimizer of \eqref{eq:nonsmooth_problem}
	which is M-stationary.
\end{corollary}
\begin{proof}	
	Because of the relation $\widehat C(\bar x)\subset C(\bar x)$,
	the fact that $\bar x$ is a strict local minimizer follows
	from \cref{cor:primal_fo_conditions_suf}.
	As $\bar x$ is, particularly, a local minimizer of \eqref{eq:nonsmooth_problem}, 
	its M-stationarity follows from \cref{thm:locmin_dir_akkt}
	as $C(\bar x)\cap\mathbb S=\emptyset$.
\end{proof}

Let us note that, in \cref{thm:locmin_dir_akkt_II,thm:locmin_dir_akkt},
the constructed AM-stationary sequences in critical directions possess
certain additional properties. 
In this regard, \cref{def:approx_stat_sequence_dir} is a compromise.
There, on the one hand, 
we do not require $\nnorm{\delta^k}\nnorm{\lambda^k}/\nnorm{x^k-\bar x}\to 0$,
which is shown in \cref{thm:locmin_dir_akkt_II}.
On the other hand, the
conditions $\nnorm{\lambda^k}\to\infty$, $x^k\notin X$
for all $k\in\N$, and $\delta^k \neq 0$ for all $k\in\N$,
which occur in \cref{thm:locmin_dir_akkt}, 
are also excluded.

\subsection{Approximate constraint qualifications}\label{sec:approx_CQs}

Noting that the proofs of \cref{thm:locmin_dir_akkt_II,thm:locmin_dir_akkt}
are utilizing an external penalty approach, 
we claim that these results may possess applications in numerical optimization.
Typically, several optimization algorithms tend to produce accumulation points
which are AM-stationary, see, e.g., 
\cite[Section~5.2]{AndreaniHaeserSchuverdtSilva2012}, 
\cite[Theorem~4.3]{JiaKanzowMehlitzWachsmuth2023},
and
\cite[Section~5]{Ramos2021}.
\cref{thm:locmin_dir_akkt} hints that the characterization of these accumulation
points might be refined such that they are either M-stationary,
AM-stationary in a critical direction, or even both. 
Then, by employing a suitable qualification condition,
one may even guarantee (directional) M-stationarity 
of the accumulation point under consideration.
To this end, we introduce the following approximate (directional) 
constraint qualifications, which parallel those defined in 
\cite[Definition 5.1]{BenkoMehlitz2024b}.

\begin{definition}\label{def:dir_AM_reg}
	Let $\bar x\in\R^n$ be feasible for \eqref{eq:nonsmooth_problem},
	and fix $d\in T^\textup{lin}_{F,\Gamma}(\bar x)\cap\mathbb S$.
	\begin{enumerate}
		\item We say that the point $\bar x$ is approximately M-regular (AM-regular)
			whenever for each sequence 
			$\{(x^k,\lambda^k,\delta^k,\xi^k)\}_{k=1}^\infty\subseteq\R^{n+\ell+\ell+n}$ 
			and $\xi\in\R^n$ satisfying
			\begin{equation}\label{eq:dir_AM_reg_stat}
				\xi^k\in\partial\langle\lambda^k,F\rangle(x^k),
				\qquad
				\lambda^k\in N_\Gamma(F(x^k)-\delta^k)
			\end{equation}
			for all $k\in\N$ and the convergences
			\begin{equation}\label{eq:dir_AM_reg_conv}
				(x^k,\delta^k,\xi^k,\nnorm{\lambda^k})\to(\bar x,0,\xi,\infty),
			\end{equation}
			the relation $\xi\in\partial\langle \lambda,F\rangle(\bar x)$ 
			is valid for some $\lambda\in N_\Gamma(F(\bar x))$.
		\item 
			We say that the point $\bar x$ is AM-regular in direction $d$
			whenever for each sequence 
			$\{(x^k,\lambda^k,\delta^k,\xi^k)\}_{k=1}^\infty\subseteq\R^{n+\ell+\ell+n}$ 
			and $\xi\in\R^n$ satisfying
			\eqref{eq:dir_AM_reg_stat} and $x^k\neq\bar x$ for all $k\in\N$,
			the convergences \eqref{eq:dir_AM_reg_conv}, and \eqref{eq:approx_stat_dir},  
			the relation $\xi\in\partial\langle \lambda,F\rangle(\bar x)$ is valid for some 
			$\lambda\in N_\Gamma(F(\bar x))$.
		\item We say that the point $\bar x$ is strongly AM-regular in direction $d$
			whenever for each sequence 
			$\{(x^k,\lambda^k,\delta^k,\xi^k)\}_{k=1}^\infty\subseteq\R^{n+\ell+\ell+n}$ 
			and $\xi\in\R^n$ satisfying
			\eqref{eq:dir_AM_reg_stat} and $x^k\neq\bar x$ for all $k\in\N$,
			the convergences \eqref{eq:dir_AM_reg_conv}, and \eqref{eq:approx_stat_dir}, 
			the relation $\xi\in\partial\langle \lambda,F\rangle(\bar x;d)$ is valid for some 
			$\lambda\in N_\Gamma(F(\bar x);F'(\bar x;d))$.		 
	\end{enumerate}
\end{definition}

Let us note that postulating $\nnorm{\lambda^k}\to\infty$ in \cref{def:dir_AM_reg}
is not restrictive.
Indeed, if $\{\lambda^k\}_{k=1}^\infty$ possesses a bounded subsequence,
then, along yet another subsequence, the latter would converge to some $\lambda\in\R^\ell$,
and one could simply take the limit in \eqref{eq:dir_AM_reg_stat}
to obtain the desired relations 
via \cref{lem:robustness_of_subdifferential_scalarization_function} and the robustness
of the (directional) limiting normal cone.

\begin{remark}\label{rem:dir_CQs}
	In contrast to \cite[Definition~5.1]{BenkoMehlitz2024b},
	we do not claim $x^k\notin X$ and $\delta^k\neq 0$ for each $k\in\N$
	in the definition of (strong) directional AM-regularity,
	leading to slightly more restrictive conditions than
	those introduced in \cite{BenkoMehlitz2024b}.
	This adjustment became necessary as we are working with a less tight
	notion of directional AM-stationarity than the one that
	could have been distilled from \cref{thm:locmin_dir_akkt} (which would equal the one
	used in \cite[Corollary~4.5]{BenkoMehlitz2024b})
	in order to cover our findings from \cref{thm:locmin_dir_akkt_II} as well,
	see the remark below \cref{cor:trivial_critical_cone} again.
	However, let us note that these minor differences between
	\cref{def:dir_AM_reg} and \cite[Definition~5.1]{BenkoMehlitz2024b}
	are negligible in the sense that the results from \cite[Section~5]{BenkoMehlitz2024b}
	carry over to the concepts from \cref{def:dir_AM_reg} via straightforward adjustments.
\end{remark}

Some obvious relations between the constraint qualifications from \cref{def:dir_AM_reg} 
are listed in the following remark,
see \cite[Remark~5.1]{BenkoMehlitz2024b} as well.
\begin{remark}\label{rem:relations_dir_CQ}
	Let $\bar x\in\R^n$ be feasible for \eqref{eq:nonsmooth_problem}.
	Then the following assertions hold.
	\begin{enumerate}
		\item If $\bar x$ is AM-regular, 
			then $\bar x$ is AM-regular in direction $d$ for all 
			$d\in T^\textup{lin}_{F,\Gamma}(\bar x)\cap\mathbb S$.
		\item\label{item:relations_dir_CQ_b} Fix $d\in T^\textup{lin}_{F,\Gamma}(\bar x)\cap\mathbb S$.
			If $\bar x$ is strongly AM-regular in direction $d$,
			then $\bar x$ is AM-regular in direction $d$.
	\end{enumerate}
\end{remark}

In the following remark, we point out that AM-regularity and strong AM-regularity in 
\textit{all} unit directions from the linearization cone are independent conditions.

\begin{remark}\label{rem:relations_to_other_CQs}
	Given a set-valued mapping $\Phi\colon\R^n\tto\R^m$,
	the constraint system $0\in\Phi(x)$, considered in \cite{BenkoMehlitz2024b},
	can equivalently be written as $(x,0)\in\gph\Phi$.
	Thus, setting $F(x)\coloneqq(x,0)$ and $\Gamma\coloneqq\gph\Phi$,
	it is covered by the setting considered here.
	In particular, \cite[Examples~5.1,~5.2]{BenkoMehlitz2024b} show that,
	given some point $\bar x\in\R^n$ feasible for \eqref{eq:nonsmooth_problem},
	AM-regularity of $\bar x$ and strong AM-regularity of $\bar x$ in all directions
	from $T^\textup{lin}_{F,\Gamma}(\bar x)\cap\mathbb S$ are independent conditions. 
\end{remark}

Thus, even if strong AM-regularity of $\bar x\in\R^n$ holds in all directions from 
$T^\textup{lin}_{F,\Gamma}(\bar x)\cap\mathbb S$, the point $\bar{x}$ is not necessarily AM-regular.
On the opposite, AM-regularity of $\bar{x}$ does not imply strong AM-regularity of 
$\bar x$ in \textit{all} (but possibly some) directions 
from $T^\textup{lin}_{F,\Gamma}(\bar x)\cap\mathbb S$. Indeed, in \cite[Example~5.1]{BenkoMehlitz2024b},
where $\bar x\coloneqq 0$ is AM-regular but not strongly AM-regular in direction 
$d\coloneqq 1\in T^\textup{lin}_{F,\Gamma}(\bar x)\cap\mathbb S$, one can verify that $\bar x$ is at least 
strongly AM-regular in direction $d\coloneqq -1 \in T^\textup{lin}_{F,\Gamma}(\bar x)\cap\mathbb S$.

In what follows, we prove that the latter result can further be strengthened.
Indeed, the following example shows that AM-regularity of $\bar{x}\in\R^n$ generally
does not imply strong AM-regularity of $\bar x$ in \textit{any} direction 
from $T^\textup{lin}_{F,\Gamma}(\bar x)\cap\mathbb S$.

\begin{example}\label{ex:AM_reg_vs_strong_dir_AM_reg}
	Consider $F\colon\R\to\R^2$ given by $F(x)\coloneqq(x,0)$, $x\in\R$, 
	and $\Gamma\subseteq\R^2$ given by
	\[
		\Gamma \coloneqq \{y\in\R^2\,|\,y_2\geq y_1^2\}\cup(\{0\}\times\R_-).
	\]
	Let us consider the unique feasible point $\bar x\coloneqq 0$.
	As we have
	\begin{align*}
		T_\Gamma(F(\bar x))
		&=
		(\R\times\R_+)\cup(\{0\}\times\R_-),
		\\
		N_\Gamma(F(\bar x))
		&=
		(\R\times\{0\})\cup(\{0\}\times\R_-),
	\end{align*}
	we find $T_{F,\Gamma}^\textup{lin}(\bar x)=\R$
	and 
	\[
		F'(\bar x)^\top N_\Gamma(F(\bar x))
		=
		\R,
	\]
	so $\bar x$ is trivially AM-regular
	and AM-regular in directions $d=\pm 1\in T^\textup{lin}_{F,\Gamma}(\bar x)\cap\mathbb S$.
	
	Let us show that strong AM-regularity of $\bar x$ in directions $d=\pm 1$ fails.
	To this end, we first consider $d\coloneqq 1$
	and investigate the sequence 
	$\{(x^k,\lambda^k,\delta^k,\xi^k)\}_{k=1}^\infty\subseteq\R^{1+2+2+1}$
	given by
	\[
		x^k\coloneqq \frac1k,
		\quad
		\lambda^k\coloneqq(2,-k),
		\quad
		\delta^k\coloneqq\left(0,-\frac1{k^2}\right),
		\quad
		\xi^k\coloneqq 2
		,
		\qquad
		\forall k\in\N.
	\]
	Then the convergences 
	\eqref{eq:dir_AM_reg_conv} hold for $\xi\coloneqq 2$,
	while $x^k\neq\bar x$ for all $k\in\N$ and \eqref{eq:approx_stat_dir} are also satisfied. 
	For all $k\in\N$, we find $\partial \langle \lambda^k, F\rangle (x^k) = \{ 2 \}$ and
	\[
		N_\Gamma(F(x^k)-\delta^k) 
		= 
		N_\Gamma\left(\left(\frac1k,\frac{1}{k^2}\right)\right)
		=
		\left\{ \left(2\alpha, -k\alpha \right) \in\R^2 \mid \alpha \geq 0 \right\},
	\]
	such that \eqref{eq:dir_AM_reg_stat} is valid for all $k\in\N$ as well.
	Due to the relation $N_\Gamma(F(\bar x);F'(\bar x)d)=\{0\}\times\R_-$ and, thus,
	\[
		F'(\bar x)^\top N_\Gamma(F(\bar x);F'(\bar x)d)
		=
		\{0\},
	\]
	it is clear that $\xi=2$ is not contained in the latter set, such that 
	$\bar x$ cannot be strongly AM-regular in direction $d$.
	Similarly, using the sequence 
	$\{(x^k,\lambda^k,\delta^k,\xi^k)\}_{k=1}^\infty\subseteq\R^{1+2+2+1}$
	given by
	\[
		x^k\coloneqq -\frac1k,
		\quad
		\lambda^k\coloneqq(-2,-k),
		\quad
		\delta^k\coloneqq\left(0,-\frac1{k^2}\right),
		\quad
		\xi^k\coloneqq -2
		,
		\qquad
		\forall k\in\N
	\]
	and $\xi \coloneqq -2$,
	one can verify that $\bar x$ is not strongly AM-regular 
	in direction $d\coloneqq -1$.
\end{example}

Let us comment on further relations of the conditions from \cref{def:dir_AM_reg}
to popular qualification conditions.
Fix a feasible point $\bar x\in\R^n$ of \eqref{eq:nonsmooth_problem}.
Then validity of NNAMCQ at $\bar x$ yields validity of AM-regularity of $\bar x$,
see e.g.\ \cite[Corollary~3.13]{Mehlitz2020b}.
Furthermore, given
$d\in T^\textup{lin}_{F,\Gamma}(\bar x)\cap\mathbb S$,
\cref{lem:robustness_of_subdifferential_scalarization_function} 
can be used to show that validity of FOSCMS$(d)$ at $\bar x$ implies that
the latter is also strongly AM-regular in direction $d$.
Further relations to other qualification conditions can be obtained 
with similar arguments as in \cite[Section~5]{BenkoMehlitz2024b}
and illustrate that the concepts from \cref{def:dir_AM_reg}
are comparatively mild.
Particularly, it has been indicated in \cite{BenkoMehlitz2024b,Mehlitz2020b}
that these concepts are not related to (directional) MSCQ.

Next, we want to illustrate by the following two examples that AM-regularity 
and directional AM-regularity are independent of GGCQ.
Let us note that GGCQ has not been taken into account in \cite[Section~5]{BenkoMehlitz2024b}.

\begin{example}\label{ex:AMreg_but_not_GGCQ}
	We again consider the feasible set from \cref{ex:AM_reg_vs_strong_dir_AM_reg}, where we have seen
	that the unique feasible point $\bar x\coloneqq 0$ fulfills AM-regularity as well as
	AM-regularity in directions $d=\pm 1$.
	We further computed $T_{F,\Gamma}^\textup{lin}(\bar x)=\R$, and as $T_X(\bar x) = \{ 0 \}$ 
	trivially holds, we find $\widehat{N}_X(\bar x)=\R$ and 
	$T_{F,\Gamma}^\textup{lin}(\bar x)^\circ = \{ 0 \}$. Thus, GGCQ is violated at $\bar x$.
\end{example}

We emphasize that for standard nonlinear optimization problems,
AM-regularity reduces to the so-called cone-continuity property,
and the latter has been shown to imply GACQ in \cite[Theorem 4.4]{AndreaniMartinezRamosSilva2016}.
Thus, the situation in \cref{ex:AMreg_but_not_GGCQ} has to be fully attributed
to the variationally complex structure of $\Gamma$.

\begin{example}\label{ex:GGCQ_but_not_AMreg}
	Recall the optimization problem from \cref{ex:GGCQ_without_MSCQ}, for which we have shown
	that the global minimizer $\bar x \coloneqq (0,0)$ fulfills GGCQ.
	We further computed $T^\textup{lin}_{F,\Gamma}(\bar x) = \R_+ \times \R_+$.
	To show that AM-regularity of $\bar x$ in direction 
	$d\coloneqq (0,1) \in T^\textup{lin}_{F,\Gamma}(\bar x)\cap\mathbb S$ 
	is violated, we pick the sequence
	$\{ (x^k,\lambda^k,\delta^k,\xi^k) \}_{k=1}^\infty\subset\R^{2+3+3+2}$ with
	\[
		x^k \coloneqq \left(\frac{1}{k^2},\frac{1}{k}\right),
		\quad
		\lambda^k \coloneqq (0,0,k),
		\quad
		\delta^k \coloneqq \left(0,0,\frac{1}{k^3}\right),
		\quad
		\xi^k \coloneqq \left(1, \frac{1}{k}\right),
		\qquad
		\forall k\in\N,
	\]
	which fulfills $x^k\neq \bar x$ for all $k\in \N$, the convergences 
	\eqref{eq:dir_AM_reg_conv} for $\xi \coloneqq (1,0)$, and \eqref{eq:approx_stat_dir}.
	The computations
	\begin{align*}
		\partial\langle \lambda^k,F\rangle(x^k) 
		&= 
		\left\{
		k 
		\begin{pmatrix}
		\frac{1}{k} \\ \frac{1}{k^2}
		\end{pmatrix}
		\right\}
		= 
		\left\{
		\begin{pmatrix}
		1 \\ \frac{1}{k}
		\end{pmatrix}
		\right\},
		\\
		N_\Gamma(F(x^k)-\delta^k) 
		&= 
		N_\Gamma\left(\left(-\frac{1}{k^2},-\frac{1}{k},0\right)\right) 
		= 
		\{ 0 \} \times \{ 0 \} \times \R
	\end{align*}
	yield that \eqref{eq:dir_AM_reg_stat} is satisfied for all $k\in\N$ as well.
	Thus, validity of AM-regularity of $\bar x$ in direction $d$ requires that 
	$\xi = (1,0)$ is contained in
	\begin{align*}
		F'(\bar x)^\top N_\Gamma(F(\bar x))
		=
		\{ (-\lambda_1,-\lambda_2)\mid \lambda_1,\lambda_2 \in \R_+,\lambda_3 \in \R\}
		= 
		\R_-\times \R_-,
	\end{align*}
	which is not fulfilled. Hence, $\bar x$ is not AM-regular in direction $d$ and, consequently, 
	not AM-regular, see \cref{rem:relations_dir_CQ}.
\end{example}

Moreover, we can also show that strong directional AM-regularity and GGCQ are independent conditions.
The fact that GGCQ can hold while strong directional AM-regularity is violated is already clear
from \cref{ex:GGCQ_but_not_AMreg} and \cref{rem:relations_dir_CQ}\,\ref{item:relations_dir_CQ_b}.
The opposite direction follows from the upcoming example, which, together with \cref{ex:GGCQ_without_MSCQ},
also shows that MSCQ$(d)$ and GGCQ are independent conditions.

\begin{example}\label{ex:dirMSCQ_without_GGCQ}
	Consider the constraint system
	\[
		F(x)\coloneqq(\max(x_1^2,x_1)-x_2,\min(x_1^2,-x_1)+x_2)\in\Gamma\coloneqq\R_-\times\R_-
	\]
	and the feasible point $\bar x\coloneqq(0,0)$.
	We find $X=\conv\{(0,0),(1,1)\}$ as well as
	\[
		T_X(\bar x)=\cone\{(1,1)\},
		\quad
		T^{\textup{lin}}_{F,\Gamma}(\bar x)=\cone\{(1,1)\}\cup\cone\{(-1,0)\}.
	\]
	Hence, it is clear that GGCQ fails to hold at $\bar x$.
	
	Let us now consider direction $d\coloneqq\frac1{\sqrt 2}(1,1) \in 
	T^\textup{lin}_{F,\Gamma}(\bar x)\cap\mathbb S$.
	Concerning MSCQ$(d)$ at $\bar x$, fix $\varepsilon>0$ and $\delta>0$ small enough
	such that all $x\in \{ \bar x\} + \mathbb B_{\varepsilon,\delta}(d)$ fulfill $x\in\R_+\times\R_+$
	and $\nnorm{x} < 1$.
	Then we calculate
	\begin{align*}
		\dist(x,F^{-1}(\Gamma)) &= \dist(x,X) = \frac{1}{\sqrt{2}}|x_1-x_2|, 
		\\
		\dist (F(x), \Gamma) &= \dist((x_1-x_2,-x_1+x_2),\R^2_-) = |x_1-x_2|,
	\end{align*}
	such that MSCQ$(d)$ at $\bar x$ can be fulfilled.
	
	In order to verify strong AM-regularity in direction $d$, let us consider any sequence 
	$\{(x^k,\lambda^k,\delta^k,\xi^k)\}_{k=1}^\infty\subseteq\R^{2+2+2+2}$ 
	and $\xi\in\R^2$ satisfying
	\eqref{eq:dir_AM_reg_stat} and $x^k\neq\bar x$ for all $k\in\N$,
	the convergences \eqref{eq:dir_AM_reg_conv}, and \eqref{eq:approx_stat_dir}.
	Then \eqref{eq:sequences_dir_conv} particularly implies $x_1^k>0$ and $x_2^k>0$ for 
	$k\in\N$ sufficiently large. 
	Thus, on the one hand, for all sufficiently large $k\in\N$, \cref{lem:sum_rule} yields
	\begin{align*}
		\partial\langle \lambda^k, F\rangle(x^k) 
		&= 
		|\lambda_1^k|\partial(\sgn(\lambda_1^k)F_1)(x^k) 
			+ |\lambda_2^k|\partial(\sgn(\lambda_2^k)F_2)(x^k) 
		\\
		&=
		\lambda_1^k 	\left\{\begin{pmatrix} 1 \\ -1 \end{pmatrix}\right\}	
			+ \lambda_2^k \left\{\begin{pmatrix} -1 \\ 1 \end{pmatrix}\right\}
		\subset
		\spa\left\{\begin{pmatrix}1\\-1\end{pmatrix}\right\},
	\end{align*}
	so that $\xi = \alpha(1,-1)$ is valid for some $\alpha\in\R$.
	On the other hand, 
	$N_\Gamma(F(\bar x); F'(\bar x;d)) = N_\Gamma(F(\bar x);0) = N_\Gamma(F(\bar x)) = \R_+ \times \R_+$
	implies
	\[
		\{\xi\in\partial\langle\lambda,F\rangle(\bar x;d)\,|\,\lambda\in N_\Gamma(F(\bar x);F'(\bar x;d))\}
		=
		\spa\left\{\begin{pmatrix}-1\\1\end{pmatrix}\right\},
	\]
	and this shows strong AM-regularity of $\bar x$ in direction $d$.
\end{example}

We now want to use the approximate (directional) constraint qualifications conditions from \cref{def:dir_AM_reg} 
to link directional AM-stationarity and (directional) M-stationarity, 
even in the absence of MSCQ.
For this purpose, consider a feasible point $\bar x\in\R^n$ of \eqref{eq:nonsmooth_problem} and some direction
$d\in T^\textup{lin}_{F,\Gamma}(\bar x)\cap\mathbb S$ such that $\bar x$ is AM-stationary in direction $d$.
Assuming that $\bar x$ is AM-regular (in direction $d$),
it is obvious by \cref{def:dir_AM_reg}
that $\bar x$ is M-stationary.
Similarly, whenever $\bar x$ is strongly AM-regular in direction $d$,
$\bar x$ is M-stationary in direction $d$.
That is why we obtain the following corollary 
of \cref{thm:locmin_dir_akkt_II,thm:locmin_dir_akkt} as well as \cref{cor:dir_asymp_stat_under_GACQ},
see \cref{lem:approx_MSt_NC} as well.

\begin{corollary}\label{cor:Mstat_via_dir_AM_reg}
	Let $\bar x\in\R^n$ be a local minimizer of \eqref{eq:nonsmooth_problem}.
	Then the following assertions hold.
	\begin{enumerate}
		\item\label{item:Mstat_via_AM_reg} 
			If $\bar x$ is AM-regular, then $\bar x$ is M-stationary.
		\item Fix $d\in\widehat C(\bar x)\cap\mathbb S$. If $\bar x$
			is AM-regular in direction $d$ (strongly AM-regular in direction $d$),
			then $\bar x$ is M-stationary (in direction $d$).
		\item Fix $d\in C(\bar x)\cap\mathbb S$, and let GACQ hold at $\bar x$.
			If $\bar x$
			is AM-regular in direction $d$ (strongly AM-regular in direction $d$),
			then $\bar x$ is M-stationary (in direction $d$).
		\item\label{item:Mstat_via_dir_AM_reg_all_directions} 	
			If $\bar x$ is AM-regular in direction $d$ for all $d\in C(\bar x)\cap\mathbb S$,
			then $\bar x$ is M-stationary.
	\end{enumerate}
\end{corollary}

Let us note that 
\cref{cor:Mstat_via_dir_AM_reg}\,\ref{item:Mstat_via_AM_reg}
parallels \cite[Theorem~3.9]{Mehlitz2020b} while
\cref{cor:Mstat_via_dir_AM_reg}\,\ref{item:Mstat_via_dir_AM_reg_all_directions}
parallels \cite[Corollary~5.1]{BenkoMehlitz2024b}. 
Moreover, under each of the four conditions considered in \cref{cor:Mstat_via_dir_AM_reg},
a refinement of \cref{thm:locmin_dir_akkt} is possible. Indeed, as 
\cref{thm:locmin_dir_akkt}~\ref{item:locmin_akkt1} holds in all four settings, the 
possibility that \cref{thm:locmin_dir_akkt}~\ref{item:locmin_akkt2} 
holds without \cref{thm:locmin_dir_akkt}~\ref{item:locmin_akkt1} is ruled out
whenever any of the four conditions is valid.
Let us, however, underline that none of the four conditions considered in
\cref{cor:Mstat_via_dir_AM_reg} simply rules out \cref{thm:locmin_dir_akkt}~\ref{item:locmin_akkt2}.

\section{Qualification conditions for problems with orthodisjunctive constraints}\label{sec:subMFC}

In this section,
we are going to introduce qualification conditions for \eqref{eq:nonsmooth_problem}
that are based on AM-stationary points.
In contrast to the conditions from \cref{def:dir_AM_reg},
which require control of all AM-stationary
sequences, the conditions we are studying here are based on \textit{one} particular
AM-stationary sequence and, hence, much easier to check.
Let us emphasize that these conditions are, thus, especially beneficial if one particular sequence is
already provided, e.g., due to the execution of a numerical algorithm. Indeed, many
algorithms for constrained optimization problems are known to produce an AM-stationary
sequence as already mentioned at the beginning of \cref{sec:approx_CQs}.

The underlying idea for constructing the qualification conditions is taken from \cite{KaemingFischerZemkoho2025,KaemingMehlitz2025}
and applies to problems whose constraints, in a certain sense, allow for a componentwise decomposition.
Thus, to start, we specify a suitable subclass of nonsmooth problems
where the new qualification conditions can be employed successfully in \cref{sec:subMFC_setting}.
Thereafter, \cref{sec:subMFC_concept} is dedicated to the actual introduction of the qualification conditions
and presents related consequences as well as a brief comparison
to the constraint qualifications discussed earlier in this paper.

\subsection{Model problem and example classes}\label{sec:subMFC_setting}

We build our considerations on model problem \eqref{eq:nonsmooth_problem}
with $f\colon\R^n\to\R$ being continuously differentiable and 
$F\colon\R^n\to\R^\ell$ being directionally differentiable and locally Lipschitz continuous.
Here, we investigate the special case where $\Gamma$ in \eqref{eq:nonsmooth_problem}
is the union of finitely many convex polyhedral sets $\Gamma_1,\dotsc,\Gamma_t\subset\R^\ell$,
with each $\Gamma_j$, $j\in\{1,\dotsc,t\}$,
being a product of closed intervals. This yields
the so-called orthodisjunctive problem
\begin{equation}\label{eq:orthodisjunctive_problem}\tag{ODP}
	\begin{aligned}
		&\min\limits_x&	&f(x)&			
		\\
		&\text{\,s.t.}&	&F(x)\in \Gamma
		\coloneqq \bigcup\limits_{j=1}^t\Gamma_j	\quad\text{with}\quad
		\Gamma_j \coloneqq  \prod\limits_{i=1}^\ell [a_i^j, b_i^j],
		\qquad 
		\forall j\in\{1,\dotsc,t\}&	
	\end{aligned}
\end{equation}
with $-\infty\leq a_i^j\leq b_i^j\leq\infty$ for all $i\in\{1,\ldots,\ell\}$ and $j\in\{1,\ldots,t\}$.
Here, we used
\[
	[-\infty,r]\coloneqq (-\infty,r],
	\qquad
	[r,\infty]\coloneqq [r,\infty),
	\qquad
	[-\infty,\infty]\coloneqq \R
\]
for $r\in \R$.

For $x\in\R^n$ and $\delta\in\R^\ell$, if $F(x)-\delta\in \Gamma$ 
for $\Gamma$ as in \eqref{eq:orthodisjunctive_problem}, 
we define the set
\[
	J(x,\delta) 
	\coloneqq
	\{ j \in \{1,\dotsc, t\} \mid F(x) - \delta \in \Gamma_j \}
\]
to identify all active components of $\Gamma$ and the set 
\[
	I(x,\delta) 
	\coloneqq 
	\{ i\in \{1,\dots, \ell \} \mid \exists j\in J(x,\delta)\colon\,
		F_i(x)-\delta_i \in \{a_i^j,b_i^j \} \}
\]
to capture the components of $F(x)-\delta$ which are an interval endpoint for at least one
active component of $\Gamma$. 
Whenever $\bar x\in\R^n$ is feasible for \eqref{eq:orthodisjunctive_problem},
we make use of $I(\bar x)\coloneqq I(\bar x;0)$ for brevity of notation.

In this section,
we will make use of the subsequently stated assumption,
which is assumed to hold throughout.

\begin{assumption}\label{ass:subMFC}
	For each $x,d\in\R^n$ and $\lambda\in\R^\ell$, we have
	\[
		\partial\langle\lambda,F\rangle(x;d)
		=
		\mathsmaller\sum\nolimits_{i=1}^\ell|\lambda_i|\partial(\sgn(\lambda_i)\,F_i)(x;d).
	\]
\end{assumption}
Thus, we assume that the scaled sum rule from \cref{lem:sum_rule} holds with equality.
Picking $d\coloneqq0$ in \cref{ass:subMFC}, we particularly have
\[
	\partial\langle\lambda,F\rangle(x)
	=
	\mathsmaller\sum\nolimits_{i=1}^\ell|\lambda_i|\partial(\sgn(\lambda_i)\,F_i)(x)
\]
for all $x\in\R^n$ and $\lambda\in\R^\ell$.

\begin{remark}\label{rem:ass_subMFC}
	We note that \cref{ass:subMFC} is valid if, for each $x\in\R^n$,
	at most one component of $F$ is nonsmooth at $x$,
	see \cref{lem:sum_rule}.
\end{remark}

Let us comment on two practically relevant example classes of optimization problems 
that satisfy \cref{ass:subMFC}.
First, whenever $F$ is continuously differentiable, 
\eqref{eq:orthodisjunctive_problem} corresponds to the orthodisjunctive optimization
problem discussed in \cite[Section~4]{KaemingMehlitz2025} and covers, among others,
optimization problems with cardinality, complementarity, switching, and vanishing constraints.
Second, following \cite[Section~5]{KaemingFischerZemkoho2025}, 
\cref{ass:subMFC} is likely to hold
for the so-called value function reformulation in bilevel optimization.
Let us elaborate on the latter.
Given the parametric optimization problem
\begin{equation}\label{eq:lower_level}\tag{POP$(w)$}	
	\min\limits_{z}\quad p(w,z)\quad \textup{s.t.}\quad q(w,z)\leq 0
\end{equation}
with parameter $w\in\R^{n_1}$, where $p\colon\R^{n_1+n_2}\to\R$ and $q\colon\R^{n_1+n_2}\to\R^s$
are continuously differentiable,
let $\Psi\colon\R^{n_1}\tto\R^{n_2}$ be the associated (global) solution mapping, i.e.,
\[
	\Psi(w) \coloneqq \argmin_z \{ p(w,z) \mid q(w,z)\leq 0 \},\qquad\forall w\in\R^{n_1}.
\] 
Then we are interested in the hierarchical optimization problem
\[
	\min\limits_{w,z}\quad P(w,z)\quad\textup{s.t.}\quad Q(w,z)\leq 0,\quad z\in\Psi(w),
\]
where $P\colon\R^{n_1+n_2}\to\R$ and $Q\colon\R^{n_1+n_2}\to\R^r$ are continuously
differentiable.
Using the so-called optimal value function of \eqref{eq:lower_level}, which is 
the extended real-valued function $\varphi\colon\R^{n_1}\to\R\cup\{-\infty,\infty\}$ 
given by
\[
	\varphi(w)
	\coloneqq
	\inf\limits_z\{p(w,z)\mid q(w,z)\leq 0\},\qquad\forall w\in\R^{n_1},
\]
where $\inf\emptyset\coloneqq\infty$ is used, the hierarchical optimization problem
is equivalent to the single-level optimization problem
\begin{equation}\label{eq:VFref}\tag{VF$_\textup{ref}$}
	\min\limits_{w,z}\quad P(w,z)\quad\textup{s.t.}\quad Q(w,z)\leq 0,\quad q(w,z)\leq 0,\quad p(w,z)-\varphi(w)\leq 0,
\end{equation}
referred to as the value function reformulation.
Clearly, $\varphi$ is very likely to be a nonsmooth function.
However, whenever $\varphi$ is directionally differentiable and locally Lipschitz continuous,
\eqref{eq:VFref} is covered by model \eqref{eq:orthodisjunctive_problem},
where $n\coloneqq n_1+n_2$, $\ell\coloneqq r+s+1$, $t\coloneqq 1$, and $\Gamma\coloneqq\R^\ell_-$,
and \cref{ass:subMFC} holds as at most the constraint including $\varphi$ is nonsmooth.
Note that moderate assumptions on the data functions $p$ and $q$ ensure that $\varphi$
is directionally differentiable, see, e.g., \cite[Section~4.3.2]{BonnansShapiro2000},
and locally Lipschitz continuous, see, e.g., 
\cite[Section~4.1.1]{MehlitzMinchenko2021} and \cite[Proposition~1.2]{YeZhu2010}.
Let us further note that directional stationarity conditions and qualification conditions have been shown to have 
remarkable potential in bilevel optimization, see \cite{BaiYe2022}.

To close this subsection, let us foreshadow why \cref{ass:subMFC} is used in this section.
Due to the componentwise structure of the constraints in \eqref{eq:orthodisjunctive_problem}, 
it might be reasonable to make \eqref{eq:dir_Mst_x} somewhat more accessible via
\begin{align*}
	-\nabla f(\bar x)
	&\in 
	\partial\langle\lambda,F\rangle(\bar x;d)
	\subset
	\mathsmaller\sum\nolimits_{i=1}^\ell|\lambda_i|\,\partial\bigl(\sgn(\lambda_i)\,F_i\bigr)(\bar x;d),
\end{align*}
which holds as equality in the presence of \cref{ass:subMFC}.
Similarly, let us note that validity of FOSCMS$(d)$ at $\bar x$ is implied by
\begin{equation}\label{eq:sFOSCMS}
	0\in\mathsmaller\sum\nolimits_{i=1}^\ell|\lambda_i|\,\partial\bigl(\sgn(\lambda_i)\,F_i\bigr)(\bar x;d),\,
	\lambda\in N_\Gamma(F(\bar x);F'(\bar x;d))
	\quad\Longrightarrow\quad
	\lambda=0,
\end{equation}
and both conditions are equivalent 
in the presence of \cref{ass:subMFC}.
The new qualification conditions we are going to introduce in \cref{sec:subMFC_concept}
build upon the componentwise structure of the constraints in \eqref{eq:orthodisjunctive_problem}.
Hence, rather than giving a lengthy introduction of 
(approximate and/or directional) M-stationarity conditions that employ, 
instead of the left-hand side, the right-hand side of the sum rule from \cref{lem:sum_rule},
which, of course, would also be possible,
we simply use \cref{ass:subMFC} as it ensures that these concepts coincide with those
discussed before.
Working out the details of this chapter for fully nonsmooth functions $F$ is, thus,
left to the interested reader but seems to be a rather canonical affair.

\subsection{Qualification conditions based on approximately M-stationary points}\label{sec:subMFC_concept}

Given an AM-stationary point $\bar x\in\R^n$ of \eqref{eq:orthodisjunctive_problem},
some critical direction $d\in C(\bar x)\cap\mathbb S$,
and some AM-stationary sequence 
$\{(x^k,\lambda^k,\delta^k,\varepsilon^k)\}_{k=1}^\infty\subset\R^{n+\ell+\ell+n}$ 
w.r.t.\ $\bar x$ in direction $d$,
we are interested in qualification conditions ensuring (directional) M-stationary of $\bar x$
while depending merely on the particular sequence
$\{(x^k,\lambda^k,\delta^k,\varepsilon^k)\}_{k=1}^\infty$.
This is very much in contrast to the regularity conditions from \cref{def:dir_AM_reg},
which depend on \emph{all} AM-stationary sequences w.r.t.\ $\bar x$ in direction $d$.
This idea, in the non-directional situation, 
has already been considered in the recent papers \cite{KaemingFischerZemkoho2025,KaemingMehlitz2025}
and will be refined here.

Before introducing our qualification conditions, let us note that we will utilize the following 
observation without further mention for brevity of presentation in this subsection.
\begin{remark}\label{rem:constant_sets}
	Whenever we consider a sequence 
	$\{(x^k,\lambda^k,\delta^k,\varepsilon^k)\}_{k=1}^\infty \subset \R^{n+\ell+\ell+n}$ 
	that is AM-stationary w.r.t.\ $\bar{x}$ 
	(in direction $d\in T_{F,\Gamma}^{\textup{lin}}(\bar{x})$) 
	in the following, 
	we may assume without loss of generality that the sets $J(x^k,\delta^k)$ are the same for all
	$k\in\N$, and, likewise, the sets $I(x^k,\delta^k)$ are the same for all $k\in\N$, 
	which is reasonable as it is always possible to consider a suitable subsequence. 
	Similarly, we may also assume that each component of $\lambda^k$ possesses a 
	constant sign for all $k\in\N$.
\end{remark}

We now generalize the so-called ODP Subset Mangasarian--Fromovitz condition from \cite{KaemingMehlitz2025},
that has been introduced for \eqref{eq:orthodisjunctive_problem} in the special case of continuously differentiable 
constraint functions, to \eqref{eq:orthodisjunctive_problem} with nonsmooth constraint functions as considered here.

\begin{definition}\label{def:normal_subMFC}
	Let $\bar{x}\in\R^n$ be an AM-stationary point of 
	\eqref{eq:orthodisjunctive_problem}. We say that the 
	\textit{ODP Subset Mangasarian-Fromovitz Condition (ODP-subMFC)} holds at 
	$\bar{x}$ if there exist an
	index set $I \subseteq I(\bar{x})$ and a sequence 
	$\{(x^k,\lambda^k,\delta^k,\varepsilon^k)\}_{k=1}^\infty
	\subset \R^{n+\ell+\ell+n}$
	such that the following conditions 
	are satisfied.
	\begin{enumerate}[label=(\roman*)]
		\item \label{subMFCI} 
			Either $I=\emptyset$, or it holds for all $u\in\R^{\ell} \setminus \{ 0 \}$ with 
			$u\geq 0$ and $u_{\{1,\dotsc,\ell\}\setminus I}=0$ that
			\begin{equation}\label{eq:normal_subMFC_cond} 
				0 \displaystyle \notin \mathsmaller\sum\nolimits_{i\in I}  
					u_i \partial\bigl(\sgn(\lambda^k_i)\, F_i\bigr)(\bar x),
				\qquad
				\forall k\in\N.
			\end{equation}
		\item \label{subMFCII} 
			The sequence $\{(x^k,\lambda^k,\delta^k,\varepsilon^k)\}_{k=1}^\infty$ 
			is AM-stationary w.r.t.\ $\bar{x}$, and $I=I(x^k,\delta^k)$ is
			valid for all $k\in \N$.
	\end{enumerate}
\end{definition}

It is easy to see that \cref{def:normal_subMFC} recovers 
\cite[Definition 4.3]{KaemingMehlitz2025} if all constraint functions are 
continuously differentiable. Indeed, in this case, we have 
$\partial(\sgn(\lambda_i^k)F_i)(\bar{x}) = \{\sgn(\lambda_i^k)\nabla F_i(\bar{x})\}$ 
in \eqref{eq:normal_subMFC_cond}. Additionally, in 
\cite[Section 4.2]{KaemingMehlitz2025}, it was shown that for smooth 
inequality-constrained problems, ODP-subMFC recovers a similar condition introduced
in \cite{KaemingFischerZemkoho2025}.

Noting that ODP-subMFC implicitly depends on the objective function of 
\eqref{eq:orthodisjunctive_problem} via the claimed existence of an
AM-stationary sequence, it is not a constraint qualification in the narrower sense
but merely a qualification condition.
Moreover, note that condition~\ref{subMFCI} of ODP-subMFC, even though depending
on $k\in\N$, boils down to the same relation for all $k\in\N$ due to \cref{rem:constant_sets},
and, thus, constitutes a single point-based condition at $\bar x$.
Further, condition~\ref{subMFCI} of ODP-subMFC
underlines its close relationship to NNAMCQ, 
which we will carve out in detail later in \cref{prop:subMFC_vs_NNAMCQ}.
However, we already want to mention at this point that ODP-subMFC offers some flexibility
in ruling out crucial constraints by appropriately choosing the index set $I$,
which is clearly not the case for NNAMCQ.

Motivated by \cref{thm:locmin_dir_akkt_II,thm:locmin_dir_akkt}
and to refine ODP-subMFC even further, we introduce the following directional version.

\begin{definition}\label{def:dir_subMFC}
	Let $\bar{x}\in\R^n$ be an AM-stationary point of
	\eqref{eq:orthodisjunctive_problem}, and let
	$d\in C(\bar{x}) \cap \mathbb{S}$ be chosen arbitrarily.
	We say that the
	\textit{ODP Subset Mangasarian-Fromovitz Condition in direction $d$ (ODP-subMFC$(d)$)}
	holds at $\bar{x}$ if there exist an
	index set $I \subseteq I(\bar{x})$ and a sequence 
	$\{(x^k,\lambda^k,\delta^k,\varepsilon^k)\}_{k=1}^\infty
	\subset \R^{n+\ell+\ell+n}$
	such that the following conditions 
	are satisfied.
	\begin{enumerate}[label=(\roman*)]
		\item \label{dir_subMFCI} 
			Either $I=\emptyset$, or it holds for all $u\in\R^{\ell} \setminus \{ 0 \}$ with 
			$u\geq 0$ and $u_{\{1,\dotsc,\ell\}\setminus I}=0$ that
			\begin{equation}\label{eq:dir_subMFC_cond} 
				0 \displaystyle \notin \mathsmaller\sum\nolimits_{i\in I}  
				u_i \partial\bigl(\sgn(\lambda^k_i)\, F_i\bigr)(\bar x;d),
				\qquad
				\forall k\in\N.
			\end{equation}
		\item \label{dir_subMFCII} 
			The sequence $\{(x^k,\lambda^k,\delta^k,\varepsilon^k)\}_{k=1}^\infty$ 
			is AM-stationary w.r.t.\ $\bar x$ in direction $d$,
			and $I=I(x^k,\delta^k)$ is valid for all $k\in\N$.
	\end{enumerate}
\end{definition}

Similarly as above, we observe that ODP-subMFC$(d)$ is not a constraint qualification
but merely a qualification condition. Condition~\ref{dir_subMFCI} of ODP-subMFC$(d)$
reads the same for all $k\in\N$ by \cref{rem:constant_sets} and
indicates that ODP-subMFC$(d)$ is related to FOSCMS$(d)$. 

Given an AM-stationary point $\bar x \in \R^n$ and a direction
$d\in C(\bar x) \cap \mathbb S$, let us note 
that~\ref{dir_subMFCII} of ODP-subMFC$(d)$ is more restrictive
than~\ref{subMFCII} of ODP-subMFC.
However, keeping \cref{thm:locmin_dir_akkt_II,thm:locmin_dir_akkt} in mind,
local minimizers are likely to come along with AM-stationary
sequences in critical directions, so~\ref{dir_subMFCII}
of ODP-subMFC$(d)$ is somewhat natural.
Given an AM-stationary sequence w.r.t.\ $\bar x$ in direction $d$,~\ref{dir_subMFCI} 
of ODP-subMFC$(d)$ is generally less restrictive than~\ref{subMFCI} of ODP-subMFC 
as the directional limiting subdifferentials
appearing in \eqref{eq:dir_subMFC_cond} can be strict subsets of the
limiting subdifferentials occurring in \eqref{eq:normal_subMFC_cond}
whenever the constraint function is nonsmooth at $\bar x$.

The following result, which generalizes \cite[Theorem 3.9]{KaemingFischerZemkoho2025} and
\cite[Theorem 4.5(a)]{KaemingMehlitz2025}, justifies our interest in ODP-subMFC and its directional version
as qualification conditions. 
\begin{theorem}\label{thm:ODPsubMFC}
	Let $\bar{x}\in\R^n$ be an AM-stationary point of 
	\eqref{eq:orthodisjunctive_problem}.
	\begin{enumerate}
		\item\label{item:ODPsubMFC_implies_Mstat}		
			If ODP-subMFC is satisfied at $\bar{x}$, then $\bar{x}$ is M-stationary.
		\item\label{item:dirODPsubMFC_implies_dirMstat}		
			Fix $d\in C(\bar x)\cap\mathbb S$.
			If ODP-subMFC$(d)$ is satisfied at $\bar{x}$, 
			then $\bar{x}$ is M-stationary in direction $d$.
	\end{enumerate}
\end{theorem}
\begin{proof}
	Here, we only prove the slightly more delicate assertion~\ref{item:dirODPsubMFC_implies_dirMstat}.
	The verification of statement~\ref{item:ODPsubMFC_implies_Mstat} is similar and can be
	distilled from the one provided for~\ref{item:dirODPsubMFC_implies_dirMstat} 
	by omitting directional information.
	Besides, both proofs parallel the one of \cite[Theorem 4.5(a)]{KaemingMehlitz2025}.
	
	If ODP-subMFC$(d)$ is satisfied at $\bar x$ for some $d\in C(\bar x) \cap \mathbb{S}$, it follows
	from \ref{dir_subMFCII} of ODP-subMFC$(d)$ at $\bar x$ 
	that there exists a sequence $\{(x^k,\lambda^k,\delta^k,\varepsilon^k)\}_{k=1}^\infty\subset\R^{n+\ell+\ell+n}$ 
	that is AM-stationary w.r.t.\ $\bar x$ in direction $d$, i.e., it 
	fulfills \eqref{eq:sequences_conv}, \eqref{eq:approx_stat_dir}, and, for all $k\in\N$,
	$\lambda^k\in N_\Gamma(F(x^k)-\delta^k)$, $x^k\neq \bar x$, and
	\begin{equation}\label{eq:proof_stationarity_cond}
		\varepsilon^k - \nabla f(x^k) 
		\in 
		\partial \langle \lambda^k, F \rangle (x^k)
		= 
		\mathsmaller\sum\nolimits_{i=1}^\ell|\lambda_i^k|\partial(\sgn(\lambda_i^k)\,F_i)(x^k),
	\end{equation}
	where we also recall \cref{ass:subMFC}. 
	Additionally, \ref{dir_subMFCII} of ODP-subMFC$(d)$ at $\bar x$ implies that we can find an index set 
	$I\subset I(\bar x)$ such that $I=I(x^k,\delta^k)$ is valid for all $k\in\N$. 
	From \cite[Remark 4.1(b)]{KaemingMehlitz2025} it follows for the special structure of 
	$\Gamma$ in \eqref{eq:orthodisjunctive_problem} that 
	$\lambda^k \in N_\Gamma(F(x^k)-\delta^k)$ implies 
	$\supp(\lambda^k) \subset I(x^k,\delta^k) = I$ for all $k\in\N$,
	such that \eqref{eq:proof_stationarity_cond} can be simplified to
	\begin{equation}\label{eq:proof_stationarity_cond2}
		\varepsilon^k - \nabla f(x^k) 
		\in
		\mathsmaller\sum\nolimits_{i\in I} |\lambda_i^k|\partial(\sgn(\lambda_i^k)\,F_i)(x^k)
		=
		\partial \langle \lambda_I^k, F_I \rangle (x^k),
	\end{equation}
	using \cref{ass:subMFC} again.
	If $I=\emptyset$, then $\lambda^k=0$ holds for all $k\in\N$,
	and we can take the limit $k\to\infty$ in \eqref{eq:proof_stationarity_cond}
	to find
	$-\nabla f(\bar x)=0$ by continuity of $\nabla f$ and \eqref{eq:sequences_conv}. 
	As $\bar x$ is feasible for 
	\eqref{eq:orthodisjunctive_problem}, this directly yields M-stationarity of $\bar x$ in
	direction $d$ with multiplier $\lambda \coloneqq 0$.
	
	If $I\neq\emptyset$, we prove that the sequence $\{ \lambda^k_I \}_{k=1}^\infty$ is bounded.
	Thus, assume on the contrary that this sequence is unbounded.
	Then $\nnorm{ \lambda^k_{I} } \to \infty$
	is valid along a subsequence (without relabeling), and, thus, without loss of generality, 
	we can find some $\bar\lambda\in\R^\ell$ with $\bar\lambda_I\neq 0$ and 
	$\bar \lambda_{\{1,\dotsc,\ell\}\setminus I}=0$ such that
	the convergence 
	\begin{equation*}
		\frac{ \lambda^k_{I} }
			{\nnorm{ \lambda^k_{I} }} 
		\to 
		\bar{\lambda}_{I}
	\end{equation*} 
	holds.
	Dividing \eqref{eq:proof_stationarity_cond2} by 
	$\nnorm{ \lambda^k_{I} }$ 
	leads to
	\begin{equation*}
		\frac{\varepsilon^k - \nabla f(x^k) }{\nnorm{ \lambda^k_{I} }}	
		\in
		\frac{\partial \langle \lambda_I^k, F_I \rangle (x^k)}{\nnorm{ \lambda^k_{I} }} 
		=
		\partial \left\langle \frac{\lambda^k_I}{\nnorm{ \lambda^k_{I}}}, F_I \right\rangle (x^k),
	\end{equation*} 
	where the last equation follows from \eqref{eq:scalarization_rule} and the positive homogeneity 
	of the limiting coderivative stated in \cite[Proposition~8.37]{RockafellarWets1998}.
	Taking the limit $k\to\infty$ in the above inclusion while applying 
	\eqref{eq:sequences_conv}, \eqref{eq:sequences_dir_conv}, and
	\cref{lem:robustness_of_subdifferential_scalarization_function}
	yields
	\begin{equation*}
		0 
		\in 
		\partial \langle \bar{\lambda}_{I}, F_I \rangle (\bar x;d)
		=
		\mathsmaller\sum\nolimits_{i\in I} |\bar \lambda_i| \partial(\sgn(\bar \lambda_i)\,F_i)(\bar x;d),
	\end{equation*}  
	where we again recall \cref{ass:subMFC}. 
	Using the set $I_\pm \coloneqq\{i\in I\,|\, \bar \lambda_i \neq 0\}$
	and $\sgn(\lambda_i^k)=\sgn(\bar \lambda_i)$ for $i\in I_\pm$, $k\in\N$ by construction 
	and \cref{rem:constant_sets},
	the above inclusion is equivalent to
	\begin{align*}
		0 
		\in 
		\mathsmaller\sum\nolimits_{i\in I_\pm} |\bar \lambda_i| \partial(\sgn(\bar \lambda_i)\,F_i)(\bar x;d)
		&=
		\mathsmaller\sum\nolimits_{i\in I_\pm} |\bar \lambda_i| \partial(\sgn(\lambda_i^k)\,F_i)(\bar x;d)
		\\
		&=
		\mathsmaller\sum\nolimits_{i\in I} |\bar \lambda_i| \partial(\sgn(\lambda_i^k)\,F_i)(\bar x;d)
	\end{align*}
	for all $k\in\N$, which violates \ref{dir_subMFCI} of ODP-subMFC$(d)$ at $\bar x$
	as $|\bar \lambda|\geq 0$, $ |\bar \lambda|_{\{1,\dotsc,\ell\}\setminus I}=0$, 
	and $ |\bar \lambda|_I\neq 0$.
	Here, $|\bar\lambda|\in\R^\ell$ is the vector 
	with entries $|\bar\lambda_i|$, $i\in\{1,\ldots,\ell\}$.
	Consequently, the sequence 
	$\{ \lambda^k_{I} \}_{k=1}^\infty$ 
	must be bounded, and, hence, admits a subsequence (without relabeling) that satisfies
	$ \lambda^k_I \to \widetilde{\lambda}_{I} $ for
	some $\widetilde\lambda\in\R^\ell$ 
	with $\widetilde\lambda_{\{1,\dotsc,\ell\}\setminus I} = 0$. 
	Taking the limit $k\to\infty$ in \eqref{eq:proof_stationarity_cond2}, it follows from 
	\eqref{eq:sequences_conv}, \eqref{eq:sequences_dir_conv}, 
	and \cref{lem:robustness_of_subdifferential_scalarization_function} that
	\[
		-\nabla f(\bar x)
		\in
		\partial \langle \widetilde \lambda_I, F_I \rangle (\bar{x};d)
		=
		\partial \langle \widetilde \lambda, F \rangle (\bar{x};d)
		.
	\] 
	Finally, $\lambda^k\in N_\Gamma(F(x^k)-\delta^k)$ for all $k\in\N$,
	\eqref{eq:sequences_conv}, \eqref{eq:sequences_dir_conv}, \eqref{eq:suitable_repr_of_constraint_viol}, and 
	the robustness of the directional limiting normal cone
	yield $\widetilde\lambda \in N_\Gamma(F(\bar{x});F'(\bar{x};d))$.
	Hence, we have shown that $\bar x$ is M-stationary in direction $d$
	with multiplier $\widetilde \lambda$.
\end{proof}

The subsequent remark summarizes particular information on the multiplier sequence obtained
in the proof of \cref{thm:ODPsubMFC} and parallels
similar observations made in \cite[Remark 3.10]{KaemingFischerZemkoho2025}
and \cite[Remark 4.6]{KaemingMehlitz2025}.
\begin{remark}\label{rem:bounded_multipliers_under_ODPsubMFC}
	Suppose that ODP-subMFC or ODP-subMFC$(d)$ can be verified at $\bar x\in\R^n$ using the index 
	set $I\subset I(\bar x)$ and the sequence 
	$\{(x^k,\lambda^k,\delta^k,\varepsilon^k)\}_{k=1}^\infty\subset \R^{n+\ell+\ell+n}$.
	It follows from the proof of \cref{thm:ODPsubMFC} that this ensures boundedness 
	of $\{ \lambda^k\}_{k=1}^\infty$ as we have shown $\supp(\lambda^k) \subset I$ for all $k\in\N$
	and, whenever $I\neq \emptyset$, boundedness of $\{ \lambda_I^k\}_{k=1}^\infty$.
	Let us note that the final part of the proof of \cref{thm:ODPsubMFC}, thus,
	also follows from \cref{lem:taking_the_limit_in_directional_approximate_conditions}.
\end{remark}

Given $\bar x\in\R^n$, we are going to show next that
ODP-subMFC at $\bar x$ and ODP-subMFC$(d)$ at $\bar x$ for any $d\in C(\bar{x}) \cap \mathbb{S}$ 
are independent conditions.

Clearly, $C(\bar x) \neq \{0\}$ is necessary for ODP-subMFC$(d)$ to hold
at $\bar x$ for any $d\in C(\bar{x}) \cap \mathbb{S}$,
while this is not required for ODP-subMFC to be valid at $\bar x$.
Indeed, ODP-subMFC holds for the optimization problem 
\[
	\min\limits_x\quad -x \quad\textup{s.t.} \quad F(x) \coloneqq x \in \Gamma \coloneqq \R_-
\] 
at $\bar x\coloneqq 0$, but $C(\bar x)=\{0\}$ trivially implies that there
cannot be any $d\in C(\bar{x}) \cap \mathbb{S}$ to fulfill
ODP-subMFC$(d)$ at $\bar x$. 
However, even if $C(\bar{x}) \cap \mathbb{S} \neq \emptyset$
can be guaranteed, it is still possible that ODP-subMFC holds at $\bar x$ while 
ODP-subMFC$(d)$ is violated at $\bar x$ for all $d\in C(\bar{x}) \cap \mathbb{S}$,
as the subsequent example illustrates.

\begin{example}\label{ex:normal_but_not_dir_subMFC}
	Consider the optimization problem
	\[
		\begin{aligned}
		&\min\limits_x\quad& &\frac32 x_1+x_2+x_3&
		\\
		&\,\textup{s.t.} \quad& &F(x)\coloneqq (g(x_1)-x_2-x_3-1,x_1+|x_2|-x_2-1,-x_3) 
		\in 
		\Gamma \coloneqq \{ 0\} \times \{ 0\} \times \R_-,&
		\end{aligned}
	\]
	where $g\colon\R\to\R$ is the continuously differentiable function given by
	\[
		g(\tau)\coloneqq
		\begin{cases}
			\sqrt\tau		&	\tau\geq\frac14
			\\
			\tau+\frac14	&	\tau<\frac14
		\end{cases}
	\] 
	for $\tau\in\R$.
	This problem has the feasible set $X=\{(1,0,0)\}$, 
	so that $\bar{x} \coloneqq (1,0,0)$ is its uniquely determined global minimizer.
	Furthermore, we find
	\begin{align*}
		\widehat C(\bar x) =  T_X(\bar{x}) = \{ 0 \} \times \{ 0 \} \times \{ 0 \},
		\quad	
		C(\bar{x}) &=  T_{F,\Gamma}^{\textup{lin}}(\bar{x})
			= \left\{ \left(2\alpha,\alpha,0 \right)\in\R^3 \,\middle|\, \alpha\in\R_- \right\}.
	\end{align*}
	
	To verify \ref{subMFCII} of ODP-subMFC at $\bar x$, we need to find a sequence 
	$\{(x^k,\lambda^k,\delta^k,\varepsilon^k)\}_{k=1}^\infty\subset\R^{3+3+3+3}$ 
	that is AM-stationary w.r.t.\ $\bar{x}$. Such a sequence needs to fulfill
	the convergences \eqref{eq:sequences_conv} and, using \cref{rem:ass_subMFC},
	\begin{equation}\label{eq:ex_stat_system}
		\begin{aligned}
		\varepsilon^k-
		\begin{pmatrix}	\frac32 \\ 1  \\ 1\end{pmatrix}
		&\in 
		\partial\langle\lambda^k,F\rangle(x^k)
		\\
		&=\,
		\mathsmaller\sum\nolimits_{i=1}^3 |\lambda_i^k|\,\partial\bigl(\sgn(\lambda_i^k)\,F_i\bigr)(x^k)
		\\
		&=\,
		\lambda_1^k
		\left\{\begin{pmatrix} g'(x^k_1) \\ -1 \\ -1 \end{pmatrix}\right\}
		+ 
		|\lambda_2^k|\,
		\Xi(x^k_2,\lambda^k_2)
		+ 	
		\lambda_3^k	
		\left\{\begin{pmatrix} 0 \\ 0 \\ -1 \end{pmatrix}\right\}	
		\end{aligned}
	\end{equation}
	with 
	\[ 
		\Xi(x_2^k,\lambda_2^k)
		=
		\{\sgn(\lambda_2^k)\}\times
		\begin{cases}
			[ -2, 0] 					& x_2^k=0, \lambda_2^k>0 \\
			\{0,2\} 					& x_2^k=0, \lambda_2^k<0 \\
			\{ 0 \} 					& x_2^k > 0 \,\textup{ or }\, \lambda^k_2=0 \\
			\{ -2\sgn(\lambda^k_2) \}	& x_2^k < 0, \lambda^k_2\neq 0 \\
		\end{cases}
		\times\{0\}	
	\]
	as well as
	\begin{align*}
		\lambda^k\in N_\Gamma(F(x^k)-\delta^k) 
		= 
		\begin{cases}
			\R \times \R \times \R_+ & F(x^k) = \delta^k
			\\
			\R \times \R \times \{ 0 \} & F_1(x^k) = \delta_1^k, F_2(x^k) = \delta_2^k, F_3(x^k) < \delta_3^k
			\\
			\emptyset & \textup{otherwise}
		\end{cases}						
	\end{align*}
	for each $k\in\N$. 
	Taking the sequence
	$\{(x^k,\lambda^k,\delta^k,\varepsilon^k)\}_{k=1}^\infty$ with
	\[
		x^k \coloneqq \left(1,0,0\right),
		\quad 
		\lambda^k \coloneqq \left(1,-2,0\right),
		\quad 
		\delta^k \coloneqq \left(0,0, \frac{1}{k}\right),
		\quad 
		\varepsilon^k \coloneqq \left(0,0,0\right),
		\qquad
		\forall k\in\N,
	\]
	we see that it is AM-stationary w.r.t.\ $\bar x$ and 
	fulfills $I(x^k,\delta^k) = \{1,2\}$ for all $k\in\N$. 
	Using $I\coloneqq\{1,2\}$ and this
	sequence in \ref{subMFCI} of ODP-subMFC at $\bar{x}$, we require for any 
	$u \in \R^2 \setminus \{ 0 \}$ with $u\geq 0$ that
	\[
		0 \displaystyle 
		\notin \mathsmaller\sum\nolimits_{i\in I}  
		u_i \partial\bigl(\sgn(\lambda^k_i)\, F_i\bigr)(\bar x)
		=
		u_1
		\left\{\begin{pmatrix} \frac12 \\ -1 \\ -1 \end{pmatrix}\right\}
		+
		u_2
		\bigl(\{-1\}\times \{0,2\} \times \{0\}\bigr)
		,
		\qquad 
		\forall k\in\N
	\]
	holds, which is clearly fulfilled.
	Thus, ODP-subMFC at $\bar x$ is satisfied.
	
	To verify ODP-subMFC$(d)$ at $\bar{x}$ using direction
	$d\coloneqq\frac{1}{\sqrt5}(-2,-1,0) \in C(\bar{x}) \cap \mathbb S$, an AM-stationary sequence
	$\{(x^k,\lambda^k,\delta^k,\varepsilon^k)\}_{k=1}^\infty\subset\R^{3+3+3+3}$
	w.r.t.\ $\bar x$ in direction $d$,
	and an index set $I\subset I(\bar x)$,
	we infer from $d_2<0$ that $x_2^k < \bar x_2 = 0$ has to hold, such that the stationarity 
	system from \eqref{eq:ex_stat_system} now reads as
	\[
		\varepsilon^k-
		\begin{pmatrix}	\frac32 \\ 1 \\ 1 \end{pmatrix}
		\in
		\lambda_1^k
		\left\{\begin{pmatrix} g'(x^k_1) \\ -1 \\ -1 \end{pmatrix}\right\}
		+ 
		\lambda_2^k 
		\left\{\begin{pmatrix} 1 \\ -2 \\ 0 \end{pmatrix}\right\}
		+ 	
		\lambda_3^k	
		\left\{\begin{pmatrix} 0 \\ 0 \\ -1 \end{pmatrix}\right\}.		
	\]
	For sufficiently large $k\in\N$,
	the first line of this system yields
	$\lambda_1^k = 2\sqrt{x_1^k} (\varepsilon_1^k - \frac32 - \lambda_2^k)$.
	Inserting the latter into the second line of the system, 
	keeping $x^k_1\to 1$ in mind, and recalling from
	\cref{rem:bounded_multipliers_under_ODPsubMFC} that $\{ \lambda_2^k \}_{k=1}^\infty$ is 
	bounded whenever ODP-subMFC$(d)$ holds at 
	$\bar x$, we obtain
	$1 = 
	\varepsilon_2^k + 2\sqrt{x_1^k} (\varepsilon_1^k - \frac32 - \lambda_2^k) + 2\lambda_2^k 
	\to -3$,
	which is a contradiction. Thus, ODP-subMFC$(d)$ cannot hold at $\bar x$.
	
	Straightforward calculations show that, indeed,
	$\bar x$ is not M-stationary in direction $d$.
\end{example}

Given $\bar x\in\R^n$, the following example shows that ODP-subMFC can be violated 
even if ODP-subMFC$(d)$ holds for all $d\in C(\bar x) \cap \mathbb{S}$.

\begin{example}\label{ex:dir_but_not_normal_subMFC}
	Consider the optimization problem
	\[
		\min\limits_x\quad x_1^2 \quad\textup{s.t.} \quad F(x)\coloneqq(-|x_1|+|x_2|,-x_1+x_2) 
		\in 
		\Gamma \coloneqq \{ 0\} \times \{ 0\}
	\] 
	with the unique global minimizer $\bar{x} \coloneqq (0,0)$ and
	\[
		C(\bar{x}) =  T_{F,\Gamma}^{\textup{lin}}(\bar{x}) = \widehat C(\bar x) = 
		T_X(\bar{x}) = X = \{ (x,x)\in\R^2 \mid x\in \R\}.
	\]	
	
	To verify \ref{subMFCII} of ODP-subMFC at $\bar x$, we need to find a sequence 
	$\{(x^k,\lambda^k,\delta^k,\varepsilon^k)\}_{k=1}^\infty \subset \R^{2+2+2+2}$ 
	that is	AM-stationary w.r.t.\ $\bar{x}$. For this purpose, it needs to satisfy
	the convergences \eqref{eq:sequences_conv} as well as, 
	using \cref{rem:ass_subMFC},
	\begin{align*}
		\varepsilon^k-
		\begin{pmatrix}	2x_1^k \\ 0 	\end{pmatrix}
		\in &\,
		\partial\langle\lambda^k,F\rangle(x^k)
		\\
		= &\,
		\mathsmaller\sum\nolimits_{i=1}^2
		|\lambda^k_i|\,\partial\bigl(\sgn(\lambda_i^k)\,F_i\bigr)(x^k)
		\\
		= &\,
		|\lambda_1^k|
		\left(
			\begin{cases}
				\{ -1, 1\} & x_1^k=0, \lambda_1^k>0 \\
				[-1,1] & x_1^k=0, \lambda_1^k<0 \\
				\{ -\sgn(\lambda^k_1) \} & x_1^k > 0, \lambda^k_1\neq 0 \\
				\{ \sgn(\lambda^k_1) \} & x_1^k < 0, \lambda^k_1\neq 0 \\
				\{0\} & \lambda^k_1=0
			\end{cases}							
			\times
			\begin{cases}
				[-1,1] & x_2^k=0, \lambda_1^k>0 \\
				\{ -1, 1\} & x_2^k=0, \lambda_1^k<0 \\
				\{ \sgn(\lambda^k_1) \} & x_2^k > 0, \lambda^k_1\neq 0 \\
				\{ -\sgn(\lambda^k_1) \} & x_2^k < 0, \lambda^k_1\neq 0 \\
				\{0\} & \lambda^k_1=0
			\end{cases}		
		\right)
		\\
		&
		+ 
		\lambda_2^k
		\left\{\begin{pmatrix} -1 \\ 1 \end{pmatrix}\right\}					
	\end{align*}
	and
	\begin{align*}
		\lambda^k\in N_\Gamma(F(x^k)-\delta^k) 
		= 
		\begin{cases}
			\R \times \R & F(x^k) = \delta^k
			\\
			\emptyset & F(x^k) \neq \delta^k
		\end{cases}						
	\end{align*}
	for each $k\in\N$. Clearly, the latter implies that we require 
	$F(x^k) = \delta^k$	for all $k\in\N$, such that necessarily $I(x^k,\delta^k)=\{1,2 \}$ 
	holds for all $k\in\N$. Thus, only the choice $I\coloneqq \{1,2\}$ is possible to 
	fulfill ODP-subMFC at $\bar{x}$. 
	In order to avoid a trivial violation of \ref{subMFCI} of ODP-subMFC at $\bar x$,
	we necessarily need $\lambda^k_i\neq 0$ for all $k\in\N$ and $i=1,2$.	
	Then \ref{subMFCI} of ODP-subMFC at $\bar x$ 
	requires for any $u\in \R^2\setminus \{0\}$ with $u\geq 0$ that
	\begin{align*}
		0 \displaystyle 
		&\notin \mathsmaller\sum\nolimits_{i\in I}  
		u_i \partial\bigl(\sgn(\lambda^k_i)\, F_i\bigr)(\bar x)
		\\
		&=
		u_1
		\left(
			\begin{cases}
				\{ -1, 1\} & \lambda_1^k>0 \\
				[-1,1] & \lambda_1^k<0 
			\end{cases}							
			\times
			\begin{cases}
				[-1,1] & \lambda_1^k>0 \\
				\{ -1, 1\} & \lambda_1^k<0
			\end{cases}		
		\right)
		+
		u_2
		\left\{\begin{pmatrix} -\sgn(\lambda_2^k) \\ \sgn(\lambda_2^k) \end{pmatrix}\right\},	
		\qquad 
		\forall k\in\N,
	\end{align*}
	which is clearly violated for $u_1=u_2=1$. Hence, ODP-subMFC at $\bar x$ 
	cannot be satisfied.
	
	Concerning ODP-subMFC$(d)$ at $\bar{x}$ using direction 
	$d\coloneqq\frac{1}{\sqrt{2}}(1,1) \in C(\bar{x}) \cap \mathbb S$,
	one can verify that the sequence 
	$\{(x^k,\lambda^k,\delta^k,\varepsilon^k)\}_{k=1}^\infty$ with
	\[
		x^k \coloneqq \left(\frac{1}{k},\frac{1}{k}\right),
		\quad 
		\lambda^k \coloneqq \left(\frac{1}{k},\frac{1}{k}\right),
		\quad 
		\delta^k \coloneqq (0,0),
		\quad 
		\varepsilon^k \coloneqq \left(0,\frac{2}{k}\right),
		\qquad
		\forall k\in\N
	\]
	is AM-stationary w.r.t.\ $\bar{x}$ in direction $d$ and fulfills 
	$I(x^k,\delta^k) = \{1,2\}$ for	all $k\in\N$. Thus, this sequence can be used in 
	\ref{dir_subMFCII} of ODP-subMFC$(d)$ at $\bar x$ and requires us 
	to set $I\coloneqq\{1,2\}$. Due to $\lambda_1^k > 0$ and $\lambda_2^k > 0$, we 
	further obtain
	\begin{align*}
		\mathsmaller\sum\nolimits_{i\in I}  
		u_i \partial\bigl(\sgn(\lambda^k_i)\, F_i\bigr)(\bar x;d)
		= 
		u_1 \partial F_1(\bar x;d) + 
		u_2 \partial F_2(\bar x;d)
		=
		\left\{(u_1+u_2) \begin{pmatrix} -1 \\ 1 \end{pmatrix}\right\},
	\end{align*}
	which contains the zero vector
	only for $u_1=u_2=0$. Thus, \ref{dir_subMFCI} of ODP-subMFC$(d)$ at $\bar x$ 
	is fulfilled as well, such that ODP-subMFC$(d)$ at $\bar x$ is satisfied. 
	
	Analogously, one can verify ODP-subMFC($d$) at $\bar x$ for the remaining critical direction 
	$d\coloneqq\frac{1}{\sqrt{2}}(-1,-1) \in C(\bar{x}) \cap \mathbb S$
	using the sequence $\{ (x^k,\lambda^k,\delta^k,\varepsilon^k )\}_{k=1}^\infty$
	with
	\[
		x^k \coloneqq \left(-\frac{1}{k},-\frac{1}{k}\right),
		\quad 
		\lambda^k \coloneqq \left(\frac{1}{k},-\frac{1}{k}\right),
		\quad 
		\delta^k \coloneqq (0,0),
		\quad 
		\varepsilon^k \coloneqq \left(0,-\frac{2}{k}\right),
		\quad
		\forall k\in\N.
	\]

	Finally, with multiplier $\lambda \coloneqq (0,0)$, it is straightforward to see that 
	$\bar x$ is indeed M-stationary in directions $d=\pm \frac{1}{\sqrt{2}}(1,1) \in C(\bar{x}) \cap \mathbb S$.
\end{example}

In the remainder of this section, we want to compare (directional) ODP-subMFC to the 
constraint qualifications introduced in \cref{sec:dir_QCs}. 
For the special case where \eqref{eq:orthodisjunctive_problem} models 
inequality-constrained nonsmooth optimization problems, 
a comprehensive comparison to various qualification conditions can be found 
in \cite[Section 4]{KaemingFischerZemkoho2025}.

To start, we present the following auxiliary result.
In order to not clutter the reading flow,
its mainly technical and lengthy proof is presented in \cref{sec:missing_proofs}.
\begin{lemma}\label{lem:lambda_neq_0}
	Let $\bar x\in\R^n$ be feasible for \eqref{eq:orthodisjunctive_problem}.
	Then the following assertions hold.
	\begin{enumerate}
	\item\label{item:lambda_neq_0_nondir}
	If $\bar x$ is an AM-stationary point of \eqref{eq:orthodisjunctive_problem}, 
	then there exists a sequence 
	$\{ (x^k, \lambda^k, \delta^k, \varepsilon^k) \}_{k=1}^\infty \subset \R^{n+\ell+\ell+n}$
	that is AM-stationary w.r.t.\ $\bar x$ 
	and fulfills 
	$\lambda^k_i \neq 0$ for all $i\in I(x^k,\delta^k)$ and $k\in\N$.
	\item\label{item:lambda_neq_0_dir}
	Fix $d\in \widehat C(\bar x)\cap\mathbb S$.
	If $\bar x$ is a local minimizer of \eqref{eq:orthodisjunctive_problem}
	such that $\nabla f(\bar x)\neq 0$ holds, 
	then there exists a sequence 
	$\{ (x^k, \lambda^k, \delta^k, \varepsilon^k) \}_{k=1}^\infty \subset \R^{n+\ell+\ell+n}$
	that is AM-stationary w.r.t.\ $\bar x$ 
	in direction $d$ and fulfills 
	$\lambda^k_i \neq 0$ for all $i\in I(x^k,\delta^k)$ and $k\in\N$.
	\end{enumerate}
\end{lemma}

In the following result, we compare ODP-subMFC and NNAMCQ
as well as the directional versions of both conditions. The corresponding proof,
which exploits \cref{lem:lambda_neq_0},
is again provided in \cref{sec:missing_proofs}.

\begin{proposition}\label{prop:subMFC_vs_NNAMCQ}
	Let $\bar x\in\R^n$ be a local minimizer of \eqref{eq:orthodisjunctive_problem}.
	Then the following relations hold.
	\begin{enumerate}
		\item\label{item:NNAMCQ_ODPsubMFC}
			If NNAMCQ is satisfied at $\bar{x}$, then ODP-subMFC holds at $\bar{x}$.
		\item\label{item:FOSCMS_dirODPsubMFC} 
			Fix $d\in\widehat C(\bar x)\cap\mathbb S$ and let $\nabla f(\bar x)\neq 0$ hold. 
			If FOSCMS$(d)$ is satisfied at $\bar{x}$,
			then ODP-subMFC$(d)$ holds at $\bar{x}$.
	\end{enumerate}
\end{proposition}

Let us note that, given a local minimizer $\bar x\in\R^n$ of \eqref{eq:orthodisjunctive_problem}, 
requiring $\nabla f(\bar x)\neq 0$ in \cref{prop:subMFC_vs_NNAMCQ}\,\ref{item:FOSCMS_dirODPsubMFC}
is not restrictive. Indeed, in the case where $\nabla f(\bar x)=0$ holds,
$\bar x$ is trivially M-stationary in each critical direction with associated multiplier zero,
and a qualification condition is not required.

The following example shows that ODP-subMFC can be violated while other
standard constraint qualifications are satisfied.
\begin{example}\label{ex:MSCQholds_ODPsubMFCviol}
	Recall the optimization problem from \cref{ex:dir_but_not_normal_subMFC},
	for which we have seen that ODP-subMFC is not valid at the global minimizer
	$\bar x \coloneqq (0,0)$.
	As the optimization problem considers a piecewise linear function $F$
	and a polyhedral set $\Gamma$, the associated feasibility mapping is polyhedral. 
	Thus, it is metrically subregular at each point of its graph, which particularly
	implies that MSCQ holds at $\bar x$.
	By \cref{lem:preimage_rule_dir_MS}\,\ref{lem:preimage_rule_dir_MS_nondir} 
	this clearly implies that GACQ and GGCQ 
	hold at $\bar x$ as well.
	Moreover, $\bar x$ is AM-regular as we find 
	$N_\Gamma (F(\bar x)) = \R\times\R$ and 
	$
		\{ \xi \in \partial \langle \lambda, F \rangle (\bar x) \,|\,\lambda\in N_\Gamma(F(\bar x))\}
		=
		\R\times\R.
	$
\end{example}

In contrast, the following example, which is taken from \cite[Example 4.8]{KaemingFischerZemkoho2025},
shows that MSCQ and AM-regularity can be violated while ODP-subMFC holds, 
and that analogous relations
hold for the respective directional versions of the conditions.
\begin{example}\label{ex:ODPsubMFCholds_MSCQviol}
	Consider the optimization problem
	\[
		\min\limits_x\quad (x_1+2)^2 + (x_2-1)^2 
		\quad \textup{s.t.} \quad 
		F(x) \coloneqq x_2^3 - 3x_2 - g(x_1) \in \Gamma \coloneqq \R_-
	\]
with
\[ 
	g(x_1) 
	\coloneqq 
	\begin{cases}
    		x_1^3-3x_1 & x_1\in(-\infty, -2) \cup (1,\infty) \\
    		-2 & x_1\in[-2,1].
  	\end{cases} 
\]
Its feasible set $X$ is depicted in \cite[Figure A1]{KaemingFischerZemkoho2025}, and the
unique global minimizer $\bar x\coloneqq (-2,1)$ fulfills
\[
	\widehat C(\bar{x}) = T_X(\bar{x}) = \R_+ \times \{ 0 \}, 
	\quad 
	C(\bar x) = T_{F,\Gamma}^{\textup{lin}}(\bar{x}) = \R_+ \times \R.
\]
Hence, GGCQ does not hold at $\bar x$.

It can be verified that the sequence 
$\{ (x^k,\lambda^k,\delta^k,\varepsilon^k) \}_{k=1}^\infty \subset \R^{2+1+1+2}$
with
\[
	x^k \coloneqq \left(-2+\frac{1}{k}, 1+\frac{1}{k}\right),
	\quad \lambda^k \coloneqq 0,
	\quad \delta^k \coloneqq \frac{6}{k^2},
	\quad \varepsilon^k \coloneqq \left( \frac{2}{k}, \frac{2}{k}\right),
	\qquad \forall k\in\N,
\]
taken from \cite[Example 4.8]{KaemingFischerZemkoho2025}, is AM-stationary 
w.r.t.\ $\bar x$ in direction $d\coloneqq \frac{1}{\sqrt{2}}(1,1)\in C(\bar{x})\cap \mathbb{S}$
with $I(x^k,\delta^k) = \emptyset$ for all $k\in\N$. Thus, choosing $I \coloneqq \emptyset$, 
both ODP-subMFC and ODP-subMFC$(d)$ are clearly satisfied at $\bar{x}$.

To verify MSCQ$(d)$ at $\bar{x}$ for $d=\frac{1}{\sqrt{2}}(1,1)$, 
we need to find positive constants $\varepsilon>0$, 
$\delta>0$, and $\kappa>0$ such that
\[
	\dist(x,F^{-1}(\Gamma)) \leq \kappa \dist(F(x),\Gamma),
	\qquad
	\forall x\in \{\bar{x}\} + \mathbb B_{\varepsilon,\delta}(d).
\]
Fix some small $\varepsilon>0$ and $\delta>0$. Then $x_\tau \coloneqq (-2+\tau,1+\tau)$ belongs to 
$\{\bar{x}\} + \mathbb B_{\varepsilon,\delta}(d)$ for all 
$\tau\in [0,\frac{\varepsilon}{\sqrt{2}}]$.
It follows from \cite[Figure A1]{KaemingFischerZemkoho2025} that 
$\dist(x_\tau,F^{-1}(\Gamma)) = \dist(x_\tau,X)=\tau$, and $\dist(F(x_\tau),\Gamma)=\tau^3+3\tau^2$
can be computed as well. As in \cite[Example A.1]{KaemingFischerZemkoho2025},
this immediately shows the violation of MSCQ$(d)$ at $\bar{x}$, which implies that MSCQ 
at $\bar x$ does not hold as well.

Concerning AM-regularity of $\bar x$ in direction $d=\frac{1}{\sqrt{2}}(1,1)$, let us consider 
the sequence $\{ (x^k, \lambda^k, \delta^k, \xi^k) \}_{k=1}^\infty\subset \R^{2+1+1+2}$ with 
\[
	x^k \coloneqq \left(-2+\frac{1}{k}, 1+\frac{1}{k}\right),
	\quad \lambda^k \coloneqq k,
	\quad \delta^k \coloneqq \frac{3}{k^2}+\frac{1}{k^3},
	\quad \xi^k \coloneqq \left(0, 6+\frac{3}{k}\right),
	\qquad \forall k\in\N
\]
and $\xi \coloneqq (0,6)$, for which \eqref{eq:approx_stat_dir}, \eqref{eq:dir_AM_reg_conv},
and $x^k\neq \bar x$ for all $k\in\N$ are fulfilled.
Observe that $F(x^k)-\delta^k=0$ and, thus, $\lambda^k\in N_\Gamma(F(x^k)-\delta^k)$
are valid for all $k\in\N$.
Moreover, for each $k\in\N$, noting that $F$ is smooth at $x^k$ yields
\begin{align*}
	\partial \langle \lambda^k, F \rangle (x^k) 
	=
	\left\{k\nabla F(x^k)\right\}
	=
	\left\{\begin{pmatrix} 0 \\ 6 + \frac{3}{k} \end{pmatrix} \right\}.
\end{align*}
Hence, the sequence fulfills \eqref{eq:dir_AM_reg_stat} for all $k\in\N$ as well. 
However, 
\begin{align*}
	\xi \notin
	\{ \xi' \in \partial \langle \lambda, F \rangle (\bar x) \,|\,\lambda\in N_\Gamma(F(\bar x)) \}
	&=
	\{ \xi' \in |\lambda| \partial(\sgn(\lambda) F)(\bar x) \,|\,\lambda\in \R_+ \}
	\\
	&=
	\left\{ \xi' \in \lambda ( [ -9, 0 ] \times\{ 0 \}) \,\middle|\,\lambda\in \R_+ \right\}
	\\
	&=
	\R_- \times \{ 0 \},
\end{align*}
such that $\bar x$ is neither AM-regular nor (strongly) AM-regular in direction $d$.
\end{example}

The subsequent proposition summarizes our above findings on the relation of
(directional) ODP-subMFC to (directional) MSCQ, (directional) AM-regularity, and GGCQ.

\begin{proposition}\label{prop:ODPsubMFC_otherCQs}
Consider a local minimizer $\bar x\in\R^n$ of \eqref{eq:orthodisjunctive_problem}.
Then the following relations hold.
	\begin{enumerate}
		\item\label{item:MSCQ_ODPsubMFC} 
			Conditions MSCQ at $\bar x$ and ODP-subMFC at $\bar x$ are independent.
		\item\label{item:dirMSCQ_dirODPsubMFC} 
			Fix $d\in C(\bar x)\cap\mathbb S$. 
			Then ODP-subMFC$(d)$ at $\bar x$ can hold even if MSCQ$(d)$ at $\bar x$ is violated.
		\item\label{item:AMreg_ODPsubMFC} 
			Conditions AM-regularity of $\bar x$ and ODP-subMFC at $\bar x$ are independent.
		\item\label{item:dirAMreg_dirODPsubMFC} 
			Fix $d\in C(\bar x)\cap\mathbb S$. 
			Then ODP-subMFC$(d)$ at $\bar x$ can hold even if (strong) 
			AM-regularity of $\bar x$ in direction $d$ 
			is violated. 
		\item\label{item:GGCQ_ODPsubMFC} 
			Conditions GGCQ at $\bar x$ and ODP-subMFC at $\bar x$ are independent.
		\item\label{item:GGCQ_dirODPsubMFC} 
			Fix $d\in C(\bar x)\cap\mathbb S$. 
			Then ODP-subMFC$(d)$ at $\bar x$ can hold even if GGCQ at $\bar x$ is violated.
	\end{enumerate}
\end{proposition}
\begin{proof}
	Assertions \ref{item:MSCQ_ODPsubMFC}, \ref{item:AMreg_ODPsubMFC}, and \ref{item:GGCQ_ODPsubMFC}  
	follow from 
	\cref{ex:MSCQholds_ODPsubMFCviol,ex:ODPsubMFCholds_MSCQviol}, whereas assertions
	\ref{item:dirMSCQ_dirODPsubMFC}, \ref{item:dirAMreg_dirODPsubMFC}, and \ref{item:GGCQ_dirODPsubMFC} 
	are derived from \cref{ex:ODPsubMFCholds_MSCQviol}.
\end{proof}

\begin{remark}
	We claim that assertions \ref{item:dirMSCQ_dirODPsubMFC}, \ref{item:dirAMreg_dirODPsubMFC} 
	and \ref{item:GGCQ_dirODPsubMFC} of \cref{prop:ODPsubMFC_otherCQs}
	can be refined analogously to their non-directional counterparts, 
	i.e., that examples exist which show that ODP-subMFC$(d)$ and MSCQ$(d)$,
	ODP-subMFC$(d)$ and (strong) AM-regularity of $\bar x$ in direction $d$, as well as
	ODP-subMFC$(d)$ and GGCQ are even independent.
	However, as the above implications suffice to justify the meaning of ODP-subMFC($d$)
	as an alternative directional qualification condition, and 	
	in the light of \cref{ex:normal_but_not_dir_subMFC,ex:dir_but_not_normal_subMFC}
	as well as \cite[Example 3.12]{KaemingFischerZemkoho2025}, 
	which demonstrate that the presentation of such examples would be a laborious task,
	we omit associated considerations.
\end{remark}

\section{Conclusions}\label{sec:conclusions}

In this paper,
we have shown that approximate directional stationarity conditions provide
necessary optimality conditions for comparatively general optimization problems
with geometric constraints even in the absence of a qualification condition,
see \cref{thm:locmin_dir_akkt_II,thm:locmin_dir_akkt},
complementing result from the literature.
Building upon this insight, 
we suggested a directional qualification condition
that can be used to infer directional stationarity of local minimizers,
see \cref{thm:ODPsubMFC}.
This condition, called directional ODP-subMFC, is based on a single 
sequence quantifying approximate directional stationarity,
which is much in contrast to standard so-called approximate constraint qualifications
that require control over \emph{all} such sequences,
see \cref{sec:approx_CQs}. By means of examples and \cref{prop:subMFC_vs_NNAMCQ},
we illustrated that directional ODP-subMFC
is a comparatively mild qualification condition.

In the literature,
the standard approach to infer directional stationarity of local minimizers is
to exploit a directional version of the metric subregularity constraint qualification,
see \cref{cor:MStat_via_MS}.
The present paper enriches this framework 
by demonstrating three alternative ways to obtain directional stationarity
of local minimizers:
\begin{enumerate}
	\item The local minimizer has to satisfy a generalized version of Guignard's constraint
		qualification, and the linearized feasibility mapping has to enjoy a directional
		metric subregularity property, see \cref{lem:dir_MSt_under_GACQ}.
	\item A generalized version of Abadie's constraint qualification has to hold
		at the local minimizer of interest,
		and an approximate directional constraint qualification has to be valid,
		see \cref{cor:dir_asymp_stat_under_GACQ,cor:Mstat_via_dir_AM_reg}.
	\item At the local minimizer of interest, directional ODP-subMFC has to be satisfied,
		see \cref{lem:approx_MSt_NC} and \cref{thm:ODPsubMFC}.
\end{enumerate}
Particularly, by means of illustrative examples, we have shown that these approaches can identify
directional stationarity of local minimizers in situations where, potentially, the directional version
of the metric subregularity constraint qualification is violated, 
see \cref{ex:GGCQ_without_MSCQ,ex:ODPsubMFCholds_MSCQviol}.

In our future work, we target to elaborate on the benefits of approximate directional
stationarity conditions in numerical optimization. Reinspecting the proofs of
\cref{thm:locmin_dir_akkt_II,thm:locmin_dir_akkt}, 
one may anticipate that external penalty methods generate approximately stationary points
w.r.t.\ critical directions. It remains to be seen whether this, under suitable assumptions,
is true for more reasonable methods from constrained optimization like
augmented Lagrangian or interior point methods.
Following our earlier contributions in \cite[Section~5]{KaemingFischerZemkoho2025}
and \cite[Section~5]{KaemingMehlitz2025}, 
we plan to work out the directional version of ODP-subMFC from \cref{sec:subMFC}
for bilevel and complementarity-constrained optimization problems, respectively.
Both problem classes are highly inherently irregular, 
stirring up the need for weak qualification conditions.
In \cite{KaemingFischerZemkoho2025,KaemingMehlitz2025},
it has been illustrated that the non-directional version of ODP-subMFC is weak enough
to be applicable to these problem classes,
so we expect a similar behavior for its directional counterpart.
The benefits of directional stationarity conditions in bilevel and complementarity-constrained
optimization have been clearly laid out in \cite{BaiYe2022,Gfrerer2014}.


\appendix

\section{Technical details}\label{sec:missing_proofs}

\begin{proof}[Proof of \cref{lem:lambda_neq_0}]
	Let us start with the proof of assertion~\ref{item:lambda_neq_0_nondir}.
	To this end, let us pick a sequence $\{ (x^k, \lambda^k, \delta^k, \varepsilon^k) \}_{k=1}^\infty \subset \R^{n+\ell+\ell+n}$
	that is AM-stationary w.r.t.\ $\bar x$, i.e., we have 
	\eqref{eq:approximate_statonarity} for each $k\in\N$ and the convergences 
	\eqref{eq:sequences_conv}.
	Thus, in particular, there exist subgradients $s_i^k\in \partial(\sgn(\lambda_i^k)\,F_i)(x^k)$ 
	for all 	$i\in\{1,\dotsc,\ell\}$ and $k\in\N$ such that
	\begin{equation}
		\label{eq:approx_stationarity}
		\varepsilon^k
		-
		\nabla f(x^k) 
		=
		\mathsmaller\sum\nolimits_{i=1}^\ell| \lambda_i^k| s_i^k,
		\qquad
		\forall k\in\N,
	\end{equation}
	where we used \cref{lem:sum_rule}.
	For any fixed $k\in\N$, we define the index set $I \coloneqq \{i \in I(x^k,\delta^k) \mid \lambda_i^k=0\}$, 
	which does not depend on $k$ by \cref{rem:constant_sets}.
	This particularly yields
	\begin{equation}
		\label{eq:zero_subgradients_I}
		s_i^k = 0, \qquad \forall i\in I,\, k\in \N.
	\end{equation}
	Moreover, for each $k\in\N$, $\lambda^k \in N_\Gamma(F(x^k)-\delta^k)$ implies that 
	there are sequences $\{ y^{k,\iota}\}_{\iota=1}^\infty\subset\R^\ell$
	and 	$\{ \lambda^{k,\iota} \}_{\iota=1}^\infty\subset\R^\ell$ with 
	$\lambda^{k,\iota}\in \widehat N_\Gamma(y^{k,\iota})$ for all $\iota\in\N$
	as well as $y^{k,\iota}\to F(x^k)-\delta^k$ 
	and $\lambda^{k,\iota}\to\lambda^k$ as $\iota\to\infty$.
	Without loss of generality, we find $I = I_1^k \cup I_2^k$ with
	$I_1^k \coloneqq \{i \in I \mid \lambda_{i}^{k,\iota} = 0\, \forall \iota\in\N\}$
	and $I_2^k \coloneqq \{i \in I \mid \lambda_{i}^{k,\iota} \neq 0\, \forall \iota\in\N\}$ 
	for all $k\in\N$.
	Finally, without loss of generality, we may assume that the latter sets are independent of $k\in\N$,
	i.e., there exist $I_1,I_2\subset I$ such that $I_1=I_1^k$ and $I_2=I_2^k$ hold for all $k\in\N$.
	At this point, exactly one of the following three cases applies, and the remainder of the 
	proof proceeds accordingly.
	
	\textbf{Case 1:} $I_2 \neq \emptyset$.
	We follow the construction from the proof of \cite[Proposition 4.4]{KaemingMehlitz2025}:
	For each $k\in\N$, we pick $\iota(k)\in\N$ large enough such that
	\begin{subequations}\label{eq:diag_sequence_quality}
		\begin{align}
		\label{eq:diag_sequence_quality1}
		&\nnorm{y^{k,\iota(k)}-(F(x^k)-\delta^k)}\leq\frac1k,
		\qquad
		\nnorm{\lambda^{k,\iota(k)}-\lambda^k}\leq\frac1k, 
		\\
		\label{eq:diag_sequence_quality2}
		&\sgn(\lambda_i^{k,\iota(k)}) = \sgn(\lambda_i^k), 
			\qquad \forall i\notin I_2,
		\\
		\label{eq:diag_sequence_quality3}
		&F(x^k)-\delta^k \notin \Gamma_j \Longrightarrow y^{k,\iota(k)} \notin \Gamma_j,
			\qquad \forall j\in\{ 1,\dotsc,t \},
		\\
		\label{eq:diag_sequence_quality4}
		&F_i(x^k)-\delta_i^k \in 
		(a_i^j,b_i^j) \Longrightarrow y_i^{k,\iota(k)} \in (a_i^j,b_i^j),
			\qquad \forall j\in J(x^k,\delta^k),\, i\in \{1,\dotsc,\ell\},
		\end{align}
	\end{subequations}
	and define
	\[
		\tilde\lambda^k
		\coloneqq
		\lambda^{k,\iota(k)},
		\quad
		\tilde\delta^k
		\coloneqq
		F(x^k)-y^{k,\iota(k)},
		\qquad
		\forall k\in\N.
	\]
	Pick any $\hat s_i^k\in \partial(\sgn(\tilde \lambda_{i}^k)\, F_{i})(x^k)$ for all 
	$i\in I_2$, $k\in\N$. Let us further define
	\[	
		\tilde\varepsilon^k
		\coloneqq
		\varepsilon^k 
		+ \mathsmaller\sum\nolimits_{i=1}^\ell (|\tilde \lambda_i^k| - |\lambda_i^k|) s_i^k 
		+ \mathsmaller\sum\nolimits_{i\in I_2} |\tilde \lambda_{i}^k| \hat s_i^k,
		\qquad
		\forall k\in\N,
	\]
	and consider the sequence 
	$\{ (x^k, \tilde\lambda^k, \tilde\delta^k, \tilde\varepsilon^k)\}_{k=1}^\infty\subset\R^{n+\ell+\ell+n}$.
	By construction and \eqref{eq:approx_stationarity} we obtain
	\begin{equation}\label{eq:properties_surrogate_sequences}
		\tilde\varepsilon^k
		-
		\nabla f(x^k) 
		=
		\mathsmaller\sum\nolimits_{i=1}^\ell|\tilde \lambda_i^k|s_i^k 
		+ \mathsmaller\sum\nolimits_{i\in I_2} |\tilde \lambda_{i}^k| \hat s_i^k,
		\quad
		\tilde\lambda^k
		\in 
		\widehat N_\Gamma(F(x^k)-\tilde\delta^k),
		\qquad
		\forall k\in\N.
	\end{equation}
	For all $k\in\N$, we have 
	$s_i^k \in \partial(\sgn(\lambda_i^k)\, F_i)(x^k) = \partial(\sgn(\tilde \lambda_i^k)\, F_i)(x^k)$
	for all $i\notin I_2$ by \eqref{eq:diag_sequence_quality2} and $s^k_{i} = 0$ for all $i\in I_2$ by
	\eqref{eq:zero_subgradients_I}, such that \eqref{eq:properties_surrogate_sequences} 
	and \cref{ass:subMFC} imply
	\[
		\tilde\varepsilon^k
		-
		\nabla f(x^k) 
		\in 
		\mathsmaller\sum\nolimits_{i=1}^\ell|\tilde \lambda_i^k|\partial(\sgn(\tilde \lambda_i^k)\,F_i)(x^k)
		=
		\partial \langle \tilde \lambda^k, F\rangle(x^k),
		\qquad
		\forall k\in\N.
	\]
	Due to \eqref{eq:diag_sequence_quality1} and $\delta^k\to 0$, we find
	\[
		\nnorm{\tilde\delta^k}
		\leq
		\nnorm{\delta^k}+\nnorm{F(x^k)-\delta^k-y^{k,\iota(k)}}
		\leq
		\nnorm{\delta^k} + \frac1k
		\to 
		0.
	\]
	Furthermore, due to \eqref{eq:diag_sequence_quality1}, $\varepsilon^k\to 0$,
	and boundedness of $\{s_i^k\}_{k=1}^\infty$, $i\in\{1,\ldots,\ell\}$,
	and $\{\hat s_i^k\}_{k=1}^\infty$, $i\in I_2$,
	see, e.g., \cite[Theorem~1.22]{Mordukhovich2018},
	\begin{equation}\label{eq:varepsilon_to_0}
		\begin{aligned}
		\nnorm{\tilde\varepsilon^k}
		&\leq
		\nnorm{\varepsilon^k} 
		+ \mathsmaller\sum\nolimits_{i=1}^\ell \bigl||\tilde \lambda_i^k| - |\lambda_i^k|\bigr| \nnorm{s_i^k} 
		+ \mathsmaller\sum\nolimits_{i\in I_2} |\tilde \lambda_{i}^k| \nnorm{\hat s_i^k}
		\\
		&\leq 
		\nnorm{\varepsilon^k} 
		+ \mathsmaller\sum\nolimits_{i=1}^\ell |\tilde \lambda_i^k - \lambda_i^k| \nnorm{s_i^k} 
		+ \mathsmaller\sum\nolimits_{i\in I_2} |\tilde \lambda_{i}^k| \nnorm{\hat s_i^k}
		\\
		&\leq 
		\nnorm{\varepsilon^k} 
		+ \mathsmaller\sum\nolimits_{i=1}^\ell \frac{1}{k} \nnorm{s_i^k} 
		+ \mathsmaller\sum\nolimits_{i\in I_2} \frac{1}{k} \nnorm{\hat s_i^k}
		\to 0
		\end{aligned}
	\end{equation}
	is obtained,
	where we used $\lambda^k_i=0$ for all $i\in I_2$ and $k\in\N$ to apply
	\eqref{eq:diag_sequence_quality1} in the third summand.
	Thus, we have shown that 
	$\{ (x^k, \tilde\lambda^k, \tilde\delta^k, \tilde\varepsilon^k)\}_{k=1}^\infty$
	is an AM-stationary sequence w.r.t.\ $\bar x$.
	Hence, we can restart the proof using the sequences
	$\{ (x^k, \lambda^k, \delta^k, \varepsilon^k)\}_{k=1}^\infty 
	\coloneqq
	\{ (x^k, \tilde\lambda^k, \tilde\delta^k, \tilde\varepsilon^k)\}_{k=1}^\infty$,
	$\{ y^{k,\iota} \}_{\iota=1}^\infty \coloneqq \{ F(x^k) - \tilde \delta^k \}_{\iota=1}^\infty$ for all $k\in\N$,
	and $\{ \lambda^{k,\iota} \}_{\iota=1}^\infty \coloneqq\{ \tilde \lambda^k \}_{\iota=1}^\infty$ for all $k\in\N$,
	for which $\lambda^{k,\iota}\in\widehat{N}_\Gamma(y^{k,\iota})$ 
	follows for all $k,\iota\in\N$ from \eqref{eq:properties_surrogate_sequences},
	and use that the set $I_2$ corresponding to these sequences fulfills $I_2 = \emptyset$.	 
	
	\textbf{Case 2:} $I_2 = \emptyset$, $I_1 \neq \emptyset$. 
	Recall that this means
	$\lambda_i^{k,\iota} = 0$ for all $i\in I_1$ and $k,\iota\in\N$.
	For all $k,\iota\in\N$, we further know $\lambda^{k,\iota} \in \widehat{N}_\Gamma(y^{k,\iota})$.
		Using $J(y) \coloneqq \{ j \in \{1,\dotsc,t\} \mid y \in \Gamma_j\}$, it
		follows from the orthodisjunctive structure in \eqref{eq:orthodisjunctive_problem}
		and \cite[Lemma 2.1]{KaemingMehlitz2025} that
		\begin{equation}\label{eq:NNAMCQ_ODPsubMFC_2}
			\widehat{N}_\Gamma(y^{k,\iota})
			= 
			\bigcap_{j\in J(y^{k,\iota})} \widehat{N}_{\Gamma_j}(y^{k,\iota})
		\end{equation}
		holds for all $k,\iota\in\N$. For every $j\in \{1,\dotsc,t\}$, set $\Gamma_j$ 
		is the product of $\ell$ closed intervals. Thus,
		for all $j\in J(y^{k,\iota})$ and $k,\iota\in\N$, we find
		\begin{equation*}
			\widehat{N}_{\Gamma_j}(y^{k,\iota}) 
			= 
			\prod_{i=1}^\ell K_i^{j,k,\iota}
		\end{equation*}
		for some $K_i^{j,k,\iota} \in \{\{ 0 \}, \R_+, \R_-, \R\}$, 
		$i\in\{1,\dotsc,\ell\}$. For all $i\in\{1,\dotsc,\ell\}$, $j\in J(y^{k,\iota})$, and
		$k,\iota\in\N$,
		we may assume without loss of generality that we can find sets
		$K_i^{j} \in \{\{ 0 \}, \R_+, \R_-, \R\}$ such that
		$K_i^j = K_i^{j,k,\iota}$ holds.
		Consequently, the intersection in \eqref{eq:NNAMCQ_ODPsubMFC_2} implies that we can find
		$L_i \in \{\{ 0 \}, \R_+, \R_-, \R\}$, $i\in\{1,\dotsc,\ell\}$, such that 
		\begin{equation*}
			\widehat{N}_{\Gamma}(y^{k,\iota})
			= 
			\prod_{i=1}^\ell L_i
		\end{equation*}
		holds for all $k,\iota\in\N$.
	Let 
	\[
		I_{11} \coloneqq \{ i\in I_1 \mid L_i = \{ 0\} \},
		\qquad
		I_{12} \coloneqq \{ i\in I_1 \mid L_i \neq \{ 0\} \}.
	\]
	Depending on these sets, we now proceed with either Case 2.1 or Case 2.2.
	
	\textbf{Case 2.1:} $I_2 = \emptyset$, $I_{12} \neq \emptyset$.
	For each $k\in\N$, define
	$\tilde\lambda^k\in\R^\ell$ by
	\[
		\tilde\lambda^k_i 
		\coloneqq 
		\begin{cases}
			\lambda^k_i & i\in\{1,\dotsc,\ell\}\setminus I_{12} \\
			\frac{1}{k} & i\in I_{12}\colon L_i \in \{ \R_+, \R\} \\
			-\frac{1}{k} & i\in I_{12}\colon L_i = \R_-.
		\end{cases}
	\]
	Then, for all $k,\iota\in\N$, as $\lambda^k\in \widehat N_\Gamma(y^{k,\iota})$ follows from
	$\lambda^{k,\iota} \in \widehat N_\Gamma(y^{k,\iota})$, $\lambda^{k,\iota}\to \lambda^k$, 
	and $\widehat N_\Gamma(y^{k,\iota})$ closed, the above definition yields 
	$\tilde\lambda^k \in \widehat N_\Gamma(y^{k,\iota})$.
	Pick any $\hat s_i^k \in \partial (\sgn(\tilde \lambda_i^k)\, F_i)(x^k)$ for all $i\in I_{12}$,
	$k\in\N$, and define
	\begin{equation}\label{eq:tilde_eps_case_2.1}
		\tilde \varepsilon^k \coloneqq \varepsilon^k 
		+ \mathsmaller\sum\nolimits_{i\in I_{12}} |\tilde \lambda_{i}^k| \hat s_i^k,
		\qquad
		\forall k\in\N.
	\end{equation}
	Using \eqref{eq:approx_stationarity}, $\lambda_i^k=0$ for all $i\in I_{12}$, $k\in\N$, 
	$\lambda_i^k = \tilde \lambda_i^k$ for all $i\in\{1,\dotsc,\ell\}\setminus I_{12}$, $k\in \N$,
	and $\sgn(\lambda_i^k)=\sgn(\tilde \lambda_i^k)$ 
	for all $i\in\{1,\dotsc,\ell\}\setminus I_{12}$, $k\in\N$, the sequence 
	$\{ (x^k,\tilde\lambda^k,\delta^k,\tilde\varepsilon^k) \}_{k=1}^\infty 
	\subset\R^{n+\ell+\ell+n}$ fulfills
	\[
		\tilde\varepsilon^k
		-
		\nabla f(x^k) 
		\in 
		\mathsmaller\sum\nolimits_{i=1}^\ell|\tilde \lambda_i^k|\partial(\sgn(\tilde \lambda_i^k)\,F_i)(x^k),
		\quad
		\tilde\lambda^k
		\in 
		N_\Gamma(F(x^k)-\delta^k),	
		\qquad
		\forall k\in\N,
	\]
	where the second part follows from $\tilde\lambda^k \in \widehat N_\Gamma(y^{k,\iota})$ 
	for all $k,\iota\in\N$ and $y^{k,\iota}\to F(x^k)-\delta^k$.
	Moreover, we use the definition of $\tilde\lambda^k$, $\varepsilon^k\to 0$, and boundedness of 
	$\{ \hat s_i^k\}_{k=1}^\infty$, $i\in I_{12}$, see, e.g., \cite[Theorem~1.22]{Mordukhovich2018},
	to find
	\[
		\nnorm{\tilde\varepsilon^k} 
		\leq 
		\nnorm{\varepsilon^k} + \mathsmaller\sum\nolimits_{i\in I_{12}} |\tilde\lambda_i^k| \nnorm{\hat s_i^k}
		\leq 
		\nnorm{\varepsilon^k} + \mathsmaller\sum\nolimits_{i\in I_{12}} \frac{1}{k}\nnorm{\hat s_i^k}
		\to 
		0.
	\]
	Thus, recalling \cref{ass:subMFC}, we have shown that 
	$\{ (x^k, \tilde\lambda^k, \delta^k, \tilde\varepsilon^k)\}_{k=1}^\infty$ 
	is an AM-stationary sequence w.r.t.\ $\bar x$, where $\tilde\lambda_i^k \neq 0$ 
	holds for all $i\in I_2 \cup I_{12} \cup (I(x^k,\delta^k)\setminus I_{11})$, $k\in\N$.
	Hence, we can restart the proof using the sequences
	$\{ (x^k, \lambda^k, \delta^k, \varepsilon^k)\}_{k=1}^\infty 
	\coloneqq
	\{ (x^k, \tilde\lambda^k, \delta^k, \tilde\varepsilon^k)\}_{k=1}^\infty$,
	$\{ y^{k,\iota} \}_{\iota=1}^\infty$ for all $k\in\N$,
	and 
	$\{ \lambda^{k,\iota} \}_{\iota=1}^\infty \coloneqq\{ \tilde \lambda^k \}_{\iota=1}^\infty$ for all $k\in\N$,
	which fulfill $\lambda^{k,\iota}\in \widehat N_\Gamma(y^{k,\iota})$ and preserve the 
	sets $I(x^k,\delta^k)$ and $\widehat N_\Gamma(y^{k,\iota})$ for all $k,\iota\in\N$,
	and use that the sets $I_{12}$ and $I_2$ corresponding to these sequences, thus, satisfy
	$I_{12} = I_2 = \emptyset$.	
	
	\textbf{Case 2.2:} $I_2=I_{12}=\emptyset$, $I_{11}\neq\emptyset$.
	Recall that this means
	\begin{equation*}
		\widehat{N}_{\Gamma}(y^{k,\iota}) 
		= 
		\bigcap_{j\in J(y^{k,\iota})} \widehat N_{\Gamma_j}(y^{k,\iota})
		=
		\prod_{i=1}^\ell L_i,
		\qquad
		\forall k,\iota\in\N
	\end{equation*}
	with $L_i = \{ 0 \}$ for all $i\in I_{11}$ and $L_i \in \{\{ 0 \}, \R_+, \R_-, \R\}$ for
	all $i\in\{1,\dotsc,\ell\}\setminus I_{11}$.
	For each $i\in I_{11}$, the above structure implies that $L_i = \{ 0 \}$ only appears 
	if, for all $k,\iota\in\N$, 
	$y_i^{k,\iota} \in (a_i^{j}, b_i^{j})$
	for some $j\in J(y^{k,\iota})$ or
	if there are $j_1,j_2\in J(y^{k,\iota})$ with $j_1\neq j_2$, 
	$(a_i^{j_1}, b_i^{j_1}) \neq\emptyset$,
	$(a_i^{j_2}, b_i^{j_2}) \neq\emptyset$, 
	$y_i^{k,\iota} = a_i^{j_1}$, and 
	$y_i^{k,\iota} = b_i^{j_2}$.
	Fix $i_0\in I_{11}$ and define, for all $k,\iota\in\N$, 
	\begin{align*}
		J_1^{k,\iota} &\coloneqq \{ j\in J(y^{k,\iota}) \mid y_{i_0}^{k,\iota} = a_{i_0}^{j}\},
		\\
		J_2^{k,\iota} &\coloneqq \{ j\in J(y^{k,\iota}) \mid y_{i_0}^{k,\iota} = b_{i_0}^{j}\},
		\\
		J_3^{k,\iota} &\coloneqq \{ j\in J(y^{k,\iota}) \mid 
			y_{i_0}^{k,\iota} \in (a_{i_0}^{j}, b_{i_0}^{j})\},
	\end{align*}
	which clearly yields $J(y^{k,\iota}) = J_1^{k,\iota} \cup J_2^{k,\iota} \cup J_3^{k,\iota}$.
	Without loss of generality, these sets can be assumed to be independent of $k,\iota\in\N$,
	i.e., there are $J_1,J_2,J_3\subset J(y^{k,\iota})$ such that $J_1=J_1^{k,\iota}$,
	$J_2=J_2^{k,\iota}$, and $J_3=J_3^{k,\iota}$ hold for all $k,\iota\in\N$. 
	Moreover, $(J_1\setminus J_2) \cup J_3\neq\emptyset$ follows from $L_{i_0}=\{0\}$, 
	as explained above.
	Again, we follow the construction 
	from the proof of \cite[Proposition 4.4]{KaemingMehlitz2025}:
	For each $k\in\N$, we pick $\iota(k)\in\N$ large enough such that
	\eqref{eq:diag_sequence_quality1}, \eqref{eq:diag_sequence_quality3},
	\eqref{eq:diag_sequence_quality4}, and
	\begin{equation}
		\label{eq:diag_sequence_quality5}
		\sgn(\lambda_i^{k,\iota(k)}) = \sgn(\lambda_i^k),
			\qquad \forall i\in\{1,\dotsc,\ell\}
	\end{equation}
	hold, where \eqref{eq:diag_sequence_quality5} can be satisfied as $I_2 = \emptyset$ led
	to this case.
	Let us define 
	$\eta_k \coloneqq \min_{j\in (J_1 \setminus J_2) \cup J_3} 2^{-k}(b_{i_0}^j - y_{i_0}^{k,\iota(k)})$
	for all $k\in\N$,
	where $\eta_k > 0$ is ensured for all $k\in\N$ due to $b_{i_0}^j > y_{i_0}^{k,\iota(k)}$ 
	for all $j\in (J_1 \setminus J_2) \cup J_3$ and $k\in\N$.
	We also note that $\eta_k\downarrow 0$ follows
	from boundedness of $\{y^{k,\iota(k)}\}_{k=1}^\infty$,
	which is a consequence of \eqref{eq:diag_sequence_quality1},
	$x^k\to\bar x$, and $\delta^k\to 0$.
	Define
	\begin{equation}\label{eq:tilde_seq_case_2.2}
		\tilde\lambda^k
		\coloneqq
		\lambda^{k,\iota(k)},
		\quad
		\tilde y_i^k 
		\coloneqq 
		\begin{cases}
			y_i^{k,\iota(k)} + \eta_k & i=i_0 \\
			y_i^{k,\iota(k)} & i \neq i_0,
		\end{cases}
		\quad
		\tilde\delta^k
		\coloneqq
		F(x^k)-\tilde y^{k},
		\qquad
		\forall k\in\N.
	\end{equation}
	With $\delta^k \to 0$, \eqref{eq:diag_sequence_quality1}, and $\eta_k \downarrow 0$, we verify 
	\begin{align*}
		\nnorm{\tilde \delta^k} 
		&\leq
		\nnorm{\delta^k}+\nnorm{F(x^k)-\delta^k- \tilde y^k}	
		\\
		&\leq
		\nnorm{\delta^k}
			+\nnorm{	F(x^k)-\delta^k-y^{k,\iota(k)}}
			+|\eta_k|
		\\
		&\leq 
		\nnorm{\delta^k} + \frac{1}{k} + |\eta_k|
		\to 0.
	\end{align*}
	Moreover, for all $k\in\N$, we obtain
	\begin{subequations}\label{lambda_neq_0_proof}
		\begin{align}
			\label{lambda_neq_0_proof1}
			F_{i_0}(x^k)-\tilde\delta_{i_0}^k &= y_{i_0}^{k,\iota(k)}+\eta_k 
				\in (a_{i_0}^{j}, b_{i_0}^{j}), 
				\qquad\forall j\in (J_1\setminus J_2) \cup J_3,
			\\
			F_{i_0}(x^k)-\tilde\delta_{i_0}^k &= y_{i_0}^{k,\iota(k)}+\eta_k 
				= b_{i_0}^j + \eta_k\notin [a_{i_0}^{j}, b_{i_0}^{j}], 
				\qquad\forall j\in J_2,
			\\
			\label{lambda_neq_0_proof2}
			F_{i}(x^k)-\tilde\delta_{i}^k &= y_{i}^{k,\iota(k)}
				\in [a_{i}^{j}, b_{i}^{j}], 
				\qquad\forall i\neq i_0,\, j\in J(y^{k,\iota(k)}).
		\end{align}
	\end{subequations}
	For all $j\notin J(y^{k,\iota(k)})$, $k\in\N$, we either have $y_{i}^{k,\iota(k)} \notin [a_i^j,b_i^j]$
	for some $i\neq i_0$, in which case $j\notin J(x^k,\tilde\delta^k)$ immediately follows from
	$F_i(x^k)-\tilde\delta_i^k = y_i^{k,\iota(k)}$, 
	or $y_{i_0}^{k,\iota(k)} \notin [a_{i_0}^j,b_{i_0}^j]$, in which 
	case we consider the tail of the sequence for which 
	$y_{i_0}^{k,\iota(k)}+\eta_k \notin [a_{i_0}^{j}, b_{i_0}^{j}]$ holds 
	due to $\eta_k \downarrow 0$, implying $j\notin J(x^k,\tilde\delta^k)$ again via 
	$F_{i_0}(x^k)-\tilde\delta_{i_0}^k = y_{i_0}^{k,\iota(k)}+\eta_k$.
	In either case, we obtain, together with \eqref{lambda_neq_0_proof},
	$J(x^k,\tilde\delta^k) = (J_1 \setminus J_2) \cup J_3 \neq \emptyset$ for all $k\in\N$,
	which especially ensures
	$i_0 \notin I(x^k,\tilde \delta^k)$ for all $k\in\N$ due to \eqref{lambda_neq_0_proof1}.
	For all $k\in\N$, this further yields 
	$J(x^k,\tilde\delta^k) = (J_1 \setminus J_2) \cup J_3 \subset J(y^{k,\iota(k)})$. 
	With this, it follows for all $k\in\N$ that
	\[
		\widehat N_\Gamma (y^{k,\iota(k)})
		=
		\bigcap_{j\in J(y^{k,\iota(k)})} \widehat N_{\Gamma_j}(y^{k,\iota(k)})
		\subset
		\bigcap_{j\in J(x^k,\tilde\delta^k)} \widehat N_{\Gamma_j}(F(x^k)-\tilde\delta^k)
		=
		\widehat N_\Gamma(F(x^k)-\tilde\delta^k)
	\]	
	holds.
	Indeed, the two equalities follow from \cite[Lemma 2.1]{KaemingMehlitz2025},
	while the inclusion can be verified componentwise as the regular limiting normal cone
	enjoys the product rule, see, e.g., \cite[Proposition 6.41]{RockafellarWets1998}.
	For all components $i\neq i_0$, the inclusion is satisfied as 
	$y_i^{k,\iota(k)} = \tilde y^k_i = F_i(x^k)-\tilde\delta_i^k$ and 
	$J(x^k,\tilde\delta^k)\subset J(y^{k,\iota(k)})$.
	In component $i_0$, both intersections yield the set $\{ 0 \}$ due to $i_0\in I_{11}$ and
	\eqref{lambda_neq_0_proof1} together with $J(x^k,\tilde\delta^k) = (J_1 \setminus J_2) \cup J_3$.
	Thus, it follows from $\tilde \lambda^k\in \widehat N_\Gamma (y^{k,\iota(k)})$ that we find 
	$\tilde  \lambda^k \in\widehat N_\Gamma (F(x^k)-\tilde \delta^k)
	\subset N_\Gamma (F(x^k)-\tilde \delta^k)$ 
	for all $k\in\N$.
	Finally, pick any $\hat s^k \in \partial (\sgn(\tilde \lambda_{i_0}^k)\, F_{i_0})(x^k)$ 
	for all $k\in\N$, and define
	\[
		\tilde \varepsilon^k \coloneqq \varepsilon^k 
		+ \mathsmaller\sum\nolimits_{i=1}^\ell (|\tilde \lambda_i^k| - |\lambda_i^k|) s_i^k 	
		+ |\tilde \lambda_{i_0}^k| \hat s^k,
		\qquad
		\forall k\in\N.
	\]
	Due to \eqref{eq:approx_stationarity}, $s_{i_0}^k=0$ for all $k\in\N$ by 
	\eqref{eq:zero_subgradients_I}, and $\sgn(\lambda_i^k)=\sgn(\tilde \lambda_i^k)$ for all
	$i\in\{1,\dotsc,\ell\}$, $k\in\N$ by \eqref{eq:diag_sequence_quality5}, the sequence 
	$\{ (x^k,\tilde\lambda^k,\tilde\delta^k,\tilde\varepsilon^k) \}_{k=1}^\infty 
	\subset\R^{n+\ell+\ell+n}$ fulfills
	\[
		\tilde\varepsilon^k
		-
		\nabla f(x^k) 
		\in 
		\mathsmaller\sum\nolimits_{i=1}^\ell|\tilde \lambda_i^k|\partial(\sgn(\tilde \lambda_i^k)\,F_i)(x^k),
		\qquad
		\forall k\in\N,
	\]
	and $\nnorm{\tilde\varepsilon^k}\to 0$ can be shown as in \eqref{eq:varepsilon_to_0}.
	Thus, recalling \cref{ass:subMFC}, we have shown that 
	$\{ (x^k, \tilde\lambda^k, \tilde \delta^k, \tilde\varepsilon^k)\}_{k=1}^\infty$ 
	is an AM-stationary sequence w.r.t.\ $\bar x$, 
	where $i_0 \notin I(x^k,\tilde\delta^k)$.
	Moreover, $I(x^k,\tilde\delta^k) \subset I(x^k,\delta^k)$ follows from 
	\eqref{eq:diag_sequence_quality3} and \eqref{eq:diag_sequence_quality4} as in the proof
	of \cite[Proposition 4.4]{KaemingMehlitz2025}.
	Hence, we can restart the proof using the sequences
	$\{ (x^k, \lambda^k, \delta^k, \varepsilon^k)\}_{k=1}^\infty 
	\coloneqq
	\{ (x^k, \tilde\lambda^k, \tilde\delta^k, \tilde\varepsilon^k)\}_{k=1}^\infty$,
	$\{ y^{k,\iota} \}_{\iota=1}^\infty \coloneqq\{ F(x^k)-\tilde\delta^k \}_{\iota=1}^\infty$ for all $k\in\N$,
	and $\{ \lambda^{k,\iota} \}_{\iota=1}^\infty \coloneqq\{ \tilde \lambda^k \}_{\iota=1}^\infty$ for all $k\in\N$,
	which fulfill $\lambda^{k,\iota}\in \widehat N_\Gamma(y^{k,\iota})$ for all $k,\iota\in\N$ as shown above,
	and use that the set $I_2$ corresponding to these sequences fulfills $I_2 = \emptyset$.	 
	However, it is not apparent that the set $I_{12}$ corresponding to these sequences
	preserves $I_{12} = \emptyset$, such that, after the restart, we may have to proceed with
	another iteration of Case 2.1. Nevertheless, as the present case fulfills 
	$I(x^k,\tilde\delta^k) \subset I(x^k,\delta^k)$ and removes the element $i_0\in I_{11}$
	from $I(x^k,\tilde\delta^k)$,
	while Case 2.1 empties $I_{12}$ 
	without enlarging any other subset of $I$, alternating between the two cases is a terminating 
	procedure.
	
	\textbf{Case 3:} $I=\emptyset$. By definition of the set $I$, the claim immediately follows.\\
	
	Next, we prove assertion~\ref{item:lambda_neq_0_dir}.
	According to \cref{thm:locmin_dir_akkt_II},
	we find an AM-stationary sequence 
	$\{ (x^k, \lambda^k, \delta^k, \varepsilon^k) \}_{k=1}^\infty \subset \R^{n+\ell+\ell+n}$
	w.r.t.\ $\bar x$ in direction $d$, i.e., we have 
	\eqref{eq:approximate_statonarity} and $x^k\neq \bar x$ for each $k\in\N$, the 
	convergences \eqref{eq:sequences_conv}, and \eqref{eq:approx_stat_dir}. 
	Due to \cref{rem:nonvanishing_seq}, we may also assume that $\delta^k\neq 0$
	and $\lambda^k\neq 0$ hold for all $k\in\N$.
	We now follow the proof of the non-directional version from above subject to the following
	exceptions.
	
	\textbf{Cases 1 and 2.2:} In these cases, we replace \eqref{eq:diag_sequence_quality1} by
	\begin{subequations}\label{eq:diag_sequence_quality1_new}
		\begin{align}
			\label{eq:diag_sequence_quality1_new1}
			\nnorm{y^{k,\iota(k)}-(F(x^k)-\delta^k)}
			&\leq
			\min\{\nnorm{x^k-\bar x},2^{-k}\}
			\nnorm{\delta^k},
			\\
			\label{eq:diag_sequence_quality1_new2}
			\nnorm{\lambda^{k,\iota(k)}-\lambda^k}
			&\leq
			\min\{\nnorm{x^k-\bar x},2^{-k}\}
			\min\{1,\nnorm{\lambda^k}\}, 
		\end{align}
	\end{subequations}
	which can indeed be fulfilled for $\iota(k)\in\N$ large enough due to 
	$\nnorm{\delta^k}>0$, $\min\{1,\nnorm{\lambda^k}\}>0$, and $x^k\neq \bar x$ 
	for all $k\in\N$. 
	For the sequence 
	$\{ (x^k, \tilde\lambda^k, \tilde\delta^k, \tilde\varepsilon^k) \}_{k=1}^\infty\subset\R^{n+\ell+\ell+n}$
	constructed during Case 1 or Case 2.2, we now need to show that it is 
	AM-stationary w.r.t.\ $\bar x$ in direction $d$. Thus, following the proof
	from above and recalling $x^k\neq \bar x$ for all $k\in\N$, it remains to 
	show \eqref{eq:approx_stat_dir} and confirm that \eqref{eq:sequences_conv} 
	is not influenced by demanding \eqref{eq:diag_sequence_quality1_new} instead of
	\eqref{eq:diag_sequence_quality1}. 
	
	Let us start with the sequence in Case 1.
	As $\nnorm{x^k-\bar x} \to 0$, 
	$\nnorm{\delta^k}\to 0$, and $\{ \min\{1,\nnorm{\lambda^k}\} \}_{k=1}^\infty$ bounded, 
	we obtain $\nnorm{x^k-\bar x} \nnorm{\delta^k} \to 0$
	and $\nnorm{x^k-\bar x} \min\{1,\nnorm{\lambda^k}\} \to 0$, such that
	the convergences $\nnorm{\tilde\delta^k}\to 0$ and 
	$\nnorm{\tilde \varepsilon^k}\to 0$ can be obtained as for the non-directional
	version, i.e., \eqref{eq:sequences_conv} holds. 
	Concerning \eqref{eq:sequences_dir_conv}, the first convergence follows from 
	the fact that
	$\{ (x^k,\lambda^k,\delta^k,\varepsilon^k) \}_{k=1}^\infty$ is
	AM-stationary w.r.t.\ $\bar x$ in direction $d$. For the second convergence,
	we use the definition of $\tilde\delta^k$, \eqref{eq:diag_sequence_quality1_new1},
	\eqref{eq:sequences_dir_conv} for 
	$\{ (x^k,\lambda^k,\delta^k,\varepsilon^k) \}_{k=1}^\infty$, and $\delta^k\to 0$, which yield
	\begin{align*}
		\frac{\nnorm{\tilde\delta^k}}{\nnorm{x^k-\bar x}} 
		&= 
		\frac{\nnorm{\delta^k}}{\nnorm{x^k-\bar x}} 
			+ \frac{\nnorm{F(x^k)-\delta^k-y^{k,\iota(k)}}}{\nnorm{x^k-\bar x}}
		\leq 	
		\frac{\nnorm{\delta^k}}{\nnorm{x^k-\bar x}} 
			+ \nnorm{\delta^k}
		\to 
		0,
	\end{align*}
	and, thus, $\tilde\delta^k/\nnorm{x^k-\bar x}\to 0$.
	To show \eqref{eq:perturbation_vs_multiplier}, we first note that \eqref{eq:diag_sequence_quality1_new} yields,
	for all $k\in\N$,
	\[
		\nnorm{\tilde\delta^k}
		\geq
		(1-2^{-k})\nnorm{\delta^k}
		>
		0,
		\qquad
		\nnorm{\tilde\lambda^k}
		\geq
		(1-2^{-k})\nnorm{\lambda^k}
		>
		0,		
	\]
	i.e., $\tilde\delta^k\neq 0$ and $\tilde\lambda^k\neq 0$. Thus, proving \eqref{eq:perturbation_vs_multiplier}
	is equivalent to showing \eqref{eq:residual_va_multiplier_nontrivial}.
	Observe that we can rewrite
	\begin{equation}\label{eq:splitting_limit}
		\lim_{k\to\infty}\left(
			\frac{\tilde\delta^k}{\nnorm{\tilde\delta^k}} - \frac{\tilde\lambda^k}{\nnorm{\tilde\lambda^k}}
		\right)
		=
		\lim_{k\to\infty}\left(
			\frac{\tilde\delta^k}{\nnorm{\tilde\delta^k}} - \frac{\delta^k}{\nnorm{\delta^k}}
			+ \frac{\lambda^k}{\nnorm{\lambda^k}}-\frac{\tilde\lambda^k}{\nnorm{\tilde\lambda^k}}
		\right)
	\end{equation}
	by \eqref{eq:residual_va_multiplier_nontrivial}
	for $\{ (x^k,\lambda^k,\delta^k,\varepsilon^k) \}_{k=1}^\infty$, 
	which follows from \eqref{eq:perturbation_vs_multiplier} 
	for $\{ (x^k,\lambda^k,\delta^k,\varepsilon^k) \}_{k=1}^\infty$ 
	as we have $\delta^k\neq 0$ and $\lambda^k\neq 0$ for all $k\in\N$.
	Then,
	\begin{align*}
		\left\Vert
			\frac{\delta^k}{\nnorm{\delta^k}}-\frac{\tilde\delta^k}{\nnorm{\tilde\delta^k}}
		\right\Vert
		&=
		\frac{\bigl\Vert\nnorm{\tilde\delta^k}\delta^k-\nnorm{\delta^k}\tilde\delta^k\bigr\Vert}
			{\nnorm{\delta^k}\nnorm{\tilde\delta^k}}
		\\
		&=
		\frac{\bigl\Vert\nnorm{\tilde\delta^k}\delta^k-\nnorm{\tilde\delta^k}\tilde\delta^k
		+\nnorm{\tilde\delta^k}\tilde\delta^k-\nnorm{\delta^k}\tilde\delta^k\bigr\Vert}
			{\nnorm{\delta^k}\nnorm{\tilde\delta^k}}
		\\
		&\leq
		\frac{2\nnorm{\tilde\delta^k}\nnorm{\tilde\delta^k-\delta^k}}
		{\nnorm{\delta^k}\nnorm{\tilde\delta^k}}
		=
		\frac{2\nnorm{\tilde\delta^k-\delta^k}}{\nnorm{\delta^k}}
		\\
		&\leq
		\frac{2\nnorm{x^k-\bar x} \nnorm{\delta^k}}{\nnorm{\delta^k}}
		= 
		2\nnorm{x^k-\bar x} \to 0,
	\end{align*}
	where the definition of $\tilde\delta^k$ allowed us to use \eqref{eq:diag_sequence_quality1_new1}
	in the second estimate. Analogously,
	\begin{align*}
		\left\Vert
			\frac{\lambda^k}{\nnorm{\lambda^k}}-\frac{\tilde\lambda^k}{\nnorm{\tilde\lambda^k}}
		\right\Vert
		&\leq
		\frac{2\nnorm{\tilde\lambda^k-\lambda^k}}{\nnorm{\lambda^k}}
		\leq
		\frac{2\nnorm{x^k-\bar x} \nnorm{\lambda^k}}{\nnorm{\lambda^k}}
		= 
		2\nnorm{x^k-\bar x} \to 0
	\end{align*}
	is obtained using the definition of $\tilde\lambda^k$ and \eqref{eq:diag_sequence_quality1_new2}.
	An insertion into \eqref{eq:splitting_limit} yields
	\[
		\frac{\tilde\delta^k}{\nnorm{\tilde\delta^k}} - \frac{\tilde\lambda^k}{\nnorm{\tilde\lambda^k}}
		\to 0,
	\]
	such that we have shown \eqref{eq:perturbation_vs_multiplier}.
	Concerning \eqref{eq:multipliers_not_arbitrarily_unbounded}, we use the definition 
	of $\tilde\delta^k$ and \eqref{eq:diag_sequence_quality1_new} to obtain
	\begin{align*}
		\frac{\nnorm{\tilde\delta^k}\nnorm{\tilde\lambda^k}}{\nnorm{x^k-\bar x}}
		\leq&\,
		\frac{\nnorm{\delta^k}\nnorm{\tilde\lambda^k}}{\nnorm{x^k-\bar x}} 
			+ \frac{\nnorm{F(x^k)-\delta^k-y^{k,\iota(k)}}\nnorm{\tilde\lambda^k}}{\nnorm{x^k-\bar x}}
		\\
		\leq&\,
		\frac{\nnorm{\delta^k}\nnorm{\lambda^k}}{\nnorm{x^k-\bar x}} 
			+ \nnorm{\delta^k}
		\\
			&+ \frac{\nnorm{F(x^k)-\delta^k-y^{k,\iota(k)}}\nnorm{\lambda^k}}{\nnorm{x^k-\bar x}} 
			+ \nnorm{F(x^k)-\delta^k-y^{k,\iota(k)}},
		\qquad
		\forall k\in\N,
	\end{align*}
	and apply \eqref{eq:multipliers_not_arbitrarily_unbounded} for 	
	$\{ (x^k,\lambda^k,\delta^k,\varepsilon^k) \}_{k=1}^\infty$, $\delta^k \to 0$, and
	\eqref{eq:diag_sequence_quality1_new1} to find that the first, second, and 
	fourth summands on the right-hand side are bounded.
	Further, we have
	\[
		\frac{\nnorm{F(x^k)-\delta^k-y^{k,\iota(k)}}\nnorm{\lambda^k}}{\nnorm{x^k-\bar x}} 
		\leq
		\nnorm{\lambda^k} \nnorm{\delta^k}
		= 
		\frac{\nnorm{\delta^k}\nnorm{\lambda^k}}{\nnorm{x^k-\bar x}} \nnorm{x^k-\bar x},
		\qquad
		\forall k\in\N
	\]
	by \eqref{eq:diag_sequence_quality1_new1}, such that, 
	due to $x^k \to \bar x$ and \eqref{eq:multipliers_not_arbitrarily_unbounded} 
	for 	$\{ (x^k,\lambda^k,\delta^k,\varepsilon^k) \}_{k=1}^\infty$, also the remaining summand 
	is bounded. Altogether, \eqref{eq:approx_stat_dir} is fulfilled for sequence 
	$\{ (x^k, \tilde\lambda^k, \tilde\delta^k, \tilde\varepsilon^k) \}_{k=1}^\infty$.
	
	In Case 2.2, we can proceed in similar fashion. 
	Replacing \eqref{eq:diag_sequence_quality1} by
	\eqref{eq:diag_sequence_quality1_new}, we define
	\[
		\eta_k
		\coloneqq
		2^{-k}\min\left\{
			\min\limits_{j\in (J_1\setminus J_2)\cup J_3}\bigl(b^j_{i_0}-y^{k,\iota(k)}_{i_0}\bigr),
			\frac{\nnorm{\delta^k}}{2}
		\right\},
		\qquad
		\forall k\in\N,
	\]
	and note $\eta_k\downarrow 0$ as $2^{-k-1} \nnorm{\delta^k} \to 0$.	
	Afterwards, we use the constructions from \eqref{eq:tilde_seq_case_2.2} to proceed.
	Analogous computations as above reveal $\tilde\delta^k\neq 0$ and $\tilde\lambda^k \neq 0$
	for all $k\in\N$
	and that $\{ (x^k, \tilde\lambda^k, \tilde\delta^k, \tilde\varepsilon^k) \}_{k=1}^\infty$
	is an AM-stationary sequence w.r.t.\ $\bar x$ in direction $d$.
	Furthermore, the above construction preserves the decisive properties in \eqref{lambda_neq_0_proof}.		
	
	\textbf{Case 2.1:} 
	In this case, we define $\{\tilde\lambda^k\}_{k=1}^\infty$ according to
	\[
		\tilde\lambda^k_i
		\coloneqq
		\begin{cases}
			\lambda^k_i	&	i\in\{1,\ldots,\ell\}\setminus I_{12}
			\\
			\frac1k\ell^{-1/2}\min\{1,\nnorm{\lambda^k}\}	&	i\in I_{12}\colon\,L_i\in\{\R_+,\R\}
			\\
			-\frac1k\ell^{-1/2}\min\{1,\nnorm{\lambda^k}\}	&	i\in I_{12}\colon\,L_i=\R_-
		\end{cases}
	\]	
	for all $k\in\N$.
	Note that the above construction preserves the property $\tilde\lambda^k_i\to 0$ for all $i\in I_{12}$,
	so that $\nnorm{\tilde\varepsilon^k}\to 0$ still holds.
	We need to show that the sequence 
	$\{ (x^k, \tilde\lambda^k, \delta^k, \tilde\varepsilon^k) \}_{k=1}^\infty\subset\R^{n+\ell+\ell+n}$
	is AM-stationary w.r.t.\ $\bar x$ in direction $d$,
	which means that it remains to show \eqref{eq:approx_stat_dir}.
	Clearly \eqref{eq:sequences_dir_conv} trivially follows from 
	$\{ (x^k,\lambda^k,\delta^k,\varepsilon^k) \}_{k=1}^\infty$ being
	AM-stationary w.r.t.\ $\bar x$ in direction $d$. 
	Concerning \eqref{eq:perturbation_vs_multiplier}, we first note that
	we have $\tilde\lambda^k\neq 0$ for all $k\in\N$
	due to $I_{12}\neq\emptyset$ and $\min\{1,\nnorm{\lambda^k}\} > 0$. 
	Hence, together with $\delta^k \neq 0$ for all $k\in\N$, proving \eqref{eq:perturbation_vs_multiplier}
	is equivalent to showing \eqref{eq:residual_va_multiplier_nontrivial}.
	For each $k\in\N$, we find
	\begin{equation*}
		\nnorm{\tilde\lambda^k-\lambda^k}
		\leq
		\frac1k\nnorm{\lambda^k},
	\end{equation*}
	which implies
	\[
		\norm{\frac{\lambda^k}{\nnorm{\lambda^k}}-\frac{\tilde\lambda^k}{\nnorm{\tilde\lambda^k}}}
		\leq
		\frac{2\nnorm{\tilde\lambda^k-\lambda^k}}{\nnorm{\lambda^k}}
		\leq
		\frac{2}{k}\to 0.
	\]	
	This yields
	\[
		\norm{\frac{\delta^k}{\nnorm{\delta^k}} - \frac{\tilde\lambda^k}{\nnorm{\tilde\lambda^k}}}
		\leq
		\norm{\frac{\delta^k}{\nnorm{\delta^k}} - \frac{\lambda^k}{\nnorm{\lambda^k}}}
		+
		\norm{\frac{\lambda^k}{\nnorm{\lambda^k}}-\frac{\tilde\lambda^k}{\nnorm{\tilde\lambda^k}}}
		\to 0,
	\]
	where we also used \eqref{eq:perturbation_vs_multiplier} 
	for $\{ (x^k,\lambda^k,\delta^k,\varepsilon^k) \}_{k=1}^\infty$ 
	as $\delta^k\neq 0$ and $\lambda^k\neq 0$ for all $k\in\N$.
	Hence, \eqref{eq:perturbation_vs_multiplier} holds for
	$\{ (x^k, \tilde\lambda^k, \delta^k, \tilde\varepsilon^k) \}_{k=1}^\infty$ as well.
	For \eqref{eq:multipliers_not_arbitrarily_unbounded}, we find that
	the definition of $\tilde\lambda^k$ implies	
	\begin{equation*}
		\nnorm{\tilde\lambda^k} 
		\leq 
		\left(1+\frac1k\right)\nnorm{\lambda^k},
		\qquad
		\forall k\in\N.
	\end{equation*}
	The latter yields
	\[
		\frac{\nnorm{\delta^k}\nnorm{\tilde\lambda^k}}{\nnorm{x^k-\bar x}}
		\leq
		\left(1+\frac1k\right)\frac{\nnorm{\delta^k}\nnorm{\lambda^k}}{\nnorm{x^k-\bar x}},
		\qquad
		\forall k\in\N,
	\]
	which implies that boundedness of the left-hand side is confirmed due to 
	\eqref{eq:multipliers_not_arbitrarily_unbounded} 
	for $\{ (x^k,\lambda^k,\delta^k,\varepsilon^k) \}_{k=1}^\infty$.
\end{proof}

\begin{proof}[Proof of \cref{prop:subMFC_vs_NNAMCQ}]
	To start, let us prove assertion~\ref{item:NNAMCQ_ODPsubMFC}.
	Since $\bar{x}$ is a local minimizer, \cref{lem:approx_MSt_NC,lem:lambda_neq_0} yield
	the existence of a sequence 
	$\{ (x^k, \lambda^k, \delta^k, \varepsilon^k) \}_{k=1}^\infty \subset \R^{n+\ell+\ell+n}$
	that is AM-stationary w.r.t.\ $\bar{x}$ and fulfills 
	$\lambda^k_i \neq 0$ for all $i\in I(x^k,\delta^k)$ and $k\in\N$.
	By definition, this particularly implies
	$\lambda^k \in N_\Gamma(F(x^k)-\delta^k)$ for all $k\in\N$, 
	and, recalling \cref{rem:constant_sets},
	we may take $I\coloneqq I(x^k,\delta^k)$ for any $k\in\N$ to fulfill \ref{subMFCII}
	of ODP-subMFC at $\bar{x}$. If $I=\emptyset$, then \ref{subMFCI}
	of ODP-subMFC at $\bar{x}$ is satisfied as well. Otherwise, we use 
	NNAMCQ at $\bar x$, which implies that
	\[
		0\in \mathsmaller\sum\nolimits_{i=1}^\ell |\lambda_i|\partial(\sgn(\lambda_i)\,F_i)(\bar x),
		\lambda \in N_\Gamma (F(\bar x)) \quad\implies\quad \lambda = 0
	\]
	holds true, where we applied \cref{ass:subMFC}.
	Let us show that \ref{subMFCI} of ODP-subMFC at $\bar{x}$ is satisfied. 
	To this end, pick some
	$u\in\R^\ell$ with $u\geq 0$ and $u_{\{1,\dotsc,\ell\} \setminus I}=0$
	such that
	\begin{equation}\label{eq:NNAMCQ_ODPsubMFC_1}
		0 \displaystyle 
		\in 
		\mathsmaller\sum\nolimits_{i\in I} u_i \partial\bigl(\sgn(\lambda^k_i)\, F_i\bigr)(\bar x).
	\end{equation}
	We are going to prove that this necessarily implies $u=0$.
	Recall \cref{rem:constant_sets} and define the vector $v\in\R^\ell$ by 
	$v_i \coloneqq \sgn(\lambda_i^k) u_i$ for all 
	$i\in\{1,\dotsc,\ell\}$ using any $k\in\N$, as well as the set $I_\pm \coloneqq \{ i\in I \mid v_i\neq 0 \}$.
	Due to $\lambda_i^k\neq 0$ for all $i\in I$ and $k\in\N$, this particularly yields 
	$I_\pm = \{ i\in I \mid u_i\neq 0 \}$ as well as
	$\sgn(v_i) = \sgn(\lambda_i^k)$ and $|v_i| = u_i$ for all $i\in I_\pm$.
	By definition of $I_\pm$, we have $v_i=0$ for $i\in I\setminus I_\pm$.
	Additionally, $v_i=0$ also holds for	all $i\notin I$ due to $u_i=0$ for all $i\notin I$.
	Altogether, we obtain $v_i= 0$ for all $i\notin I_\pm$.
	Using the above identities, \eqref{eq:NNAMCQ_ODPsubMFC_1} is equivalent to
	\begin{equation}\label{eq:NNAMCQ_ODPsubMFC_3}
		\begin{aligned}
		0 \displaystyle 
		\in 
		\mathsmaller\sum\nolimits_{i\in I_\pm} u_i \partial\bigl(\sgn(\lambda^k_i)\, F_i\bigr)(\bar x) 
		&=
		\mathsmaller\sum\nolimits_{i\in I_\pm} |v_i| \partial\bigl(\sgn(v_i)\, F_i\bigr)(\bar x)
		\\
		&=
		\mathsmaller\sum\nolimits_{i=1}^\ell |v_i| \partial\bigl(\sgn(v_i)\, F_i\bigr)(\bar x).	
		\end{aligned}
	\end{equation}
	Due to the orthodisjunctive structure of $\Gamma$, it follows from 
	\cite[Corollary 1]{AdamCervinkaPistek2016} that, for all $k\in\N$,
	there exists a set $S^k\subset \Gamma$ such that
	\[
		N_\Gamma(F(x^k)-\delta^k) = \bigcup_{y\in S^k} \widehat{N}_\Gamma(y).
	\]
	The orthodisjunctive structure in \eqref{eq:orthodisjunctive_problem} ensures that,
 	for all $y \in S^k$, $k\in \N$, we have the representation
	\begin{equation*}
		\widehat{N}_{\Gamma}(y) 
		= 
		\prod_{i=1}^\ell L_i^y
	\end{equation*}
	using some $L_i^y \in \{\{ 0 \}, \R_+, \R_-, \R\}$, 
	$i\in\{1,\dotsc,\ell\}$,
	see, e.g., \cite[Remark~4.1\,(a)]{KaemingMehlitz2025} or the proof of \cref{lem:lambda_neq_0} above.
	Thus, for all $k\in\N$,
	\[
		\lambda^k \in N_\Gamma(F(x^k)-\delta^k) = \bigcup_{y\in S^k} \prod_{i=1}^\ell L_i^y
	\]
	implies $v\in N_\Gamma(F(x^k)-\delta^k)$ due to $\sgn(v_i)=\sgn(\lambda_i^k)$ 
	for all $i\in I_\pm$ and 
	$v_i=0$ for all $i\notin I_\pm$.
	By the robustness of the limiting normal cone, this yields $v\in {N}_\Gamma(F(\bar{x}))$.
	Together with \eqref{eq:NNAMCQ_ODPsubMFC_3}, 
	NNAMCQ at $\bar x$ implies $v=0$.
	Due to $|v_i|=u_i$ for all $i\in I_\pm$, $u_i = 0$ for $i\in I\setminus I_\pm$ by
	definition of $I_\pm$, and $u_i=0$ for all $i\in \{1,\ldots,\ell\}\setminus I$,
	we obtain $u=0$,
	which proves that~\ref{subMFCI} of ODP-subMFC holds at $\bar x$.
	Hence, ODP-subMFC is valid at $\bar x$.
	
	Let us proceed with the proof of assertion~\ref{item:FOSCMS_dirODPsubMFC}.
	Due to $d\in \widehat C(\bar x) \cap \mathbb{S}$ and $\nabla f(\bar x)\neq 0$,
	we can use \cref{lem:lambda_neq_0} to guarantee the 
	existence of an AM-stationary sequence 
	$\{ (x^k,\lambda^k,\delta^k,\varepsilon^k) \}_{k=1}^\infty \subset\R^{n+\ell+\ell+n}$
	w.r.t.\ $\bar x$ in direction $d$ with $\lambda^k_i\neq 0$ for all $i\in I(x^k,\delta^k)$ and $k\in\N$. 
	We may now follow the proof of assertion~\ref{item:NNAMCQ_ODPsubMFC} subject to the following
	exceptions. If $I\coloneqq I(x^k,\delta^k)$ for any $k\in\N$ fulfills 
	$I\neq\emptyset$, we now need to use FOSCMS$(d)$ at $\bar x$, which reads as 
	\[
		0 \in \mathsmaller\sum\nolimits_{i=1}^\ell|\lambda_i|\partial(\sgn(\lambda_i)\,F_i)(\bar x;d),
		\lambda \in N_\Gamma(F(\bar x); F'(\bar{x};d))\quad \implies\quad \lambda=0
	\]
	after applying \cref{ass:subMFC}. With FOSCMS$(d)$ we then need to show that
	\[
		0 \displaystyle 
		\in 
		\mathsmaller\sum\nolimits_{i=1}^{\ell} |v_i| \partial\bigl(\sgn(v_i)\, F_i\bigr)(\bar x;d)
		\quad \implies\quad v=0
	\]
	holds for all $v\in\R^\ell$ that can, for any $u\in\R^\ell$ 
	with $u \geq 0$ and $u_{\{1,\dotsc,\ell\}\setminus I} = 0$, be defined via 
	$v_i\coloneqq \sgn(\lambda_i^k)u_i$ for all $i\in\{1,\dotsc,\ell\}$ using some $k\in \N$. 
	To this end, analogously to the proof of assertion~\ref{item:NNAMCQ_ODPsubMFC},
	we realize that $\lambda^k\in N_\Gamma(F(x^k)-\delta^k)$ implies 
	$v\in N_\Gamma(F(x^k)-\delta^k)$ for all $k\in\N$.
	Using \eqref{eq:sequences_conv}, \eqref{eq:sequences_dir_conv}, \eqref{eq:suitable_repr_of_constraint_viol}, and
	the robustness of the directional limiting normal cone then
	yields $v \in N_\Gamma(F(\bar{x});F'(\bar{x};d))$. 
	Thus, FOSCMS$(d)$ at $\bar{x}$ implies ODP-subMFC$(d)$ at $\bar{x}$. 
\end{proof}

\end{document}